\documentclass[reqno,oneside,12pt]{amsart}

\usepackage[utf8]{inputenc}
\usepackage{esint}

\usepackage[highlightmarkup=background]{changes}
\definechangesauthor[name={Silvia Bertoluzza}, color=red]{sb}
\definechangesauthor[name={Marco}, color=blue]{gm}

\usepackage{amsmath}       
\usepackage{amssymb}       

\usepackage{chapterbib}    
\usepackage{color}
\usepackage{bm,bbm}
\usepackage{overpic}
\usepackage{times}
\usepackage{epsfig}
\usepackage{enumitem}
\usepackage{graphicx,graphics,rotating}
\usepackage{subfigure}
\usepackage{multicol}
\usepackage{multirow}
\usepackage{tikz}
\usepackage{pgfplots}
\usepackage{pgfplotstable}
\pgfplotsset{compat=newest}
\usepackage{siunitx}
\usepackage{booktabs}
\usepackage{enumitem}
\usepackage{stmaryrd}
\usepackage[sort,compress]{cite}

\theoremstyle{plain}
\newtheorem{theorem}{Theorem}[section]
\newtheorem{lemma}[theorem]{Lemma}
\newtheorem{corollary}[theorem]{Corollary}
\newtheorem{proposition}[theorem]{Proposition}

\theoremstyle{definition}

\newtheorem{remark}[theorem]{Remark}

\theoremstyle{plain}

\definecolor{MyDarkGreen}{rgb}{0,0.45,0}

\def\trait #1 #2 #3 {\vrule width #1pt height #2pt depth #3pt}
\def\fin{\hfill
        \trait .3 5 0
        \trait 5 .3 0
        \kern-5pt
        \trait 5 5 -4.7
        \trait 0.3 5 0
\medskip}
\newcommand{\ENDPROOF}{\fin}

\newcommand{\dims}{\,\textit{dim}}
\newcommand{\RANK}{\,\textit{rank}}

\newcommand{\REAL}{\mathbbm{R}}

\newcommand{\nv}{\mathbf{n}}

\newcommand{\uv}{\mathbf{u}}
\newcommand{\vv}{\mathbf{v}}
\newcommand{\wv}{\mathbf{w}}
\newcommand{\xv}{\mathbf{x}}
\newcommand{\yv}{\mathbf{y}}

\newcommand{\zerov}{\mathbf{0}}

\newcommand{\as}{a}

\newcommand{\cs}{c}

\newcommand{\es}{e}
\newcommand{\fs}{f}
\newcommand{\gs}{g}

\newcommand{\ms}{m}
\newcommand{\ns}{n}

\newcommand{\ps}{p}
\newcommand{\qs}{q}
\newcommand{\rs}{r}
\renewcommand{\ss}{s}

\newcommand{\us}{u}
\newcommand{\vs}{v}
\newcommand{\ws}{w}
\newcommand{\xs}{x}
\newcommand{\ys}{y}

\newcommand{\Cs}{C}

\newcommand{\Fs}{F}

\newcommand{\Ms}{M}
\newcommand{\Ns}{N}

\newcommand{\Qs}{Q}

\newcommand{\Vs}{V}
\newcommand{\Ws}{W}

\newcommand{\qst}{\widetilde{\qs}}

\newcommand{\matD}{\mathsf{D}}

\newcommand{\matG}{\mathsf{G}}
\newcommand{\matH}{\mathsf{H}}
\newcommand{\matI}{\mathsf{I}}

\newcommand{\matM}{\mathsf{M}}

\newcommand{\matP}{\mathsf{P}}

\newcommand{\matR}{\mathsf{R}}
\newcommand{\matS}{\mathsf{S}}

\newcommand{\invS}{\mathsf{S}^\dagger}

\newcommand{\invinvS}{\matS^{\dagger\dagger}}
\newcommand{\invinvtSigma}{\widetilde{\mathsf{\Sigma}}^{\dagger\dagger}}
\newcommand{\invtSigma}{\widetilde{\mathsf{\Sigma}}^{\dagger}}
\newcommand{\tSigma}{\widetilde{\mathsf{\Sigma}}}
\newcommand{\matrSigma}{\mathsf{\Sigma}}
\newcommand{\matPs}  {\matP^*}

\newcommand{\calH}{\mathcal{H}}

\newcommand{\calM}{\mathcal{M}}
\newcommand{\calN}{\mathcal{N}}
\newcommand{\calO}{\mathcal{O}}

\newcommand{\deltav}  {\bm\delta}     
\newcommand{\zetav}   {\bm\zeta}      
\newcommand{\etav}    {\bm\eta}       
\newcommand{\kappav}  {\bm\kappa}      
\newcommand{\lambdav} {\bm\lambda}    

\newcommand{\PS}[1]{\mathbbm{P}_{#1}}

\newcommand{\HONE}   {H^1}
\newcommand{\HONEzr} {H^1_0}
\newcommand{\HONEzrbr}{\HONE_{\small{\oslash}}(\P)}
\newcommand{\Htracezr}{\HS{\frac12}_{\small{\oslash}}(\bP)}

\newcommand{\HONEnck}{H^{1,nc}_{k}}

\newcommand{\LTWO}  {L^2}
\newcommand{\LINF}  {L^{\infty}}

\newcommand{\HS}[1] {H^{#1}}

\renewcommand{\P} {P}
\newcommand  {\E} {e}

\newcommand{\Ep}  {\E^{\prime}}

\newcommand{\Ei}  {\E_{i}}
\newcommand{\Ej}  {\E_{j}}

\newcommand{\hh}{h}
\newcommand{\Th}{\Omega_{\hh}}

\newcommand{\hP}{\hh_{\P}}

\newcommand{\hE}{\hh_{\E}}

\newcommand{\hhP}{\widehat{\hh}_{\P}}

\newcommand{\hEp}{\hh_{\Ep}}

\newcommand{\mP}{\ABS{\P}}

\newcommand{\Eset}{\mathcal{E}}    

\newcommand{\NMB}{N}

\newcommand{\NSs}{\NMB^*}      

\newcommand{\dV}{}
\newcommand{\dS}{}

\newcommand{\nor}  {\mathbf{n}}
\newcommand{\norE} {\nor_{\E}}

\newcommand{\norP} {\mathbf{n}_{\P}}

\newcommand{\fsh}{\fs_{\hh}}

\newcommand{\ush}{\us}

\newcommand{\vsh}{\vs}
\newcommand{\vshh}{\widehat{\vs}}

\newcommand{\vsE}{\vs_{\E}}

\newcommand{\wsh}{\ws_{\hh}}

\newcommand{\asP}{\as^{\P}}

\newcommand{\ash}{\as_{\hh}}

\newcommand{\ashP}{\as^P_{\hh}}

\newcommand{\scalV} [2]{\left(#1,#2\right)_{}}
\newcommand{\scalVp}[2]{\left(#1,#2\right)_{*}}

\newcommand{\snorm}  [2]{|#1|_{#2}}

\newcommand{\SNORM}  [2]{\left|#1\right|_{#2}}

\newcommand{\tb}{\interleave}
\newcommand{\norm}   [2]{\|#1\|_{#2}}

\newcommand{\NORM}   [2]{\left\|#1\right\|_{#2}}
\newcommand{\tNORM}   [2]{ \tb #1 \tb_{#2}}

\newcommand{\snormV} [1]{|#1|_{}}
\newcommand{\snormVp}[1]{|#1|_{*}}
\newcommand{\normV}  [1]{\|#1 \|_{}}
\newcommand{\normVp} [1]{\| #1 \|_{*}}

\newcommand{\ABS}    [1]{\left|#1\right|}

\newcommand{\jump}[1]{\lbrack\!\lbrack\,#1\,\rbrack\!\rbrack}

\newcommand{\Vhk}{V^{\hh}_{k}}
\newcommand{\VhkP}{V^{\hh}_k(\P)}

\newcommand{\Piz}[1]{\Pi^{0}_{#1}}

\newcommand{\PinP}[1]{\Pi^{\nabla,\P}_{#1}}

\newcommand{\bscal}[1]{\langle #1 \rangle}

\newcommand{\restrict}[2]{{#1}_{|{#2}}}

\newcommand{\EOD}{\end{document}}

\newcommand{\bil}[2]{\langle#1,#2\rangle}

\newcommand{\roundPrecision}{2}
\newcommand{\bP} {\partial\P}            
\newcommand{\mbP}{\ABS{\partial\P}}      

\newcommand{\Nk}[1]{N_{#1}}

\newcommand{\Vhken}{\Vs_{k}^{\hh,\text{en}}}

\newcommand{\VhkenP}{\Vs_{k}^{\hh,\text{en}}(\P)}

\newcommand{\basis}{\mathfrak{B}}
\newcommand{\hbasis}{\widehat{\basis}}

\newcommand{\te}{\eta}
\newcommand{\ebase}{e}
\newcommand{\hebase}{\widehat{\ebase}}
\newcommand{\hte}{\widehat{\te}}

\newcommand{\tW}{W^*}

\newcommand{\matr}[1]{\mathbf{#1}}

\newcommand{\tsP}{\sigma^*_\P}
\newcommand{\basedofs}{\mathfrak B}
\newcommand{\Nbdofs}{k N}

\newcommand{\G}{\mathcal{G}}
\newcommand{\hG}{\mathcal{G}^*}

\newcommand{\gridauxP} {\mathcal{G}_\text{aux}(\bP)}
\newcommand{\hgridauxP}{\mathcal{G}^*_\text{aux}(\bP)}

\newcommand{\Const}{K(\G)}
\newcommand{\Lin}{\widetilde K(\hG)}

\DeclareMathOperator{\spa}{span}

\newcommand{\Qtilde}{\widetilde Q}
\newcommand{\Gl}{\widetilde{\matG}_\ell^*}

\newcommand{\SRone}  {\textbf{(G1)}}
\newcommand{\SRtwo}  {\textbf{(G2)}}
\newcommand{\SRtwop} {\textbf{(G2a)}}
\newcommand{\SRtwopp}{\textbf{(G2b)}}
\newcommand{\SRthree}{\textbf{(G3)}}
\newcommand{\SRthreea}{\textbf{(G3.1)}}
\newcommand{\SRthreeb}{\textbf{(G3.2)}}

\newcommand{\gb}{\gamma_2}

\newcommand{\Edges}   {\mathcal{E}}          
\newcommand{\EdgesP}  {\mathcal{E}_\P}
\newcommand{\EdgesPp} {\mathcal{E}_\P^{1}}
\newcommand{\EdgesPpp}{\mathcal{E}_\P^{2}}
\newcommand{\EdgesG}{\mathcal{E}_\Gamma}

\newcommand{\tEe}{\omega_e}

\newcommand{\constSteinbach}{c_0}
\newcommand{\tNk}[1]{\widetilde N_{#1}}

\newcommand{\sP} {\sigma^{\P}}

\newcommand{\sPlin}{\widetilde{\sigma}^{\P}}
\newcommand{\Baselin}{\widetilde{\mathfrak{B}}}
\newcommand{\tphi}{\widetilde{\phi}}

\newcommand{\ringVhkP}{\mathring{V}^h_k(\P)}

\newcommand{\portedge}{\pi^0_{\bP}}

\DeclareMathOperator{\KER}{ker}

\newcommand{\Vsp}{\Vs^{\prime}} 
\newcommand{\Wss}{\Ws^*}

\newcommand{\tVhk}{V_k^*}

\newcommand{\gK}    {\gamma_\P}       
\newcommand{\gKstar}{\gamma_\P^*}

\newcommand{\spcdot}{\hspace{0.2mm}\cdot\hspace{0.2mm}}

\newcommand{\TERM}[1]{\textbf{(#1)}}

\newcommand{\constinfsup}{\beta}

\newcommand{\Wbar}{\widehat\Ws}
\newcommand{\Pbar}{\widehat\Pi}

\newcommand{\Wbars}{\widehat\Ws^*}
\newcommand{\Pbars}{\widehat\Pi^*}

\newcommand{\lbdv}{\lambdav}
\newcommand{\sss}{\ss^*}

\newcommand{\monomials}[2]{\calM_{#1}(#2)}
\newcommand{\mono}{m_\alpha}

\newcommand{\MeshThree}{$\calM_3$}

\newcommand{\MeshOne}  {$\calM_1$}
\newcommand{\MeshOneA} {$\calM_{1A}$}
\newcommand{\MeshOneB} {$\calM_{1B}$}

\newcommand{\MeshTwo}  {$\calM_2$}
\newcommand{\MeshTwoB} {$\calM_{2B}$}
\newcommand{\MeshTwoA} {$\calM_{2A}$}

\newcommand{\Error}[1]{\es_{#1}^{\us}}

\newcommand{\Stab}[1]{$\sigma_{#1}$}

\newcommand{\Nel}{N_\textup{el}}
\newcommand{\Ned}{N_\textup{ed}}
\newcommand{\gamh}{\gamma_{\hh}}

\newcommand{\CsM}{A}
\newcommand{\CPihat}{C_{\Pbar}}
\newcommand{\Cpoinc}{C_{\text{poi}}}

\newcommand{\tNaux}{N^*_\text{aux}}
\newcommand{\Naux}{N_\text{aux}}

\newcommand{\lowpass}{\mathsf{h}}
\newcommand{\bandpass}{\mathsf{g}}
\newcommand{\Pj}{Q_j}

\newcommand{\lb}{b}

\newcommand{\Ext}{E}

\newcommand{\Dlow}{D_1}
\newcommand{\Dhigh}{D_2}

\newcommand{\solh}{\us_\hh}

\newcommand{\ssltwo}{s^0_{L2}}
\newcommand{\sswav}{s^0_\mathrm{wav}}

\newcommand{\sslb}{s^0_{\mathrm{sLB}}}
\newcommand{\ssrlb}{s^0_{\mathrm{rLB}}}

\newcommand{\tmatM}{\widetilde{\matM}}
\newcommand{\tmatR}{\widetilde{\matR}}

\begin{document}

\title{Stabilization of the nonconforming Virtual Element Method}

\author[IMATI]{S.~Bertoluzza$^*$}
\author[IMATI]{G.~Manzini}
\author[IMATI]{M.~Pennacchio}
\author[IMATI]{D.~Prada}


\address[IMATI]{
  Istituto di Matematica Applicata e Tecnologie Informatiche, "E. Magenes", CNR, 
  via Ferrata 5A, 27100 Pavia, Italy
}

\thanks{$^*$ Corresponding author}


\begin{abstract}
  We address the issue of designing robust stabilization terms for the
  nonconforming virtual element method.
  To this end, we transfer the problem of defining the stabilizing
  bilinear form from the elemental nonconforming virtual element
  space, whose functions are not known in closed form, to the dual
  space spanned by the known functionals providing the degrees of
  freedom.
  By this approach, we manage to construct different bilinear forms
  yielding optimal or quasi-optimal stability bounds and error
  estimates, under weaker assumptions on the tessellation than the
  ones usually considered in this framework.
  In particular, we prove optimality under geometrical assumptions
  allowing a mesh to have a very large number of arbitrarily small
  edges per element.
  Finally, we numerically assess the performance of the VEM for
  several different stabilizations fitting with our new framework on a
  set of representative test cases.
\end{abstract}


\maketitle


  

\section{Introduction}
Solving partial differential equations on polygonal and polyhedral
meshes has become a major issue in the last decades,
and a number of numerical methods have been proposed to this end in
the technical literature.
Many of these methods are based on some kind of generalization of the
finite element method (FEM) and must address the critical issue that the
construction of shape functions on elements with arbitrary geometric
shapes is a very difficult task.
The virtual element method (VEM), originally proposed
in~\cite{BeiraodaVeiga-Brezzi-Cangiani-Manzini-Marini-Russo:2013} for
the Poisson equation and then extended to
convection-reaction-diffusion problems with variable coefficients
in~\cite{BeiraodaVeiga-Brezzi-Marini-Russo:2016b}, brilliantly overcomes
this issue.
The method was designed from the very beginning to work on generally
shaped elements with high order of accuracy, and does not require an
explicit knowledge of the basis functions that generate the finite
element approximation space.
Indeed, the formulation of the method and its practical
implementations are based on suitable polynomial projections that are
always computable from a careful choice of the degrees of freedom.
Optimal numerical approximations of arbitrary order and arbitrary
regularity to PDE solutions are possible in two and three dimensions
using very general mesh families, including meshes that are often
considered as pathological in other methods.
VEM is intimately connected with other finite
element approaches:
the connection between the VEM and finite elements on
polygonal/polyhedral meshes is thoroughly investigated
in~\cite{Manzini-Russo-Sukumar:2014,
  Cangiani-Manzini-Russo-Sukumar:2015, DiPietro-Droniou-Manzini:2018},
between VEM and discontinuous skeletal gradient discretizations
in~\cite{DiPietro-Droniou-Manzini:2018}, and between the VEM and the
BEM-based FEM method
in~\cite{Cangiani-Gyrya-Manzini-Sutton:2017:GBC:chbook}.

The conforming VEM was originally developed as a variational
reformulation of the \emph{nodal} mimetic finite difference (MFD)
method~\cite{%
Brezzi-Buffa-Lipnikov:2009,%
BeiraodaVeiga-Lipnikov-Manzini:2011,%
Manzini-Lipnikov-Moulton-Shashkov:2017%
}
for solving diffusion problems on unstructured polygonal meshes. The issue of its efficient implementation is considered in several papers (cf. \cite{Bertoluzza-Pennacchio-Prada:2020,
	Bertoluzza-Pennacchio-Prada:2017,
	Antonietti-Mascotto-Verani:2018,
	Dassi-Scacchi:2019,
	Dassi-Scacchi:2020,
	Calvo:2019}.
A survey on the MFD method can be found in the review
paper~\cite{Lipnikov-Manzini-Shashkov:2014} and the research
monograph~\cite{BeiraodaVeiga-Lipnikov-Manzini:2014}.
The scheme inherits the flexibility of the MFD method with respect to
the admissible meshes and this feature is well reflected in the many
significant applications that have been developed so far, see, for
example,~\cite{%
BeiraodaVeiga-Manzini:2014,%
BeiraodaVeiga-Manzini:2015, Berrone-Pieraccini-Scialo-Vicini:2015,%
Mora-Rivera-Rodriguez:2015,%
Paulino-Gain:2015,%
Antonietti-BeiraodaVeiga-Scacchi-Verani:2016,%
BeiraodaVeiga-Chernov-Mascotto-Russo:2016,%
BeiraodaVeiga-Brezzi-Marini-Russo:2016b,%
Cangiani-Georgoulis-Pryer-Sutton:2016,%
Perugia-Pietra-Russo:2016,%
Wriggers-Rust-Reddy:2016,
Certik-Gardini-Manzini-Vacca:2018:ApplMath:journal,
Dassi-Mascotto:2018,%
Benvenuti-Chiozzi-Manzini-Sukumar:2019:CMAME:journal,%
Antonietti-Manzini-Verani:2019:CAMWA:journal,
Antonietti-Bertoluzza-Prada-Verani:2020,
Certik-Gardini-Manzini-Mascotto-Vacca:2020%
}.

The nonconforming virtual element method was originally proposed
in~\cite{AyusodeDios-Lipnikov-Manzini:2016} for the solution of the
Poisson equation.
Then, it was extended to
general elliptic
equations~\cite{Cangiani-Manzini-Sutton:2017,Berrone-Borio-Manzini:2018},
fractional reaction-subdiffusion
equations~\cite{Li-Zhao-Huang-Chen:2019},
eigenvalue problems~\cite{Gardini-Manzini-Vacca:2019},
Helmholtz
equations~\cite{Mascotto:Perugia-Pichler:2019-M3AS,Mascotto-Perugia-Pichler:2019-CMAME,Mascotto-Pichler:2019},
Stokes, Darcy-Stokes and Navier-Stokes
equations~\cite{Cangiani-Gyrya-Manzini:2016,Zhao-Zhang-Mao-Chen:2019,Zhao-Zhang-Mao-Chen:2020},
elasticity problems~\cite{Zhang-Zhao-Yang-Chen:2019},
nonconforming anisotropic estimates~\cite{Cao-Long:2019},
plate bending problems, biharmonic equation, highly-order elliptic
equations~\cite{Zhao-Chen-Zhang:2016,Antonietti-Manzini-Verani:2018,Zhang-Zhao-Chen:2020,Huang-Yu:2021}.

\medskip
The nonconforming virtual element method possesses several interesting
features.
First, the VEM admits meshes whose elements are polygons (2D) and
polyhedra (3D) with, in principle, almost arbitrary geometric shapes.
This flexibility in the mesh choice may have a significant
impact in both numerical approximation and mesh generation.
Second, we can construct stable virtual element methods in a
straightforward way for any polynomial degree.
Moreover, such construction can be readily generalized from two to
three space dimensions, and, in principle, to any space dimensions.
Third, the formulation and implementation of the nonconforming VEM
needs less degrees of freedom than other methods, such as, for example
discontinuous Galerkin.
Note also that unknowns associated with the interior of the mesh
elements can be eliminated by static condensation.
This feature makes the VEM competitive in terms of computational
efficiency with respect to other discretization methods.

\medskip
As it happens in the conforming VEM, the stability and convergence of the
nonconforming VEM rely on the fundamental properties of
\emph{consistency} and \emph{stability}.
Consistency is an exactness property that states that the approximated
bilinear forms of the discrete variational formulation are exact on
the subspace of polynomials locally defined in each element.
In turn, stability follows from a suitable \emph{stabilization term},
whose role is to control the non polynomial component of the
discretization. When the polygonal elements satisfy  a  quite restrictive shape regularity condition, basically equivalent to requiring that they  can be decomposed into a (small) number of shape regular triangles, we know that the euclidean
product of the degrees of freedom of the virtual element functions is
an effective stabilization term and provides optimal results.
However, greater care must be taken in designing the stabilization
term when we consider more general elements, such as, for instance,
elements with very small edges.

In the conforming case, the design of computable stabilization terms
yielding optimal results relies on the fact that we can compute the
trace of the virtual element functions on the elemental boundaries
from the degrees of freedom.
Conversely, in the nonconforming case, the knowledge of the degrees of
freedom does not allow us to retrieve the trace of the corresponding
functions without solving a partial differential equation in the
element.
In the VEM terminology, we say that the trace on the elemental
boundary of a nonconforming virtual function is
``\emph{noncomputable}''.
On the positive side, the functionals yielding the degrees of freedom
in a polygonal element $\P$ span a known subspace $\tVhk(\P)$ of the
dual space $(H^1(\P))'$. Such a space satisfies a uniformly stable duality relation with the local VEM
space $\VhkP$. This property, which is inherent to the nonconforming nature of the
approximation space and does not hold for the conforming VEM, allows us to reduce the problem of designing the stabilization bilinear form on the non conforming VEM space $\VhkP$, to the design of a semi-inner product in $\tVhk(\P)$, yielding a suitable seminorm for $(H^1(\P))'$. 
We can then consider and analyze different
strategies for the construction of such semi-inner product, yielding
optimal or quasi-optimal stability and convergence results under
weaker assumptions on the polygonal tessellation.

\medskip
We conclude this introductory section with a review of some basic
definitions about the functional setting and the notation that we use
in the paper.
The rest of the paper is organized as follows.
In Section~\ref{sec:ncvem} we introduce the model problem and
its discretization by the nonconforming virtual element approximation.
In Section~\ref{sec:dual:scalar:product} we present an
abstract theoretical framework for the algebraic construction of the semi-inner
products in finite dimensional dual spaces.
In Section~\ref{sec:stab:theory} and~\ref{sec:lowstab} we discuss the
construction of the stabilization terms for the nonconforming VEM in
such a framework.
In Section~\ref{sec:numerical} we investigate the performance of the
method on a set of suitable numerical experiments.
In Section~\ref{sec:conclusions} we offer our final remarks and
conclusions.

\subsection{Basic definitions, notation and functional setting}

Let the computational domain $\Omega$ be an open, bounded, connected
subset of $\REAL^{2}$ with polygonal boundary $\Gamma$.
We consider a family of domain partitionings
$\mathcal{T}=\{\Th\}_{\hh\in\calH}$.
Every partition $\Th$, the \emph{mesh}, is a finite collection of non overlapping
polygonal elements $\P$, which are such that
$\overline{\Omega}=\cup_{\P\in\Th}\overline{\P}$.
Further assumptions on the mesh family $\mathcal{T}$ and the meshes
$\Th$ will be detailed in Section~\ref{sec:stab:theory}.
	
\medskip
For $\P \in \Th$, we denote the boundary of $\P$ by $\partial\P$, its diameter by $\hP=\max_{\xv,\yv\in\P}\ABS{\xv-\yv}$, its area by $\mP$, and the
outward unit normal to the boundary by $\nor_\P$.
Each elemental boundary $\partial\P$ is formed by a sequence of
one-dimensional non-intersecting straight edges $\E$ with
lenght $\hE$.
The symbols $\EdgesP$, $\EdgesG$ and $\Edges$ respectively denote the set of
edges that form the boundary of the element $\P$, the set of mesh edges on the boundary $\Gamma$, and the set of all the
mesh edges. 

\medskip
We use standard definitions and notations for Sobolev spaces, and for the corresponding norms and seminorms,  cf.~\cite{Adams-Fournier:2003}. More precisely, let $\omega$ be a $d$-dimensional domain, $d=1,2$.
We let $\LTWO(\omega)$ denote the Hilbert functional space of the
real-valued, square integrable functions defined on $\omega$, and
$\HS{m}(\omega)$ the Sobolev functional space of the real-valued
functions in $\LTWO(\omega)$ whose weak derivatives up to the order
$m$ are also in $\LTWO(\omega)$. We let $\norm{\cdot}{0,\omega}$ denote the standard norm in $\LTWO(\omega)$, and $\norm{\spcdot}{m,\omega}$ and $\snorm{\spcdot}{m,\omega}$ denote
respectively the standard norm and seminorm in $\HS{m}(\omega)$.
On the elemental boundary $\bP$, we also consider the functional space
\begin{align}
\HS{\frac12}(\bP) =
\Big\{
\vs\in\LTWO(\bP)
\,\,\textrm{such~that}\,\,
\NORM{\vs}{0,\bP} + \SNORM{\vs}{1/2,\bP}<\infty
\Big\},
\label{eq:HS-half:def}
\end{align}
and its dual $\HS{-\frac12}(\bP)$.
In~\eqref{eq:HS-half:def},
$\SNORM{\,\cdot\,}{1/2,\bP}$ is the
seminorm defined by
\begin{align}
\SNORM{ \vs }{1/2,\bP}^2 =
\int_{\bP\times\bP}\frac{\ABS{\vs(x)-\vs(y)}^2}{\ABS{x-y}^2}\,dx\,dy.
\label{eq:HS-half:seminorm:def}
\end{align}
We recall that the trace $\restrict{\vs}{\bP}$
of a function $\vs\in\HONE(\P)$
belongs to $\HS{\frac12}(\bP)$.
Similar definitions hold for  $\HS{\frac12}(\E)$,
$\HS{-\frac12}(\E)$, $\HS{\frac12}(\Gamma)$,
$\HS{-\frac12}(\Gamma)$, and for the corresponding norms and seminorms.

\medskip

For a given nonnegative integer $\ell$, we let $\PS{\ell}(\omega)$
denote the space of polynomials of degree up to $\ell$ defined on
$\omega$, and we conventionally define $\PS{-1}(\omega)=\{0\}$.
Furthermore, $\PS{\ell}(\Th)$ denotes the space of discontinuous
bivariate polynomials of degree up to $\ell$ defined on the elements
of $\Th$:
\begin{align*}
  \PS{\ell}(\Th) = \big\{ \ps\in\LTWO(\Omega):\,\restrict{\ps}\P\in\PS{\ell}(\P)\,\,\forall\P\in\Th \big\}.
\end{align*}
We let $\monomials{\ell}{\omega}$ denote the set of
scaled monomials on $\omega$ of degree up to $\ell$, given by
\begin{align*}
  \monomials{\ell}{\omega}
  = \left\{
  \ms_\alpha(\xv) = \left(\frac{\xv - \xv_\omega}{h_\omega}\right)^\alpha,
  \quad\alpha\in\mathbb{N}^d\  \text{ with } \ABS{\alpha}\leq\ell
  \right\},
\end{align*}
where $\xv_\omega$ denotes the center of mass of $\omega$ and
$h_\omega$ its diameter. The set $\monomials{\ell}{\omega}$ forms a basis for the space $\PS{\ell}(\omega)$.

\medskip

\medskip
On $\Th$ and for every integer $m>0$, we consider the broken Sobolev
space
\begin{align*}
  \HS{m}(\Th) 
  = \Big\{\,\vs\in\LTWO(\Omega)\,:\,\restrict{\vs}{\P}\in\HS{m}(\P)\,\textrm{for~all~}\P\in\Th\Big\},
\end{align*}
endowed with the broken Sobolev norm and seminorm
\begin{align}
  \label{eq:Hs:norm-broken}
  \norm{\vs}{m,\hh}^2 = \sum_{\P\in\Th}\norm{\vs}{m,\P}^{2}, \qquad   \snorm{\vs}{m,\hh}^2 = \sum_{\P\in\Th}\snorm{\vs}{m,\P}^{2}
  \qquad\forall\,\vs\in\HS{m}(\Th).
\end{align}

\medskip
Let $\E\in \Edges_{\P^+}\cap\Edges_{\P^-}$ be an internal edge shared
by the polygonal elements $\P^{+}$ and $\P^{-}$, and $\vs$ a function
of $\HONE(\Th)$.
We denote 
the traces of $\vs$ on $\E$ from inside the elements $\P^{\pm}$ by
$\vsE^{\pm}$,
and the unit normal vectors to $\E$ pointing from $\P^{\pm}$ to
$\P^{\mp}$ by $\norE^{\pm}$.
Then, we introduce the \emph{jump operator}, which is defined as
\begin{align*}
  \jump{\vs} = 
  \begin{cases}
    \vsE^{+}\norE^{+}+\vsE^{-}\norE^{-} & \mbox{for every internal edge $\E\in\Edges_{\P^{+}}\cap\Edges_{P^{-}}$,}\\
    \vsE\norE                        & \mbox{for every boundary edge $\E\in \EdgesG$.}
  \end{cases}
\end{align*}
The normal vectors to the edges on the domain boundary $\Gamma$ are
pointing out of $\Omega$.

\medskip
For any positive integer  $k$, the nonconforming space
$\HONEnck(\Th)$ is the subspace of the broken Sobolev space
$\HONE(\Th)$ defined as
\begin{align}
  \label{eq:H1-nc:def}
  \HONEnck(\Th) = 
  \left\{\,
    \vs\in\HONE(\Th)\,:\,\int_{\E}\jump{\vs}\cdot\norE\,\qs\,\dS=0\,
    \,\,\forall\,\qs\in\PS{k-1}(\E),
    \,\,\forall\E\in\Edges
    \,
  \right\}.
\end{align}
The nonconforming space with $k=1$ has the minimal regularity that is
required in the formulation of the VEM, see Section~\ref{sec:ncvem},
and for the convergence analysis, see
Reference~\cite{AyusodeDios-Lipnikov-Manzini:2016}.

\medskip

  For the discontinuous functions of $\HONE(\Th)$,
  $\snorm{\cdot}{1,\hh}$ is only a seminorm.
  However, it becomes a norm on the nonconforming space
  $\HONEnck(\Th)$ since the Poincar\'e-Friedrichs type inequality 
  $\norm{\vs}{0}^{2}\leq\Cs\snorm{\vs}{1,\hh}^{2}$ holds for every
  $\vs\in\HONEnck(\Th)$, $k\geq1$.
  Here, $\Cs$ is a real, positive constant independent of $\hh$,
  cf.~\cite{Brenner:2003};
 see also \cite[Lemma
      2.6]{Bertoluzza-Prada:2020}, which can be leveraged to obtain
    such a bound under weaker conditions on the mesh.

\medskip
We introduce the elliptic projection operator
$\PinP{k}:\HONE(\P)\to\PS{k}(\P)$, defined as follows: for every
$\vs\in\HONE(\P)$, the $k$-degree polynomial $\PinP{k}\vs$ is the
solution of the variational problem:
\begin{align*}
  \int_{\P}\nabla\left(\PinP{k}\vs-\vs\right)\cdot\nabla\qs\,\dV
  = 0 \qquad\forall\qs\in\PS{k}(\P),
\end{align*}
with the additional condition
\begin{align}
  \int_{\partial\P}\left(\PinP{k}\vs-\vs\right)\dS=0,
  \label{condconst1} 
\end{align}
which handles the  kernel  of  the gradient  operator.
%
We note that  $\PinP{k}$ is a
polynomial-preserving operator, i.e.,
$\PinP{k}\qs=\qs$ for every polynomial function $\qs\in\PS{k}(\P)$. 

\medskip
\noindent
Finally, throughout the paper we use the notation $\vs\simeq\ws$,
$\vs\lesssim\ws$ and $\vs\gtrsim\ws$ to indicate that there are
suitable positive, real constants $\cs_*$ and $\cs^*$ such that
$\cs_*\vs\leq\ws\leq\cs^*\vs$, $\vs\leq\cs^*\ws$ and
$\vs\geq\cs^*\ws$.
These constants are independent of the mesh size $\hh$ but may depend
on  other discretization parameters  such as  the mesh
regularity constants and the polynomial order of the method.
The constants $\cs_*$ and $\cs^*$, and the generic
  constant $\Cs$, may have a different
value at each occurence.
Moreover, we use the notation
$\langle\Fs,\vs\rangle$ to indicate the action of $\Fs\in\Vsp$ on the
element $\vs\in\Vs$, $\Vs$ and $\Vsp$ being  different couples
of dual reflexive Hilbert spaces, whose precise definition will be
clear from the context.

\section{The nonconforming virtual element method}
\label{sec:ncvem}
We consider the Poisson problem with homogeneous Dirichlet boundary conditions for
the scalar unknown $\us$:
\begin{subequations}\label{eq:Poisson}
  \begin{align}
    -\Delta\us         &= \fs \phantom{0}  \text{in}\;\Omega,\label{eq:Poisson:a}\\
    \us                &= 0 \phantom{\fs}  \text{on}\;\Gamma,\label{eq:Poisson:b}
  \end{align}
\end{subequations}
where we assume that $\fs\in\LTWO(\Omega)$.

Let $\HONE_0(\Omega)$ denote, as usual, the linear subspace of functions of $\HONE(\Omega)$ with
zero trace on $\Gamma$.
The variational formulation of
problem~\eqref{eq:Poisson:a}-\eqref{eq:Poisson:b} reads as:
\begin{align}
  \mbox{\textit{find  $\us\in\HONE_0(\Omega)$ such that}}\quad
  \as(\us,\vs) = \big(\fs,\vs\big)
  \quad\forall\vs\in\HONE_0(\Omega),
  \label{eq:pblm:var}
\end{align}
where the bilinear form $\as(\cdot,\cdot):\HONE(\Omega)\times\HONE(\Omega)\to\REAL$ is
given by
\begin{equation}
  \as(\us,\vs) = \int_\Omega\nabla\us\cdot\nabla\vs\,\dV,\quad
  \forall\us,\vs\in\HONE(\Omega).
  \label{eq:bilform}
\end{equation}
The essential Dirichlet boundary condition~\eqref{eq:Poisson:b} is
incorporated in the definition of the functional space $\HONE_0(\Omega)$.

\medskip
To formulate the nonconforming virtual element approximation of
variational problem~\eqref{eq:pblm:var}, we need three mathematical
objects:
\begin{itemize}
\item[$\bullet$] the virtual element space $\Vhk$, which is a
  finite-dimensional subspace of the nonconforming space
  $\HONEnck(\Th)$, suitably incorporating a weak form of the homogeneous boundary conditions;

  \medskip
\item[$\bullet$] the virtual element bilinear form
  $\ash(\cdot,\cdot):\Vhk\times\Vhk\to\REAL$, which approximates the
  bilinear form $\as(\cdot,\cdot)$.
  We require $\ash(\cdot,\cdot)$ to be coercive, continuous, and
  computable from the degrees of freedom of its arguments;
 \medskip
  
\item[$\bullet$] an element $\fsh$ of the dual space $
  (\HONEnck(\Th))'$, which approximates the forcing term $\fs$.
\end{itemize}

\medskip
Given these objects, according to the variational form of the continuous problem
in~\eqref{eq:pblm:var}, the virtual element method reads as:~
\begin{align}
  \mbox{\textit{Find $\solh\in\Vhk$ such~that}}\quad
  \ash(\solh,\vsh) = \bscal{\fsh,\vsh}
  \quad\forall\vsh\in\Vhk.
  \label{eq:pblm:vem}
\end{align}
In the following sections we recall the definition of the nonconforming virtual element space $\Vhk$, and of the bilinear form $\ash$.


\subsection{The nonconforming virtual element space}
\label{subsec:NCVEM:space}
Let $\P$ be a generic element of the mesh $\Th$ and $k\geq1$ an integer
number.
We define the nonconforming virtual element space on $\P$ (see
\cite{AyusodeDios-Lipnikov-Manzini:2016}) as
\begin{align}
  \VhkP = \bigg\{\,
  \vsh\in\HONE(\P)\,:\,
  &
  \frac{\partial\vsh}{\partial\nv}\in\PS{k-1}(\E)\,\,\forall\E\subset\EdgesP,\,
  \Delta\vsh\in\PS{k-2}(\P)
  \bigg\},
  \label{eq:VhkP:def}
\end{align}
and its ``modified'' or ``enhanced" variant (see
\cite{Ahmad-Alsaedi-Brezzi-Marini-Russo:2013,Cangiani-Manzini-Sutton:2017})
as
\begin{align}
  \Vhken(\P) = \bigg\{\,
  \vsh\in\HONE(\P)\,:\,
  &
  \frac{\partial\vsh}{\partial\nv}\in\PS{k-1}(\E)\,\,\forall\E\subset\EdgesP,\,
  \Delta\vsh\in\PS{k}(\P),\nonumber\\[0.25em]
  &
  \int_{\P}\big(\vsh-\PinP{k}\vsh)\mono\dV=0
  \quad\forall\mono\in\monomials{k}{\P}\setminus\monomials{k-2}{\P}
  \bigg\}.
  \label{eq:VhkenP:def}
\end{align}
We recall that
$\monomials{k}{\P}\setminus\monomials{k-2}{\P}$, in the definition above, is the subset of the scaled monomials of degree  equal to $k-1$ and $k$.

\medskip

\noindent

The following key properties hold for $\VhkP$ and $\VhkenP$:

\begin{enumerate}[label=(\roman{*}),parsep=2mm,leftmargin=*]
\item the polynomial space $\PS{k}(\P)$ is a subspace of both $\VhkP$
  and $\VhkenP$;
\item the virtual element functions in both $\VhkP$ and $\VhkenP$ are
  uniquely determined by the following set of \emph{degrees of
  freedom}:
  \begin{itemize}
  \item[\TERM{D1}] the values of the polynomial moments of $\vsh$ of
    order up to $k-1$ on each edge $\E\in\EdgesP$:
    \begin{equation} 
      \frac{1}{\hE}
      \int_{\E}\vsh\,
      \mono\,\dS
      \quad\forall  \mono \in \monomials{k-1}{\E},\,
      \forall\E\in\EdgesP;
      \label{eq:dofs:01}
    \end{equation}
  \item[\TERM{D2}] the values of the polynomial moments of $\vsh$ of
    order up to $k-2$ on $\P$:
    \begin{align}
      \frac{1}{\mP}\int_{\P}\vsh\,\mono \,\dV
      \quad\forall\mono\in\monomials{k-2}{\P}.
      \label{eq:dofs:02}
    \end{align}
  \end{itemize}
\end{enumerate}

\begin{remark}\label{rem:basis}
Other choices are possible for the degrees of freedom. In \TERM{D1} and \TERM{D2} the sets $\monomials{k-1}{\E}$ and $\monomials{k-2}{\P}$ can be replaced with any other basis for the spaces $\PS{k-1}(\E)$ and $\PS{k-2}(\P)$. We point out that the stabilizing bilinear terms that we are going to construct do not depend on the particular basis chosen and that the bounds that we will prove hold independently of such a choice.   \end{remark}

\medskip
Property $(i)$ is a direct consequence of the space definition and
guarantees the optimal order of approximation. Property $(ii)$ has
been proven in~\cite{AyusodeDios-Lipnikov-Manzini:2016}.

  The
polynomial projection $\PinP{k}\vsh$ is computable using only the
values from the linear functionals in \TERM{D1}--\TERM{D2}.
We recall that 
  for $k>1$, the average of functions in $\VhkP$ is
  computable, and we could replace~\eqref{condconst1} with the condition that
  $\PinP{k}\vs-\vs$ is average free in $\P$.

\medskip
The \emph{global nonconforming virtual element space $\Vhk$ of order
$k\geq1$ subordinate to the mesh $\Th$} is obtained by gluing together
the elemental spaces $\VhkP$ to form a subspace of the nonconforming
space $\HONEnck(\Th)$.
The formal definition reads as:
\begin{align}
  \Vhk:=\Big\{\,\vsh\in\HONEnck(\Th)\,:\,\restrict{\vsh}{\P}\in\VhkP,
  \,\,\,\forall\P\in\Th,\,\,\, \int_{\E}\vsh q = 0, \ \forall q \in \PS{k-1}(\E),\ \forall \E \in \EdgesG \Big\}.
  \label{eq:def:nvem0}
\end{align}
The boundary conditions are enforced in weak form in the definition of the space, by requiring that, for all boundary edges $\E \in \EdgesG$, $v_{|e}$ is orthogonal to the space of  polynomials of degree at most $k-1$ on $\E$. A similar definition holds for the global space $\Vhken$, obtained by
gluing together the elemental spaces $\VhkenP$.
The set of degrees of freedom for $\Vhk$ and $\Vhken$ is given by
collecting the values \textbf{(D1)} for all the mesh edges and
\textbf{(D2)} for all the mesh elements.
The unisolvence of such degrees of freedom
in the global space $\Vhk$ follows from the unisolvence of
the degrees of freedom \textbf{(D1)}--\textbf{(D2)} in each elemental
space, cf. \cite{AyusodeDios-Lipnikov-Manzini:2016}.

\subsection{Virtual element discretization}
\label{subsec:approx}
  Hereafter, we only detail the formulation of the virtual element
  discretization for the \emph{nonenhanced} space $\VhkP$.
  The corresponding formulation for the enhanced space $\Vhken$ is
  identical.

The virtual element approximation is defined on the broken Sobolev
space $\HONEnck(\Th)$.
Since the functions of this space can be discontinuous at the
elemental boundaries $\bP$, we extend the bilinear form
$\as(\cdot,\cdot)$ to the broken Sobolev
space $\HONE(\Th)$ as follows:
\begin{align*}
\as(\us,\vs) 
= \sum_{\P\in\Th}\asP(\us,\vs)
= \sum_{\P\in\Th}\int_{\P}\nabla\us\cdot\nabla\vs\dV
\qquad\forall\us,\vs\in\HONE(\Th).
\end{align*}
The discrete bilinear form $\ash(\cdot,\cdot)$ is given by the sum of
elemental contributions
\begin{align}
  \ash(\ush,\vsh)&=\sum_{\P\in\Th}\ashP(\ush,\vsh),\label{eq:ash:global:def}
\end{align}
with
\begin{align}
  \ashP(\ush,\vsh) = \asP (\PinP{k}\ush,\PinP{k}\vsh) +
  \sP\Big((I-\PinP{k})\ush,(I-\PinP{k})\vsh\Big),
  \label{eq:ash:local:def}
\end{align}
where $\sP(\cdot,\cdot)$ can be any computable, symmetric and positive
semidefinite bilinear form such that
\begin{equation}\label{eq:requirement:S}
  \Cs_*\asP(\vsh,\vsh)
  \leq\sP(\vsh,\vsh)
  \leq\Cs^*\asP(\vsh,\vsh)
  \quad\forall\vsh\in\VhkP\cap\KER\big(\PinP{k}\big)
\end{equation}
for some pair of positive constants $\Cs_*$ and $\Cs^*$ that are
independent of $\P$ and $\hh$, where
$\KER\big(\PinP{k}\big)=\big\{\vs\in\HONE(\P):\PinP{k}\vs=0\big\}$ is
the kernel of the projection operator $\PinP{k}$.


\medskip
{Provided \eqref{eq:requirement:S} holds, the}
discrete bilinear form $\ashP(\cdot,\cdot)$ satisfies the following
properties:

\begin{itemize}[topsep=2mm]
\item[-] {\emph{$k$-consistency}}: for all $\vsh\in\VhkP$ and for all
  $\qs\in\PS{k}(\P)$ it holds that
  \begin{align}
    \ashP(\vsh,\qs) &= \asP(\vsh,\qs);
    \label{eq:k-consistency:ah}
  \end{align}
  
  \medskip
\item[-] {\emph{stability}}: there exist two positive constants
  $(\alpha_*,\,\alpha^*)$, independent of $\P$ and $\hh$, such that
  \begin{align}
    \alpha_*\asP(\vsh,\vsh)
    &\leq\ashP(\vsh,\vsh)
    \leq\alpha^*\asP(\vsh,\vsh)\quad\forall\vsh\in\VhkP.
    \label{eq:stability:a}
  \end{align}
\end{itemize}
In particular, the first term in the definition of $\ashP$
in~\eqref{eq:ash:local:def} provides the $k$-consistency of the
method, i.e., the exactness on polynomials of degree $k$, which follows
from the invariance of $\PinP{k}$ on polynomials.
The second term in the definition of $\ashP$ ensures the stability of
the method, cf.  also~\cite{AyusodeDios-Lipnikov-Manzini:2016}, and is
zero if one of its two entries is a polynomial of degree at most $k$.
The stability property follows from a straightforward calculation, by
taking $\alpha^*=\max(1,\Cs^*)$ and $\alpha_*=\min(1,\Cs_*)$,
cf.~\cite{BeiraodaVeiga-Brezzi-Cangiani-Manzini-Marini-Russo:2013}.

\medskip
As far as the right-hand side is concerned, we approximate $f$ with
$\fsh$ such that $\langle\fsh,\vs\rangle$ is computable (we refer
to
\cite{BeiraodaVeiga-Brezzi-Cangiani-Manzini-Marini-Russo:2013,AyusodeDios-Lipnikov-Manzini:2016}
for more details).

\medskip
In this setting, we can prove the abstract convergence result stated
in Theorem~\ref{theorem:abstract:convergence:scalar} below.
We report this result omitting its proof,
which can be found
in~\cite{AyusodeDios-Lipnikov-Manzini:2016}.

\begin{theorem}[Abstract convergence result]
  \label{theorem:abstract:convergence:scalar}
  Let $\us\in\Vs$ be the solution to problem~\eqref{eq:pblm:var} and
  $\us_h\in\Vhk$ (or $\Vhken$) be the solution to
  problem~\eqref{eq:pblm:var} in the nonconforming setting introduced
  above.
  Then, it holds that
  \begin{align}
    \alpha_*	\norm{\us-\us_h}{1,\hh}
    &\leq\sup_{\ws\in\Vhk\setminus\{0\}}\frac{\ABS{\bil{\fsh}{\wsh}-\scalV{\fs}{\wsh}}}{\SNORM{\wsh}{1,\hh}}
    + \sup_{\ws\in\Vhk\setminus\{0\}}\frac{\ABS{\calN(\us;\wsh)}}{\SNORM{\wsh}{1,\hh}}
    \nonumber\\[0.5em] 
    &	+ 	\alpha^*\inf_{\vsh\in\Vhk}\SNORM{\us-\vsh}{1,\hh} + 	(\alpha^*+1) \inf_{\qs\in\PS{\Th}}
    \SNORM{\us-\qs}{1,\hh}.
  \end{align}
  where $\alpha_*$ and $\alpha*$ are defined in \eqref{eq:stability:a}, and  $\calN(\us;\cdot)$ is the continuous linear functional
  \begin{align*}
    \calN(\us;\wsh) = \as(\us,\wsh) - \big(\fs,\wsh\big),
  \end{align*}
  which defines the conformity error for every virtual element
  function $\wsh$ in $\Vhk$ or $\Vhken$.
\end{theorem}

Bounds on the different terms on the right hand side are provided in
\cite{AyusodeDios-Lipnikov-Manzini:2016} and, if the solution $u$ is
smooth, they yield optimal error estimates, provided
$(1+\alpha^*)/\alpha_*$ is bounded uniformly in $h$. The aim of this
paper is to design the stabilization term $\sP$ so that this holds
true.


\section{Algebraic construction of semi-inner products and 
  semi-norms in abstract finite dimensional subspaces}
\label{sec:dual:scalar:product}

The focus of this paper is on the construction of suitable bilinear
forms $\sP$ satisfying \eqref{eq:requirement:S} under conditions on
the mesh $\Th$ as weak as possible.
  This problem has been addressed in \cite{BeiraodaVeiga-Lovadina-Russo:2017} for the
  conforming virtual element method. In that case, after showing that it is sufficient for the stabilizing bilinear form to only act on the trace of the virtual function on $\bP$, one can take advantage of the computability of such traces, which are known piecewise polynomials.
  In the nonconforming case, we have an additional difficulty: contrary to what happens in the conforming case,
  the trace on $\bP$ of the nonconforming virtual element functions
  is not computable, and we only have access to the degrees of
  freedom, which correspond to known functionals in the space
  $(\HONE(\P))'$. 
  Our idea is to design suitable bilinear forms on the space spanned
  by such functionals and build the stabilization term by a duality
  technique first introduced in \cite{Bertoluzza_algebraic_dual}, which, in this section, we present in a general abstract setting.

\medskip
Let $\Vs$ be a Hilbert space and $\Vsp$ its dual space, respectively
endowed with the inner products $\scalV{\spcdot}{\spcdot}$ and
$\scalVp{\spcdot}{\spcdot}$, and the induced norms $\normV{\spcdot}$
and $\normVp{\spcdot}$.
We denote the duality product by $\bil{\spcdot}{\spcdot}$ and use
Roman fonts for the elements of $\Vs$ and Greek fonts for the elements
of $\Vsp$.
In addition, we consider:
\begin{itemize}[parsep=2mm,leftmargin=*,labelindent=\parindent]
\item a continuous seminorm $\snormV{\spcdot}:\Vs\to\REAL^+$ with
  kernel $\Wbar\subset\Vs$; without loss of generality, after possibly
  multiplying the seminorm by a fixed constant, we can assume that
  $\snormV{ \cdot }\leq \normV {\cdot}$;
  
\item a projection operator $\Pbar:\Vs\to\Wbar$, which is linear,
  bounded and idempotent, i.e.,
  \begin{align*}
    \normV{\Pbar\vs}\leq\CPihat\normV{\vs}\quad\text{for~every $\vs\in\Vs$},
    \qquad\text{and}\qquad
    \Pbar^2=\Pbar;
  \end{align*}
  
\item the seminorm $\snormVp{\spcdot}:\Vsp\to\REAL^+$, 
  defined by duality with the seminorm $\snormV{\spcdot}$:
  \begin{align}
    \forall\eta\in\Vsp:\,
    \snormVp{\eta}=\sup_{\vs\in\KER(\Pbar)}\frac{\bil{\eta}{\vs}}{\snormV{\vs}};
    \label{eq:snorm:Vp:def}
  \end{align}
  
\item the projection operator $\Pbars:\Vsp\to\Vsp$, which is the
  adjoint of $\Pbar$ with respect to the duality product
  $\bil{\spcdot}{\spcdot}$:
  \begin{align*}
    \bil{\Pbars\eta}{\vs}
    =   \bil{\eta}{\Pbar\vs} 
    \qquad\forall\eta\in\Vsp,\,\vs\in\Vs.
  \end{align*}
  This definition implies that the operator $\Pbars$ is also linear,
  bounded and idempotent, i.e.,
  \begin{align*}
    \normVp{\Pbar\eta}\leq\CPihat\normVp{\eta}\quad\text{for~every $\eta\in\Vsp$},
    \qquad\text{and}\qquad
    (\Pbars)^2=\Pbars.
  \end{align*}
\end{itemize}

\medskip
We make the following assumptions:

\medskip\noindent
\TERM{A1} the space $\Wbar=\KER\big(\snormV{\spcdot}\big)$ is finite
dimensional; 

\smallskip\noindent
\TERM{A2} a Poincar\'e type inequality of the form  $\normV{\vs}\leq\Cpoinc\snormV{\vs}$
holds on $\KER(\Pbar)$.

\medskip
\noindent
Assumptions \TERM{A1}--\TERM{A2} and the previous definitions imply
that
\begin{itemize}[parsep=2mm,leftmargin=*,labelindent=\parindent]
\item the subspace $\Wbars=\Pbars(\Vsp)\subset \Vsp$ is finite
  dimensional and coincides with the kernel of the dual seminorm
  $\snormVp{\spcdot}$, i.e., $\Wbars=\KER(\snormVp{\spcdot})$;
\item  the following equivalence relations in  $\Vs$ holds:
  \begin{align}
\snormV {\vs} =    \snormV{\vs-\Pbar\vs}\leq  \normV{\vs-\Pbar\vs} \leq \Cpoinc | v |
    \quad\forall\vs\in\Vs,
    \label{eq:equivalence:norm:snorm:V}
  \end{align}
  from which, by triangular inequality, we can prove the
  Poincar\'e-like inequality
  \begin{align*}
    \normV{\vs}\leq\Cpoinc\snormV{\vs}+\normV{\Pbar\vs}
    \quad\forall\vs\in\Vs;
  \end{align*}
\item the following equivalence relations in $\Vsp$ holds:
  \begin{align}
    \Cpoinc^{-1}\snormVp{\eta}
    = \Cpoinc^{-1}\snormVp{\eta-\Pbars\eta}
    \leq\normVp{\eta-\Pbars\eta} \leq \snormVp{\eta}
    \quad\forall\eta\in\Vsp; 
    \label{eq:equivalence:norm:snorm:Vp}
  \end{align}
\item the seminorm $\snormV{\spcdot}$ and the seminorm defined in
  $\Vs$ by duality with the seminorm $\snormVp{\spcdot}$ in $V'$ are
  equals:
  \begin{align}
    \snormV{\vs} =
    \sup_{\eta\in\KER(\Pbars)}\frac{\bil{\vs}{\eta}}{\snormVp{\eta}};
    \label{eq:reflexivity}
  \end{align}
\item the identity and inequality chain
    \begin{align}
      \ABS{\bil{\eta}{\vs}}
      = \ABS{\bil{\eta-\Pbars\eta}{\vs}}
      = \ABS{\bil{\eta}{\vs-\Pbar\vs}}
      \leq \snormVp{\eta}\,\snormV{\vs-\Pbar\vs}
      = \snormVp{\eta}\,\snormV{\vs}
      \label{eq:bil:bound:seminorm}
    \end{align}
    holds for every $\eta\in\KER(\Pbars)$ and $\vs\in\Vs$ and follows
    from the definition of the seminorms $\snormV{\vs}$ and
    $\snormVp{\eta}$, and
    \eqref{eq:equivalence:norm:snorm:V}-\eqref{eq:equivalence:norm:snorm:Vp}.
\end{itemize}

\medskip
We now introduce two finite dimensional subspaces $\Ws\subset\Vs$ and
$\Wss \subset \Vsp$, and make the further assumptions:

\smallskip\noindent
\TERM{A3} $\Wbar\subset\Ws$ and $\Wbars\subset \Wss$;

\smallskip\noindent
\TERM{A4} the two following inf-sup conditions hold for the pair of
spaces $\Ws$ and $\Wss$:
\begin{align}
  \inf_{\ws\in\Ws}\sup_{\eta\in\Wss}\frac{\bil{\eta}{\ws}}{\normVp{\eta}\,\normV{\vs}}\geq\constinfsup
  \quad\textrm{and}\quad
  \inf_{\eta\in\Wss}\sup_{\ws\in\Ws}\frac{\bil{\eta}{\ws}}{\normVp{\eta}\,\normV{\vs}}\geq\constinfsup.
  \label{eq:abstractinfsup}
\end{align}
Remark that if both inf-sup conditions hold, then we have that
$\dims(\Ws) = \dims(\Wss)$.
On the other hand, if $\dims(\Ws) = \dims(\Wss)$, then either one of the
two inf-sup conditions in \eqref{eq:abstractinfsup} implies the other.

Furthermore, using the inf-sup conditions above we can prove the
equivalence relation in $\Wss$:
\begin{align}
  \Cpoinc^{-1}\beta \snormVp{\eta}
  \leq \sup_{\ws\in\Ws\cap\KER(\Pbar)}\frac{\bil{\eta}{\ws}}{\snormV{\vs}}
  \leq \snormVp{\eta}
  \qquad\forall\eta\in\Wss.
  \label{eq:seminorm:equivalence:V:W}
\end{align}

\medskip
\noindent
Let now $\Ns=\dims(\Ws)=\dims(\Wss)$ and $\Ms=\dims(\Wbar)=\dims(\Wbars)$.
We consider a set of elements $\basis=\{\ebase_m\}_{m=1,\ldots,\Ns}$,
forming a basis for the space $\Ws$, and the
corresponding set of elements $\basis^*=\{\te_n\}_{n=1,\ldots,\Ns}$,
forming a basis for the space $\tW $ and satisfying the
biorthogonality property
\begin{align}
  \bil{\te_n}{\ebase_m} = \delta_{n,m}
  \quad\ms,\ns=1,\ldots,\Ns.
  \label{eq:biorthogonality}
\end{align}
The validity of the inf-sup condition \eqref{eq:abstractinfsup}
implies that such a basis $\basis^*$ exists.
Analogously, we will consider a set of elements
$\hbasis=\{\hebase_j\}_{j=1,\ldots,\Ms}$, forming a basis for $\Wbar$
and the corresponding set of elements
$\hbasis^*=\{\hte_i\}_{i=1,\ldots,\Ms}$, forming a basis for the space
$\Wbars $, and such that, for all $\vs\in \Vs$, we have
\begin{align}
  \Pbar \vs = \sum_{i=1}^\Ms \langle \hte_i,\vs \rangle \hebase_i.
  \label{eq:analyticPihat}
\end{align}
  As $\Pbar$ is a
  projector,
  the basis sets $\hbasis$ and $\hbasis^*$ satisfy a biorthogonality
  property analogous to \eqref{eq:biorthogonality}.

As $\basis$ and $\basis^*$ are bases for $\Ws$ and $\Wss$, we can
expand any $\vs\in\Ws$ and $\zeta\in\Wss$ as:
\begin{align*}
  \vs=\sum_{m=1}^{\Ns}\vs_m\ebase_m \quad\textrm{and}\quad
  \zeta=\sum_{n=1}^{\Ns}\zeta_n\te_n.
\end{align*}
The $\Ns$-sized vectors $\vv=(\vs_m)$ and $\zetav=(\zeta_n)$ collect
the expansion coefficients of $\vs$ and $\zeta$ and
  are respectively referred to as \emph{the vector representations} of
  $\vs$ and $\zeta$.
We will use an analogous notations for the elements of $\Wbar$ and
$\Wbars$, which will be represented by $\Ms$-sized vectors collecting
the coefficients of their expansions in terms of the bases $\hbasis$
and $\hbasis^ *$.
According with this basis choice, thanks to the biorthogonality
property, we can express the duality product between $\zeta\in \Wss$
and $\vs\in \Ws$ as follows:
\begin{align}
  \bil{\zeta}{\vs} = \zetav^T\vv.
  \label{eq:duality-product:vector:representation}
\end{align}
Now, we consider a symmetric and positive semidefinite matrix
$\matS\in\REAL^{\Ns\times\Ns}$ and the bilinear form
$\ss(\cdot,\cdot):\Ws\times\Ws\to\REAL$ defined by
\begin{align}
  \ss(\vs,\ws) = \wv^T\matS\vv,
  \label{eq:normV:matrix:representation}
\end{align}
where $\vv,\wv\in\REAL^{\Ns}$ are the vector representations of
$\vs,\ws\in\Ws$.
We assume that  there
exist positive constants $\CsM$ and $\alpha$ such that for all $\vs,
\ws \in \Ws$
\begin{align}
  \ss(\vs,\ws)\leq\CsM\snormV{\vs}\,\snormV{\ws}
  \quad\textrm{and}\quad
  \alpha\snormV{\vs}^2\leq\ss(\vs,\vs).
  \label{eq:continuity-semicoercivity:s}
\end{align}


We next introduce a reflexive generalized inverse $\invS \in
\REAL^{\Ns\times\Ns}$ of $\matS$, which we define as follows.
Let $\matP\in\REAL^{\Ms\times\Ns}$ be the matrix representation of the
projection operator $\Pbar$, defined in such a way that
$\matP\wv\in\REAL^{\Ms}$ is the vector representing $\Pbar\ws \in
\Wbar$ if $\wv\in\REAL^{\Ns}$ is the vector representing $\ws\in\Ws$.
  The matrix $\matP$ has maximum rank,
i.e. $\RANK(\matP)=\min(\Ms,\Ns)=\Ms$, and it
projects onto the kernel of $\matS$, which coincides with the kernel
of $|\cdot|$.

Then, given $\etav\in\REAL^{\Ns}$, the saddle point problem
\begin{align}
  \left\{\!\!
  \begin{array}{ll}
    \matS\wv + \matP^T\lbdv &= \etav,\\
    \matP\wv                &= \zerov,
  \end{array}
  \right.
  \label{eq:reflexive:generalized:inverse:mixed:pblm}
\end{align}
has a unique solution $(\wv,\lbdv)\in\REAL^{\Ns}\times\REAL^{\Ms}$,
and the corresponding coefficient matrix
is nonsingular~\cite{Brezzi:2002:Durham}. 
Then, we set
\begin{equation}\label{defSinv}
  \invS \etav = \wv
  \qquad \text{with $(\wv,\lbdv)\in \REAL^ \Ns \times \REAL^ \Ms$ solution to \eqref{eq:reflexive:generalized:inverse:mixed:pblm},}
\end{equation}
or, equivalently
\begin{equation}\label{defSinv:alt}
 \invS \etav  =
 \big( \matI \quad 0 \big)
 \left(\begin{array}{cc}\matS & \matP^T \\ \matP & \zerov \end{array}\right)^{-1}
 \left(\begin{array}{c}\matI \\ \zerov\end{array}\right)\,\etav,
\end{equation}
which  gives us
\begin{align*}
  \invS = 
  \big( \matI \quad 0 \big)
  \left(\begin{array}{cc}\matS & \matP^T \\ \matP & \zerov \end{array}\right)^{-1}
  \left(\begin{array}{c}\matI \\ \zerov\end{array}\right),
\end{align*}
from which we also deduce that $\invS$ is a symmetric matrix.
  In this setting, the saddle point problem
  \eqref{eq:reflexive:generalized:inverse:mixed:pblm} is well posed,
  and
$\wv=\invS\etav\in\REAL^{\Ns}$ if and only if there exists
a vector $\lbdv\in\REAL^{\Ms}$ such that the pair $(\wv,\lbdv)$
satisfies \eqref{eq:reflexive:generalized:inverse:mixed:pblm}.
 By exploiting such a fact, it can be
  shown that  the matrices $\matS$ and $\invS$ satisfy the identities
\begin{align}
  \matS\invS\matS = \matS
  \quad\textrm{and}\quad
  \invS\matS\invS=\invS,
  \label{eq:reflexive:generalized:inverse:properties}
\end{align}
so that $\invS$ is indeed a reflexive generalized inverse of $\matS$
and viceversa.
  If $\matPs\in\REAL^{M\times N}$ is the matrix representing $\Pbars$,
  we can prove that

\begin{align}
  \invS\matS = \matI_{\Ns} - (\matPs)^T\matP,
  \label{eq:invSmatS}
\end{align}
where $\matI_{\Ns}\in\REAL^{\Ns\times\Ns}$ is the identity
matrix.

Using the biorthogonality property
  \eqref{eq:biorthogonality}, we can show that $\matPs$ coincides with the matrix 
  representing the inclusion of $\Wbar$ into $\Ws$: if $\widehat w \in \Wbar \subset \Ws$ is written as
  \[
  \widehat w = \sum_{i=1}^M \widehat w_i \hebase_i = \sum_{n=1}^N  w_n \ebase_n,
  \]
  the vectors
  $\widehat\wv = (\widehat w_i)$ and $\wv = (w_n)$ satisfy $\wv = \matPs \widehat \wv$. We can then see that the matrix $\matI_{\Ns} - (\matPs)^T\matP$ represents the operator $(1-\Pbar)$. 

\begin{remark}
  The
  projector $\Pbar$  has different matrix
  representations, depending on whether it is seen as an operator
  from $\Ws$ to $\Wbar$ or as an operator
  from $\Ws$ to $\Ws$.
  In the first case, $\Pbar w $  is represented
  with respect to the basis $\hbasis$ and the
  operator is represented by the $\Ms \times \Ns$ matrix $\matP$.
  In the second case, the basis used to express $\Pbar w $
   is $\basis$  and the
  operator is represented by the $\Ns \times \Ns$ matrix $(\matPs)^T \matP$. An analogous observation holds for the operator $\Pbars$. 
\end{remark}

\

We have now all the ingredients to define a bilinear form on $\Wss$
acting, on such a subspace, as a semi-inner product inducing a
semi-norm equivalent to the dual semi-norm $| \cdot |_*$.
More precisely, the bilinear form $\sss: \Wss \times \Wss \to \REAL$
is defined by
\begin{align*}
  \sss(\eta,\zeta) = \etav^T \invS \zetav,
\end{align*}
where, once again, $\etav$, $\zetav\in\REAL^{\Ns}$ are the vector
representation of $\eta$, $\zeta\in\Wss$.
As $\invS$ is symmetric and positive semidefinite, $\sss(\cdot,\cdot)$
is indeed a semi-inner product on $\Wss$, and we have the following
Proposition.
\begin{proposition}\label{prop:2.7}
  For every $\eta$, $\zeta\in\Wss$, it holds that
  \begin{equation}\label{eq:prop:2.7}
    \Cpoinc^{-2}  \CsM^{-1}\beta^2\snormVp{\eta}\leq\sss(\eta,\eta),
    \quad\textrm{and}\quad
    \sss(\eta,\zeta)
    \leq \alpha^{-1}\snormVp{\eta}\,\snormVp{\zeta}.
  \end{equation}
\end{proposition}
\begin{proof}
  Let $\eta\in\Wss$ and $\ws\in\Ws\cap\KER(\Pbar)$ with vector
  representations $\etav$, $\wv\in\REAL^{\Ns}$, and recall that
  $(\matI_{\Ns}-(\matPs)^T\matP)\wv$ is the vector representation of
  $\ws-\Pbar\ws$.
  Since $\Pbar\ws=0$, \eqref{eq:invSmatS} yields
  \begin{align*}
    \bil{\eta}{\ws}
    = \bil{\eta}{\ws-\Pbar\ws}
    = \etav^T\big(\matI_{\Ns}-(\matPs)^T\matP\big)\wv
    = \etav^T\invS\matS\wv.
  \end{align*}
  The matrix $\invS$ is symmetric and positive semidefinite, so there
  exists a $\Ns \times \Ns$ matrix $\matG$ such that
  $\invS=\matG^T\matG$.
  We substitute such decomposition, we apply the Cauchy-Schwarz
  inequality and the first indentity
  of~\eqref{eq:reflexive:generalized:inverse:properties}, and we find
  that
  \begin{align*}
    \etav^T\invS\matS\wv
    &= (\matG\etav)^T\,(\matG\matS\wv)
    \leq \sqrt{\etav^T\matG^T\matG\etav}\,\sqrt{\wv^T\matS^T\matG^T\matG\matS\wv}
    = \sqrt{\etav^T\invS\etav}\,\sqrt{\wv^T\matS\invS\matS\wv}
    \\[0.25em]
    &= \sqrt{\etav^T\invS\etav}\,\sqrt{\wv^T\matS\wv}
    \leq \CsM^{1/2} \sqrt{\etav^T\invS\etav}\,\snormV{\ws}.
  \end{align*}
  Then for
  $\ws \in\Ws\cap\KER(\Pbar)$ we have that
  \begin{align*}
    \frac{\bil{\eta}{\ws}}{\snormV{\ws}}\leq \CsM^{1/2} \sqrt{\etav^T\invS\etav},
  \end{align*} and,
  using the lower bound in ~\eqref{eq:seminorm:equivalence:V:W}, we find that for every $\eta\in\Wss$
  \begin{align*}
    \snormVp{\eta} \leq\Cpoinc \beta^{-1}
    \sup_{\ws\in\Ws\cap\KER(\Pbar)}\frac{\bil{\eta}{\ws}}{\snormV{\ws}}
    \leq \Cpoinc \beta^{-1} \CsM^{1/2}  \sqrt{\etav^T\invS\etav},
  \end{align*}
  which gives us the first bound in \eqref{eq:prop:2.7}.
  Conversely, for any given $\eta\in\Wss$ and its vector
  representation $\etav\in\REAL^{\Ns}$, we let $\ws\in\Ws$ be the
  element with vector representation $\wv=\invS\etav$.
  Then, we start from the vector representation of the duality
  product~\eqref{eq:duality-product:vector:representation}, use
  inequality~\eqref{eq:bil:bound:seminorm}, the matrix representation
  of $\ws$ in~\eqref{eq:normV:matrix:representation}, and the second
  indentity of~\eqref{eq:reflexive:generalized:inverse:properties}
  and, since, by the definition of $\invS$, $w \in \ker(\Pbar)$, we
  can write:
  \begin{align*}
    \etav^T\invS\etav
    &= \etav^T\wv
    = \bil{\eta}{\ws}
    \leq \snormVp{\eta}\,\snormV{\ws}
    \leq \alpha^{-1/2} \snormVp{\eta}\,\sqrt{\wv^T\matS\wv}\\[2mm]
    & = \alpha^{-1/2} \snormVp{\eta}\,\sqrt{\etav^T\invS\matS\invS\etav}
    = \alpha^{-1/2} \snormVp{\eta}\,\sqrt{\etav^T\invS\etav}.
  \end{align*}
  We divide both sides by $\sqrt{\etav^T\invS\etav}$ and obtain that
  $\sqrt{\etav^T\invS\etav}\leq\alpha^{-1/2}\snormVp{\eta}$.
  Analogously, we have that
  $\sqrt{\zetav^T\invS\zetav}\leq\alpha^{-1/2}\snormVp{\zeta}$.

  \medskip
  Then, to prove the second relation in~\eqref{eq:prop:2.7}, we simply
  apply the Cauchy-Schwartz inequality and the above bounds, and we
  obtain
  \begin{align*}
    \sss(\eta,\zeta)
    = \zetav^T\invS\etav
    \leq \sqrt{\zetav^T\invS\zetav}\,\sqrt{\etav\invS\etav}
    \leq \alpha^ {-1} \snormVp{\eta}\,\snormVp{\zeta}.
  \end{align*}
\end{proof}

\begin{remark} 
  \label{rem:twicedual}
  We can also use the above approach to build
  a semi-inner product equivalent to $\snormV{\cdot}$ in any finite
  dimensional subspace $\widetilde{\Ws}\subset\Vs$ containing $\Wbar$
  and verifying inf-sup conditions of the form
  \eqref{eq:abstractinfsup}, with $\widetilde{\Ws}$ replacing $\Ws$.
  By applying the same reasoning as above with the roles of $\Vs$ and
  $\Vsp$ switched, we  introduce the reflexive
  generalized inverse $\invinvS$ of $\invS$ defined as
  \begin{align*}
    \invinvS = 
    \big( \matI \quad 0 \big)
    \left(\begin{array}{cc}\invS & (\matP^*)^T \\ \matP^*& \zerov \end{array}\right)^{-1}
    \left(\begin{array}{c}\matI \\ \zerov\end{array}\right),   
  \end{align*}
  where $\matP^*$ is the matrix realizing the adjoint projector
  $\Pbar^*$. Under our assumptions it is possible to prove that
  $\invinvS = \matS$.
  Then, we  define the bilinear form
  $\widetilde{s}:\widetilde{\Ws}\times\widetilde{\Ws}\to\REAL$ as
  \begin{align*}
    \widetilde{s}(\widetilde\ws,\widetilde\vs):=\vv^T\invinvS\wv = \vv^T\matS\wv,
  \end{align*}
  where $\vv$ and $\wv$ are, this time, the vectors
  representing the functions $\widetilde\ws$ and $\widetilde\vs$ with
  respect to the basis
  $\widetilde\basis=\{\widetilde\ebase_m\}_{m=1,\ldots,\Ns}$ for
  $\widetilde \Ws$, that is biorthogonal to $\basis^*$.
  By applying Proposition \ref{prop:2.7},
    we find that
  \begin{align*}
    \widetilde{s}(\widetilde \vs,\widetilde\vs)\gtrsim\ABS{\vs}^2,
    \qquad
    \widetilde{s}(\widetilde \vs,\widetilde\ws)\lesssim\ABS{\vs}\,\ABS{\ws}
  \end{align*}
 for all $\widetilde\vs,\widetilde\ws \in
    \widetilde\Ws$.
The implicit constants in these bounds depend only on the constants
  $\Cpoinc$, $\CsM$, $\beta$ and $\alpha$, and on the inf-sup constant
  $\widetilde \beta$ relative to the duality between $\Wss$ and
  $\widetilde \Ws$.
  In other words, the ``stiffness'' matrix $\matS$ constructed on
  $\Ws$ can be used to define an equivalent semi-inner product on any
  other subspace $\widetilde\Ws\subset\Vs$ containing $\Wbar$ and
  verifying an inf-sup conditions of the form
  \eqref{eq:abstractinfsup}.
\end{remark}

  

\newcommand{\Poly}{\mathbb{P}}
\newcommand{\Polyort}{\mathring{\mathbb{P}}}
\newcommand{\PolyHarm}{\mathcal{A}}
\newcommand{\constants}{\mathbb{P}_0}
\newcommand{\pibou}{\Pi^\partial_\P}

\section{Stabilization in the nonconforming virtual element method}
\label{sec:stab:theory}

We now focus on the problem of building stabilization terms for the
nonconforming virtual element method described in Section~\ref{sec:ncvem}.
The aim is to achieve robustness with respect to the mesh size, under
as weak assumptions as possible, on the shape of the elements. We start by making the following minimal shape regularity assumption, which we assume to be always satisfied:
\begin{itemize}[parsep=1mm,topsep=2mm]
	\item[\SRone] \label{eq:SR.i} there exist a positive constant
	$\gamma_0$ such that for all $\Th$, every element $\P\in\Th$ is
	star-shaped with respect to a ball of radius greater than
	$\gamma_0\hP$.
\end{itemize}
  We have the following lemma, whose proof is postponed to
  Appendix~\ref{appendix:enhanced}.
\begin{lemma}
  \label{lemma:enhanced}
  Let $\vsh\in\VhkP$ and $\vshh\in\VhkenP$ be  virtual
  element functions satisfying
  \begin{align}\label{eq:cond:v:vhat}
    \int_{\E}\vsh\eta
    = \int_{\E}\vshh\eta\quad\
    \forall\eta\in\PS{k-1}(\E),
    \forall\E\in\EdgesP,\qquad\int_{\P}\vsh\qs
    = \int_{\P}\vshh\qs,\
    \forall\qs\in\PS{k-2}(\P).
  \end{align}
  Then, it holds that
  \begin{align*}
    \SNORM{\vshh}{1,\P} \simeq \SNORM{\vsh}{1,\P}.
  \end{align*}
\end{lemma}
 Thanks to this lemma, we can limit our analysis to the ``plain''
 discretization defined by \eqref{eq:VhkP:def}.
The construction and the analysis of the new stabilization terms will
consist in several steps:
\begin{enumerate}[itemsep=2mm, topsep=2mm, label* = {\emph{Step \arabic*.}}]
\item We show that the nonconforming virtual
  element space $\VhkP$ and the subspace $\tVhk$
  of $(H^1(\P))'$ spanned by the functionals yielding the degrees of
  freedom \TERM{D1}--\TERM{D2} are in a stable duality relation,
  i.e., 
  they satisfy an inf-sup condition of the form
  \eqref{eq:abstractinfsup}.
  \item We next show that, if we restrict ourselves to a suitably chosen subspace $\ringVhkP$ of $\VhkP$, a similar stable duality relation holds with the subspace spanned by the functionals corresponding to the sole  boundary degrees of freedom \TERM{D1}, which is isomorphic to the subspace $\Nk{k-1}(\bP)\subset H^{-\frac12}(\bP)$ of piecewise polynomials on the boundary mesh $\EdgesP$.

\item As $\ker \PinP{k} \subseteq \ringVhkP$, putting ourselves in the framework of
  Section~\ref{sec:dual:scalar:product}, we can then transfer the problem
  of building the bilinear form $\sP$ defined on the space $\VhkP$, to
  whose elements we do not have direct access, to the problem of
  building a $H^{-\frac12}(\bP)$ semi-inner product on the space $\Nk{k-1}(\bP)$.

\item We finally show that, on $\Nk{k-1}(\bP)$, the $H^{-\frac12}(\bP)$ semi inner product can be split
  as the sum of a global contribution acting on piecewise constants,
  and local contributions acting on average-free polynomials of degree
  $k-1$ on each edge.
 We postpone the treatment of the former to the
    next section and, for the latter, we prove that a suitably scaled
    $L^2$ inner product yields optimal estimates.  
\end{enumerate}
\subsubsection*{Mesh assumptions}
Before going into the details of the construction of the stabilization
term, we  present the precise assumptions
on the polygonal tessellations $\Th$.
As already stated, we assume that \TERM{G1}~is always satisfied.
First, we observe that we can write the
stabilization proposed in~\cite{AyusodeDios-Lipnikov-Manzini:2016} as
\begin{align*}
  \sP(\us,\vs) = \vv^T\uv,
\end{align*}
where $\uv$ and $\vv$ are the vectors collecting the degrees of
freedom \TERM{D1}--\TERM{D2} of the virtual element functions $\us$ and
$\vs$.
This bilinear form satisfies \eqref{eq:requirement:S}, provided that the family of polygonal meshes $\Th$
 satisfies the following additional shape regularity assumption:
  
\begin{itemize}[parsep=1mm,topsep=2mm]
\item[\SRtwo] \label{eq:SR.ii} there exist a positive constant
  $\gamma_1$ such that for all $\Th$, the distance between any two
  vertices of every element $\P\in\Th$ is greater than $\gamma_1\hP$.
\end{itemize}
  
\medskip
Assumption \SRtwo{} implies 
 that the
size of adjacent edges are comparable.
  It also implies that the number of edges in the boundary of a
  polygonal element is uniformly bounded from above and the minimum
  edge length cannot decrease faster than the mesh size $\hh$ during
  the refinement process for $\hh\to0$.
  So, mesh families where the number of edges can become arbitrarily
  high as $\hh\to0$ are not admissible.
Such an assumption is quite strong, and, to allow more freedom in the
choice of the mesh, weaker alternatives have been considered in the
literature.
  Assumption \SRtwo~ can be replaced by either one of assumptions
  \SRtwop~ and \SRtwopp~below.
  The former assumption allows elements to have a very large number of
  very small edges; the latter one to have very small edges 
  adjacent to large edges.
\begin{itemize}[parsep=1mm,topsep=2mm]
\item[\SRtwop] There exists a real positive constant $\gb$ such that
  for all meshes $\Th$ and every pair of adjacent
    edges $\E,\Ep \in \EdgesP$, $\P \in \Th$, it holds that
  \begin{align*}
    \frac{1}{\gb}\leq\frac{\hE}{\hEp}\leq\gb.
  \end{align*}
\item[\SRtwopp] There exists an integer positive constant $\NSs$ such
  that 
    for all meshes $\Th$, every $\P\in\Th$
  has at most $\NSs$ edges.
\end{itemize}

  To allow the meshes a greater flexibility, we combine $\SRone$ with
  the following assumption, which essentially requires
  that, for $\P \in \Th$,  a part of $\bP$ satisfies $\SRtwop$ and the remaining part
  satisfies $\SRtwopp$.
\begin{itemize}[parsep=1mm,topsep=2mm]
\item[\SRthree] There exist two constants $\gb > 0$ and $\NSs \in
  \mathbb{N}$ such that for all $\Th$, the edge set $\EdgesP$ of every
  polygon $\P\in\Th$ can be split as $\EdgesP=\EdgesPp\cup\EdgesPpp$,
  where $\EdgesPp$ and $\EdgesPpp$ are such that
  \begin{itemize}[itemsep=2mm]
  \item[\SRthreea] \label{eq:thmass1} the inequality
    \begin{align*}
      \frac{1}{\gb}\leq\frac{\hE}{\hEp}\leq\gb
    \end{align*}
    holds for any pair of adjacent edges  $\E,\Ep\in\EdgesP$ with
    $\E\in\EdgesPp$;
  \item[\SRthreeb]  $\EdgesPpp$ contains at most $\NSs$ edges.
  \end{itemize}
\end{itemize}
  Assumption $\SRthree$ allows for situations where a large number of
  small edges coexists with some large edges.
  We can think of families of meshes for which such an assumption is
  not satisfied, but they would be extremely pathological.

\subsubsection*{Step 1. Degrees of freedom: definition and stable duality}
Let $\P \in \Th$. We devote this section to verifying that the local
 nonconforming virtual element
space $\VhkP \subset \HONE(\P)$, defined by \eqref{eq:VhkP:def}, and
the space $\tVhk(\P)$ spanned in $(\HONE(\P))'$ by the functionals
yielding the degrees of freedom \TERM{D1}--\TERM{D2} fall into the
framework considered in Section \ref{sec:dual:scalar:product}.
To this aim we introduce the space of discontinuous piecewise
polynomials of degree $k-1$ that are defined on the elemental boundary
$\bP$,
\begin{align*}
  \Nk{k-1}(\bP) =
  \Big\{
  \lambda\in\LTWO(\bP):
  \restrict{\lambda}{\E}\in\PS{k-1}(\E),\,\forall\E\in\EdgesP \Big\}
  \subset\HS{-\frac12}(\bP),
\end{align*}
and we let $\tVhk(\P)\subset (\HONE(\P))'$ be defined as
\begin{align*}
  \tVhk(\P)=\gKstar\Nk{k-1}(\bP)\oplus\PS{k-2}(\P)\subset(\HONE(\P))',
\end{align*}
where $\gKstar:\HS{-\frac12}(\bP)\to(\HONE(\P))'$ is the adjoint of
the trace operator $\gK:\HONE(\P)\to\HS{\frac12}(\bP)$: for all
$\xi\in\HS{-\frac{1}{2}}(\bP)$
\begin{align}\langle \gKstar\xi, \vs \rangle = \langle \xi , \gK\vs\rangle, \qquad\forall\vs\in\HONE(\P).
  \label{eq:adjoint:trace:operator}
\end{align}
In fact, for any given virtual
elemental function $\vsh \in \VhkP$, the degrees of freedom \TERM{D1}--\TERM{D2}
of $\vsh$  stem from the
action of a basis of $\tVhk(\P)$.

\medskip
Let $\HONEzrbr$ denote the subspace of functions in $\HONE(\P)$ whose
integral on the polygonal boundary $\bP$ is zero:
\begin{align*}
  \HONEzrbr
  = \bigg\{
  \vs\in\HONE(\P)\,:\,\int_{\bP}\vs\dS=0,
  \bigg\}
\end{align*}
where,
  for the sake of notational simplicity and with some abuse of
  notation, we let the same symbol $\vs$ denote  both a function $\vs \in \HONE(\P)$ and 
its trace on $\bP$.
By duality with such a subspace of $\HONE(\P)$, we define the dual
seminorm $\SNORM{\,\cdot\,}{-1,\P}:(\HONE(\P))'\to\REAL^+$  as
\begin{align*}
  \forall\eta\in(\HONE(\P))':\quad
  \SNORM{\eta}{-1,\P}
  = \sup_{\vs\in\HONEzrbr}
  \frac{\bil{\eta}{\vs}}{\SNORM{\vs}{1,\P}}.
\end{align*}
We can prove the following proposition, which is a
  stronger version of the unisolvency property for the degrees of
  freedom.
  In fact, not only it implies  unisolvency,
but also that the space spanned by the functionals yielding the
degrees of freedom provides control, uniformly in $\hh$, on the $\HONE$ norm of the virtual
element functions.
Let
\begin{align*}
  \fint_{\bP}\vs = \frac{1}{\mbP}\int_{\bP}\vs\dS
  \end{align*}
denote the average of $\vs$ on $\bP$.
\begin{lemma}\label{lem:unisolvency}
  For all $\vs\in\VhkP$ it holds that
  \begin{align*}
    \sup_{\eta\in\tVhk(\P)} 
    \frac{\bil{\eta}{\vs}}{\sqrt{ \ABS{\bil{\eta}{1}}^2 + \SNORM{\eta}{-1,\P}^2}}
    \geq
    \sqrt{ \ABS{\fint_{\bP}\vs}^2 + \SNORM{\vs}{1,\P}^2 }.
  \end{align*}
\end{lemma}
\begin{proof}
  Let $\vs\in\VhkP$ and take $\eta_{\vs}\in\tVhk(\P)$ given by
  \begin{align*}
    \eta_{\vs} = 
    \gKstar\nabla\vs\cdot\nor_\P - \Delta\vs + \frac{1}{\mbP}\gKstar\fint_{\bP}\vs.
  \end{align*}
  According to~\eqref{eq:adjoint:trace:operator} and to the
  definition\ of $\eta_{\vs}$ given above, for every function
  $\ws\in\HONE(\P)$ we find that
  \begin{align}
    \bil{\eta_{\vs}}{\ws}
    &= \int_{\bP}\ws\nabla\vs\cdot\nor_\P\dS - \int_{\P}\ws\Delta\vs\dV + \left(\frac{1}{\mbP}\int_{\bP}\ws\dS\right)\,\left(\fint_{\bP}\vs\dS\right)\nonumber\\[0.5em]
    &= \int_{\bP}\nabla\vs\cdot\nabla\ws\dS + \left(\fint_{\bP}\ws\dS\right)\,\left(\fint_{\bP}\vs\dS\right),
    \label{eq:lemma:unisolvence:00}
  \end{align}
  which, for $w \in \HONEzrbr$, reduces to 
  \begin{align*}
     \bil{\eta_{\vs}}{\ws} = \int_{\bP}\nabla\vs\cdot\nabla\ws\dS.
  \end{align*}
   Then
  \begin{align}
    \SNORM{\eta_{\vs}}{-1,\P}
    = \sup_{\ws\in\HONEzrbr}\frac{\bil{\eta_{\vs}}{\ws}}{\SNORM{\ws}{1,\P}}
    = \sup_{\ws\in\HONEzrbr}\frac{\int_{\P}\nabla\vs\cdot\nabla\ws\dV}{\SNORM{\ws}{1,\P}}
    = \SNORM{\vs}{1,\P}.
    \label{eq:lemma:unisolvence:10}
  \end{align}
  Moreover, taking $\ws=1$ in~\eqref{eq:lemma:unisolvence:00} yields
  \begin{align}
    \bil{\eta_{\vs}}{1}
    = \fint_{\bP}\vs\dS.
    \label{eq:lemma:unisolvence:20}
  \end{align}
  Adding the square of~\eqref{eq:lemma:unisolvence:10}
  and~\eqref{eq:lemma:unisolvence:20} yields
  \begin{align}
    \ABS{\bil{\eta_{\vs}}{1}}^2 + \SNORM{\eta_{\vs}}{-1,\P}^2 =
    \ABS{\fint_{\bP}\vs\dS}^2 + \SNORM{\vs}{1,\P}^2,
    \label{eq:lemma:unisolvence:30}
  \end{align}
  and taking $\ws=\vs$ in~\eqref{eq:lemma:unisolvence:00} gives us the
  identity
  \begin{align*}
    \bil{\eta_{\vs}}{\vs}
    = \int_\P\ABS{\nabla\vs}^2\dV + \ABS{\fint_{\bP}\vs\dS}^2.
  \end{align*}
  Finally, we combine this identity
  with~\eqref{eq:lemma:unisolvence:30}
  to obtain
  \begin{align*}
    \sqrt{ \ABS{\fint_{\bP}\vs\dS}^2 + \SNORM{\vs}{1,\P}^2 }
    = \frac{\bil{\eta_{\vs}}{\vs}}{\sqrt{\ABS{\bil{\eta_{\vs}}{1}}^2 + \SNORM{\eta_{\vs}}{-1,\P}^2}}
    \leq 
    \sup_{\eta\in\tVhk(\P)}\frac{\bil{\eta}{\vs}}{\sqrt{\ABS{\bil{\eta}{1}}^2 + \SNORM{\eta}{-1,\P}^2}},
  \end{align*}
  which holds for every $\vs\in\Vhk(\P)$ and is the assertion of the
  lemma.
\end{proof}

\subsubsection*{Step 2. Reduction to the boundary}
If we restrict ourselves to a suitable subspace $\ringVhkP$ of $\VhkP$,
we have a stable duality result with the space spanned by the
functionals yielding the boundary degrees of freedom \TERM{D1}.
More precisely, consider 
the space of harmonic polynomials of degree at most $k$,
\begin{align*}
  \PolyHarm_k(\P) =
  \Big\{\,\qs\in\PS{k}(\P):\Delta\qs=0\,\Big\}\subset\PS{k}(\P),
\end{align*}
and the space of polynomials of degree at most $k$ orthogonal to all
polynomials in $\PolyHarm_k(\P)$,
\begin{align*}
  \Polyort_k(\P) = \bigg\{\,\ps\in\PS{k}(\P):
  \int_{\P}\ps\qs\dV = 0\,\,\forall\qs\in\PolyHarm_k(\P)\bigg\}
  \subset\PS{k}(\P).
\end{align*}
Let
\begin{align}
  \ringVhkP = \left\{\vs\in\VhkP:\,\int_{\P}\nabla\vs\cdot\nabla\qs\dV=0,\,\,
  \forall\qs\in\Polyort_k(\P)\right\}.
  \label{eq:poliort}
\end{align}
Remark that $\VhkP \cap \ker\PinP{k} \subset \ringVhkP$.
We have the following lemma.
\begin{lemma}\label{lem:4.1}
 For all $\vs\in\ringVhkP$ we have
   \begin{align*}
   \sup_{\eta\in\Nk{k-1}(\bP)} 
   \frac{\bil{\gKstar\eta}{\vs}}{\sqrt{\ABS{\bil{\gKstar\eta}{1}}^2 + \SNORM{\gKstar\eta}{-1,\P}^2}}
   \gtrsim
   \sqrt{ \ABS{\fint_{\bP}\vs}^2 + \SNORM{\vs}{1,\P}^2 }.
   \end{align*}
\end{lemma}

  To prove Lemma~\ref{lem:4.1}, we need two technical lemmas, which
  have been proven in~\cite{BeiraodaVeiga-Lovadina-Russo:2017} for the conforming virtual
  element method and are also true in the nonconforming case.
As the proof is the same, we omit it.
\begin{lemma}
  \label{lem:inverse}
 The following inverse
  inequality holds for all $\vs\in\VhkP$:
  \begin{align*}
    \NORM{ \Delta\vs }{0,\P} \lesssim \hP^{-1} \SNORM{\vs}{1,\P}.
  \end{align*}
\end{lemma}
\begin{lemma}\label{lem:pitilde}
 For all $\vs\in\VhkP$ there
  exists a polynomial function $\qst\in\Polyort_{k}(\P)$ such that
  \begin{align*}
    \Delta\qst = \Delta\vs
    \qquad{and}\qquad
    \SNORM{\qst}{1,\P} \lesssim \hP\NORM{\Delta\vs}{0,\P}.
  \end{align*}
\end{lemma}

We can now prove Lemma~\ref{lem:4.1}.

\medskip
\noindent
\textit{Proof of Lemma~\ref{lem:4.1}.}  Consider a function
$\vs\in\ringVhkP$.
Thanks to Lemmas \ref{lem:pitilde} and \ref{lem:inverse}, there exists
a polynomial $\qst\in\Polyort_{k}(\P)$ such that
\begin{align}\label{eq:defqtilde}
  \Delta\qst = \Delta\vs
  \qquad{and}\qquad
  \SNORM{\qst}{1,\P} \lesssim \hP\NORM{\Delta\vs}{0,\P} \lesssim  \SNORM{\vs}{1,\P}.
\end{align}
We take $\eta_{\vs}\in\Nk{k-1}(\bP)$ given by
\begin{align}
  \eta_{\vs} = \nabla(\vs-\qst)\cdot\nor_\P + \frac{1}{\mbP}  \fint_{\bP}\vs.
  \label{eq:lemma4.1:00}
\end{align}
For any $\ws \in H^1(\P)$ we have 
\begin{align}
\langle \gKstar \eta_v, w \rangle &= \int_{\bP} \nabla(\vs-\qst)\cdot\nor_\P  \ws\dS + \int_{\bP}\ws  \frac{1}{\mbP}  \fint_{\bP}\vs =\nonumber\\
&= \int_{\P }\nabla(\vs-\qst) \cdot \nabla\ws + \int_{\P} \Delta(\vs-\qst) \ws + \left(\fint \vs\right)\left(\fint \ws\right)\nonumber\\
&= \int_{\P }\nabla(\vs-\qst) \cdot \nabla\ws  + \left(\fint \vs\right)\left(\fint \ws\right).\label{48a}
\end{align}
As $\vs\in\ringVhkP$, we then have
\[
\langle \gKstar \eta_v, v \rangle = \int_{\P }\nabla(\vs-\qst) \cdot \nabla\vs + \left(\fint \vs\right)^2 = | v |_{1,\P}^2 + \left(\fint \vs\right)^2.
\]

Moreover, using the triangular inequality and the  bound on $\qst$ in \eqref{eq:defqtilde} we see that
\begin{align}
  \SNORM{ \gKstar\eta_{\vs} }{-1,\P}
  = \sup_{\ws\in\HONEzrbr} \frac{ \bil{\gKstar\eta_{\vs}}{\ws} } {\SNORM{\ws}{1,\P}}
  = \sup_{\ws\in\HONEzrbr} \frac{ \int_{\P}\nabla\ws\cdot\nabla(\vs - \qst)\dV }{\SNORM{\ws}{1,\P}}
  = \SNORM{\vs-\qst}{1,\P}
  \lesssim \SNORM{\vs}{1,\P},
  \label{eq:lemma4.1:10}
\end{align}
and, setting $w = 1$ in \eqref{48a},
\[
\langle \gKstar \eta_v, 1 \rangle = \fint \vs,
\]
which yields
\[
| \langle \gKstar \eta_v, 1 \rangle |^2+ | \eta_v |_{-1,\P}^2 \lesssim  | v |_{1,\P}^2 + \left(\fint \vs\right)^2.
\]
  
  Then, for every $\vs \in \ringVhkP$
\begin{align*}
\sqrt{  \SNORM{\vs}{1,\P}^2 +   \left(\fint \vs\right)^2} 
  &=        \frac{ \bil{\gKstar\eta_{\vs}}{\vs }}{ \sqrt{  \SNORM{\vs}{1,\P}^2 +   \left(\fint \vs\right)^2}  }
  \lesssim \frac{ \bil{\gKstar \eta_{\vs}}{\vs} }{ \sqrt{
  | \langle \gKstar \eta_v, 1 \rangle |^2+ | \eta_v |_{-1,\P}^2		
  		} }\\
  &\leq      \sup_{\eta\in\Nk{k-1}(\bP)} 
  \frac{ \bil{\gKstar \eta}{\vs} }{ \sqrt{
  		| \langle \gKstar \eta, 1 \rangle |^2+ | \eta
  		 |_{-1,\P}^2		
  	} }
\end{align*}
which is the assertion of the lemma.
\ENDPROOF

 As the kernel of $\PinP{k}$ is included in $\ringVhkP$, this will allow us to neglect the interior degrees of freedom \TERM{D2} when
 designing the stabilization bilinear form.

We conclude by remarking that we have $\dims(\ringVhkP) = \dims (\Nk{k-1}(\bP))$. Indeed, we have the splitting (cf. \cite{Brackx-Constales-Ronveaux-Serras:1989})
 \begin{align*}
 \PS{k}(\P)=\PolyHarm_k(\P) \oplus\ABS{\xv}^2\PS{k-2}(\P),
 \end{align*}
 so that $\dims(\Polyort_k(\P))=\dims(\PS{k-2}(\P))$, which implies that $\dims(\ringVhkP) \geq \dims(\VhkP) - \dims (\PS{k-2}(\P)) = \dims(\Nk{k-1}(\bP))$. The converse inequality is a consequence of Lemma \ref{lem:4.1}.


\

\subsubsection*{Step 3. Transfer to the dual}
We  use the approach of Section~\ref{sec:dual:scalar:product} with
these definitions:

\begin{itemize}[parsep=2mm,leftmargin=*,labelindent=\parindent]
\item $\Vs=(\HONE(\P))'$ and $\Vsp=\HONE(\P)$;
\item $\Ws = \gKstar(\Nk{k-1}(\bP))$ and $\Wss = \ringVhkP$;
\item $\Wbar=\gKstar(\PS{0}(\bP))$ and
  $\Wbars=\PS{0}(\P)$.
\end{itemize}
Remark that $(\HONE(\P))'$, which is
naturally a dual space, plays here the role of the primal space, and,
vice-versa, $\HONE(\P)$ plays the role of the dual space.

\medskip
The projector operators $\Pbar: (\HONE(\P))' \to \gKstar(\PS{0}(\bP))$
and $\Pbars:\HONE(\P)\to \PS{0}(\P)$ are, respectively, defined as
\begin{align*}
  \Pbar(\eta) = \ABS{\bP}^{-1} \bil{\eta}{1}\gKstar(1) 
  \quad\textrm{and}\quad
  \Pbars(\us)=\fint_{\bP}\us,
\end{align*}
(we recall that $\constants(\omega)$ is the restriction to $\omega$ of
the space of constant functions).

  Thanks to Lemma \ref{lem:pitilde}, assumptions \TERM{A1}--\TERM{A4}
  are satisfied, provided  we endow the spaces $\HONE(\P)$ and $(\HONE(\P))'$
  with the couple of dual norms (cf. \cite{Bertoluzza-Prada:2020})
\begin{align*}
  \tNORM{w}{1,\P} = \sqrt{ \ABS{\fint_{\bP} v\dS }^2 + \SNORM{v}{1,\P}^2}, \qquad 
  \tNORM{\zeta}{-1,\P} = \sqrt{ \ABS{\langle \zeta, 1 \rangle}^2 + \SNORM{\zeta}{-1,\P}^2}.
\end{align*}
  In order to build a bilinear form $\sP$
  satisfying~\eqref{eq:requirement:S} on the space $\Wss = \ringVhkP$,
  to whose elements we do not have access (not even to the boundary
  values), we can instead build a bilinear form $\tsP$ on the space
  $\Nk{k-1}(\bP)$ (whose element are known in closed
  form), satisfying
\begin{align}
  \tsP(\eta,\eta)\simeq\SNORM{ \gKstar\eta }{-1,\P}^2
  \quad\textrm{and}\quad
  \tsP(\eta,\mu)\lesssim\SNORM{ \gKstar\eta }{-1,\P} \SNORM{ \gKstar\mu }{-1,\P}.
  \label{eq:equivtilde}
\end{align}

Once $\tsP$ is built, we  consider:

\begin{itemize}[parsep=2mm,leftmargin=*,labelindent=\parindent]
\item the set $\basedofs=\big\{\zeta_i,\,i=1,\ldots,\Nbdofs\big\}$ of
  the piecewise polynomials of degrees up to $k-1$ used to evaluate
  the degrees of freedom \TERM{D1} associated with the elemental
  boundary $\bP$.
The set $\basedofs$ is a basis of the space $\Nk{k-1}(\bP)$;
\item the basis functions $\phi_i\in\ringVhkP$ associated with
the elements $\zeta_i$ of the basis $\basis$, verifying
\begin{align*}
  \int_{\bP}\phi_i\zeta_j\dS = \delta_{ij},\quad i,j=1,\ldots,\Nbdofs.
\end{align*}
\end{itemize}

The value of a degree of freedom of a function in $\ringVhkP$
corresponding to the unknown basis function $\phi_i$ coincides with
its $i$-th boundary degree of freedom in the complete local VEM space.
Then, we apply the framework of Section~\ref{sec:dual:scalar:product}.
We let $\matS=(s_{ij})$ denote the stiffness matrix associated to the
bilinear form $\ss=\tsP$, which  is
\begin{align*}
  \ss_{ij}=\ss(\zeta_i,\zeta_j)=\tsP(\zeta_j,\zeta_i)\, \qquad i,j = 1,\cdots,\Nbdofs,
\end{align*}
  We define matrix $\matr{\Sigma}=(\sigma_{ij})$ by $\matr{\Sigma
  }=\invS$, where $\invS$ is the reflexive generalized inverse of
  $\matS$ of Section~\ref{sec:dual:scalar:product}, and the bilinear
  form $\sP(\cdot,\cdot)$ by setting
\begin{align}\label{eq:sP:def}
  \sP(\phi_j,\phi_i)=\sigma_{ij}\,
  \qquad i,j = 1,\cdots,\Nbdofs.
\end{align}

Proposition \ref{prop:2.7} states that $\sP(\cdot,\cdot)$ satisfies
\eqref{eq:requirement:S}.
  We also have that 
\begin{align*}
  \sP(v,w) = \wv^T \matr{\Sigma} \vv = \wv^T \invS \vv,
\end{align*}
  where $\wv$ and $\vv$ are the vectors collecting the boundary degrees of
  freedom \TERM{D1} of two functions $\ws$ and $\vs$ in $\ringVhkP$.
  So,
we do not actually need to build the basis functions $\phi_i$, but we
define the action of the bilinear form $\sP$ directly on the vectors
of degrees of freedom.
This strategy allows us to reduce the construction of a bilinear form
$\sP(\cdot,\cdot)$ satisfying~\eqref{eq:requirement:S} to the
construction of a bilinear form $\tsP(\cdot,\cdot)$
satisfying~\eqref{eq:equivtilde}.

\subsubsection*{Step 4. Factoring out higher order polynomials}
We deal now with the construction of a bilinear form satisfying
\eqref{eq:equivtilde}.
To this end, we consider the seminorm
$\SNORM{\,\cdot\,}{-1/2,\bP}:\HS{-\frac12}(\bP)\to\REAL^+$ defined by
\begin{align}
  \SNORM{\eta}{-1/2,\bP} = \sup_{ \phi\in \Htracezr}\frac{\bil{\eta}{\vs}}{\SNORM{\phi}{1/2,\bP}},
  \label{eq:seminorm:minus:one:half:def}
\end{align}
where the functional space $ \Htracezr$ is defined as
\begin{align*}
  \Htracezr =
  \bigg\{
  \vs\in\HS{\frac12}(\bP)
  \,\,\textrm{such~that}\,\,
  \int_{\bP}\vs\dS = 0
  \bigg\}.
\end{align*}

Observe that, for all $\eta \in H^{-\frac12}(\bP)$, it holds that 
\begin{align*}
  \SNORM{ \gKstar\eta }{-1,\P}\simeq\SNORM{\eta}{-1/2,\P}.
\end{align*}
Then, we can
rewrite~\eqref{eq:equivtilde} as
\begin{align}
  \tsP(\eta,\eta)\simeq \SNORM{ \eta }{-1/2,\bP}^2
  \quad\textrm{and}\quad
  \tsP(\eta,\mu) \lesssim \SNORM{ \eta }{-1/2,\bP}\,\SNORM{\mu}{-1/2,\bP}.
  \label{eq:equivtilde2}
\end{align}

\medskip
We now split $\Nk{k-1}(\bP)$ as
\begin{align*}
  \Nk{k-1}(\bP) = \Nk{0}(\bP)\oplus\Nk{0}^\perp(\bP),
\end{align*}
where $\Nk{0}(\bP)$ is the space of functions that are
 constant on each edge of $\bP$, and
\begin{align*}
  \Nk{0}^\perp(\bP)
  = \bigg\{\eta\in\Nk{k-1}(\bP):\,\int_{e}\eta\dS = 0,\,\text{ for all edge } e \in\EdgesP\bigg\}
\end{align*}
is the space of piecewise polynomials of order up to $k-1$ with zero
average on each edge of $\bP$.
We start by providing a lower bound, which holds for all
$\eta\in\Nk{k-1}(\bP)$ under the very weak assumption \SRthree~ on the
edge partition $\EdgesP$ of $\bP$.
\begin{lemma}\label{lem:boundbelow}
  Assume that  \SRthree~holds.
  Then, for all $\eta\in\Nk{k-1}(\bP)$ with $\int_{\bP} \eta \dS = 0$ we have
  \begin{align}
    \snorm{\eta}{-1/2,\bP}^2 \gtrsim \sum_{\E\in\Eset^{\P}}\hE\int_{\E}\ABS{\eta }^2\dS.
    \label{eq:boundbelow}
  \end{align}
\end{lemma}
The proof of this Lemma is quite technical and we report it in
Appendix \ref{appendix:technical}.
  On $\Nk{0}^\perp(\bP)$ we can also prove an upper bound, as stated by the following lemma.
\begin{lemma}
  \label{lem:norm_meno}
 For all
  $\eta\in\LTWO(\bP)$ such that $\int_{\E}\eta\dS=0$ for all edges
  $\E\in \EdgesP$, it holds that
  \begin{align}
    \snorm{\eta}{-1/2,\bP}^2 \lesssim \sum_{\E\in\Eset^{\P}} \hE\int_{\E}\ABS{\eta}^2\dS.
    \label{eq:normmeno}
  \end{align}
\end{lemma}
\begin{proof}
  Consider $\eta\in \LTWO(\bP)$ such that its average on every edge
  $\E\in\EdgesP$ is zero.
  Let $\vs\in\HS{\frac12}(\bP)$ and denote its average on $\E$ by
  $\bar{\vs}^{\E}$.
  The Cauchy-Schwarz inequality and a Poincar\'e-like inequality yield
  \begin{align*}
    \int_{\E}\eta\vs\dS
    = \int_{\E}\eta\big(\vs-\bar{\vs}^{\E}\big)\dS
    \leq \NORM{\eta}{0,\E}\,\NORM{\vs-\bar{\vs}^{\E}}{0,\E}
    \leq \NORM{\eta}{0,\E}\,\hE^{\frac12}\SNORM{\vs}{1/2,\E},
  \end{align*}
  which holds for every edge $\E\in\EdgesP$.
  Using again the Cauchy-Schwarz inequality yields:
  \begin{align*}
    \bil{\eta}{\vs}
    &= \int_{\bP}\eta\vs\dS
    \leq \sum_{\E\in\EdgesP}\NORM{\eta}{0,\E}\,\hE^{\frac12}\SNORM{\vs}{1/2,\E}
    \leq \left(\sum_{\E\in\EdgesP}\hE\NORM{\eta}{0,\E}^2\right)^{\frac12}\,\left(\sum_{\E\in\EdgesP}\SNORM{\vs}{1/2,\E}^2\right)^{\frac12}
    \\[0.5em]
    &\lesssim \left(\sum_{\E\in\EdgesP}\hE\NORM{\eta}{0,\E}^2\right)^{\frac12}\,\SNORM{\vs}{1/2,\bP}.
  \end{align*}
  The assertion of the lemma follows by using this inequality in the
  definition \eqref{eq:seminorm:minus:one:half:def} of the seminorm $\SNORM{\eta}{-1/2,\bP}$.
\end{proof}
The following corollary is a straightforward consequence of
Lemmas~\ref{lem:boundbelow} and~\ref{lem:norm_meno}.
\begin{corollary}
  \label{cor:equivaveragefree}
  If assumption~\SRthree~holds, then, for all
  $\eta\in\Nk{0}^\perp(\bP)$ we have
  \begin{align}
    \snorm{\eta}{-1/2,\bP}^2\simeq\sum_{\E\in\Eset^{\P}}\hE\int_{\E}\ABS{\eta}^2\dS.
    \label{eq:eq:normmeno}
  \end{align}
\end{corollary}

Now, every $\eta\in\Nk{k-1}(\bP)$ can be split as $\eta=\eta^0+\eta^\perp$
with $\eta^0\in\Nk{0}(\bP)$ and $\eta^\perp\in\Nk{0}^\perp(\bP)$, and
we have
\begin{align*}
  \NORM{ \eta^\perp }{0,\E}
  =    \NORM{ \eta - \fint_{\E}\eta\dS }{0,\E}
  \leq \NORM{ \eta }{0,\E}.
\end{align*}
Then, using Lemma \ref{lem:boundbelow} and Lemma \ref{lem:norm_meno}, we can write
\begin{align*}
  \SNORM{\eta^ \perp}{-1/2,\bP}^2 
  &\lesssim \sum_e h_e \NORM{\eta^ \perp}{0,e}^2  =
  \sum_e h_e \NORM{\eta - \fint_e \eta  \pm \fint_{\bP}\eta}{0,e}^2 \\
  &\leq
  \sum_e h_e \left(\NORM{\eta - \fint_{\bP} \eta}{0,e}^2 + \NORM{\fint_e(\eta - \fint_{\bP} \eta)}{0,e}^2 \right)\\
  &\lesssim \sum_e h_e \NORM{\eta - \fint_{\bP}\eta}{0,e}^2 
  \lesssim \SNORM{\eta - \fint_{\bP}\eta}{-1/2,\bP}^2  = \SNORM{\eta}{-1/2,\bP}^2,
\end{align*}
and, by triangular inequality,
\begin{align*}
  \SNORM{\eta^0}{-1/2,\bP} \lesssim \SNORM{\eta}{-1/2,\bP}  + \SNORM{\eta^ \perp}{-1/2,\bP}^2 \lesssim \SNORM{\eta}{-1/2,\bP}.
\end{align*}
Corollary \ref{cor:equivaveragefree} yields the following
result.
\begin{corollary}
  \label{cor:splitting}
  If assumption~\SRthree~ holds, then, for
  $\eta\in\Nk{k-1}(\bP)$ split as $\eta=\eta^0+\eta^\perp$ with
  $\eta^0\in\Nk{0}(\bP)$ and $\eta^\perp\in\Nk{0}^\perp(\bP)$, we have
  \begin{align*}
    \snorm{ \eta }{-1/2,\bP}^2
    \simeq \SNORM{ \eta^0 }{-1/2,\bP}^2 + \SNORM{ \eta^\perp }{-1/2,\bP}^2
    \simeq \SNORM{ \eta^0 }{-1/2,\bP}^2 + \sum_{\E\in\Eset^{\P}} \hE\int_{\E}\ABS{\eta - \eta^0}^2\dS.
  \end{align*}
\end{corollary}
In view of Corollary~\ref{cor:splitting}, we  define the
stabilizing bilinear form
$\tsP(\cdot,\cdot):\Nk{k-1}(\bP)\times\Nk{k-1}(\bP)\to\REAL$ as
\begin{align}\label{eq:stP:def}
  \tsP(\eta,\zeta)
  = \ss^0(\eta^0,\zeta^0)
  + \sum_{\E\in\EdgesP} \hE\int_{\E}(\eta - \eta^0)(\zeta - \zeta^0)\dS, 
\end{align}
where $\ss^0(\cdot,\cdot):\Nk{0}(\bP)\times\Nk{0}(\bP)\to\REAL$ can be
any bilinear form satisfying
\begin{align}\label{eq:cond:s0}
  \ss^0(\eta^0,\zeta^0) \lesssim \SNORM{\eta^0}{-1/2,\bP}\SNORM{\zeta^0}{-1/2,\bP},
  \qquad\textrm{and}\quad
  \ss^0(\eta^0,\eta^0) \gtrsim \SNORM{\eta^0}{-1/2,\bP}^2
\end{align}
for all $\eta^0$,$\zeta^0\in\Nk{0}(\bP)$.
In the next section we will provide three different
strategies to build suitable bilinear forms $\ss^0(\cdot,\cdot)$.

\section{Stabilization for the lowest order nonconforming VEM}
\label{sec:lowstab}
We devote this section to the construction and analysis of several
possible bilinear forms
$\ss^0(\cdot,\cdot):\Nk{0}(\bP)\times\Nk{0}(\bP)\to\REAL$ satisfying
\eqref{eq:cond:s0}.
We consider three different strategies.
The first one is to define $\ss^0$ as a weighted $\LTWO$ inner
product, at the price of the loss of a logarithmic factor in the
stability estimate.
The second strategy is to resort to the use of a wavelet decomposition
of the space $\Nk{0}(\bP)$, and take advantage of the equivalent
expressions for the Sobolev norms of negative and/or fractionary order
that such bases allow.
Finally, in the spirit of Remark \ref{rem:twicedual},  we construct a second, explicitly known, discrete space, in a stable
  duality relation with $\Nk{0}(\bP)$.
  For this discrete space we explicitly define a bilinear form
  inducing the $\HS{\frac12}(\bP)$ seminorm, that we use to construct
  the bilinear form $\ss^0(\cdot,\cdot)$ by duality.

\subsection{A quasi optimal stabilization term}\label{sec:quasi:optimal}
We can define the bilinear form $\ss^0$ as
\begin{align}
  \ss^0(\eta,\mu) =\ssltwo(\eta,\mu)
  =\sum_{\E\in\EdgesP}\hE\int_{\E}
  \left(\eta - \fint_{\bP}\eta\right)
  \left(\mu  - \fint_{\bP}\mu \right)\dS,
  \label{eq:s0:def}
\end{align}
and we have the following lemma. 
\begin{lemma}
  If assumption ~\SRtwopp~ holds, then, setting
  \[\hhP = \min_{\E\in\EdgesP}\hE,\] for all $\eta\in\Nk{0}(\bP)$ 
 we have
  \begin{align*}
    (1+\log(\hP/\hhP))^{-1} \snorm{\eta}{-1/2,\bP}^2
    \lesssim  \ss^0(\eta,\eta)
    \lesssim  \SNORM{\eta}{-1/2,\bP}^2.
  \end{align*}
\end{lemma}

\begin{proof}
  Thanks to Lemma \ref{lem:boundbelow}, we only need to prove the
  first inequality.
  We consider an auxiliary quasi-uniform mesh $\gridauxP$ on $\bP$
  with mesh size $\hhP$ containing, as nodes,
  all the vertices of $\P$, and we let $\hgridauxP$ denote the dual
  mesh of $\gridauxP$, whose nodes are the midpoints of the elements
  of $\gridauxP$ (see Figure \ref{fig:gaux}).
  
  \begin{figure}\centering
    \includegraphics[width=10cm]{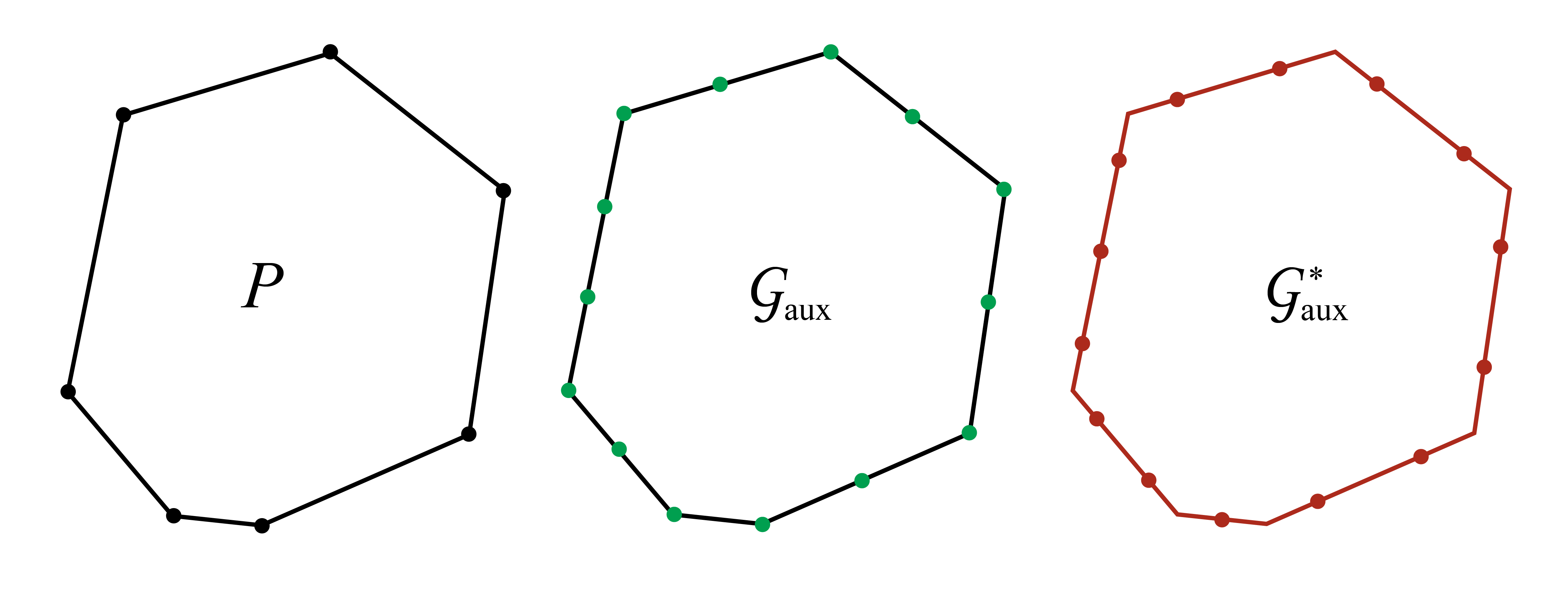} 
    \caption{A polygon $P$ (left), the auxiliary quasi-uniform grid
      $\gridauxP$ (center) and the dual grid $\hgridauxP$
      (right). Remark that the vertices of the polygonal element are
      not nodes of $\hgridauxP$.}\label{fig:gaux}
  \end{figure}
  Then, we let $\Naux(\bP)$ and $\tNaux(\bP)$ denote, respectively,
  the space of piecewise constant functions on the mesh $\gridauxP$,
  and the space of average free continuous piecewise
   linear functions
  on the mesh $\hgridauxP$.
  Observe that $\Nk{0}(\bP)\subseteq N_\text{aux}(\bP)$. We know
  (cf. \cite{Steinbach:2002}, see also Corollary \ref{cor:Steinbach} in the
  following) that for $\eta \in \Nk{0}(\bP)$ it holds that
  \begin{equation}\label{eq:infsupsteinbach}
    \SNORM{\eta}{-1/2,\bP}  \lesssim \sup_{v \in \tNaux(\bP)} \frac{\int_{\bP} \eta v \dS}{\SNORM{v}{1/2,\bP}}.
  \end{equation}
  Now, for $\eta \in \Nk{0}(\bP)$ and $\vs \in \tNaux(\bP)$, applying
  the Cauchy-Schwartz inequality twice, we obtain:
  \begin{align*}
    \int_{\bP}\eta\vs\dS
    &
    \lesssim \NORM{\vs}{\LINF(\bP)}\int_{\bP}\ABS{\eta}\dS
    =        \NORM{\vs}{\LINF(\bP)}\sum_{\E\in \EdgesP}\int_{\E}\ABS{\eta}\dS
    \\[0.5em]
    &
    \lesssim \NORM{\vs}{\LINF(\bP)}\sqrt{\sharp(\EdgesP)}\sqrt{\sum_{\E\in\EdgesP}\left(\int_{\E}\ABS{\eta}\dS\right)^2}
    \lesssim \NORM{\vs}{\LINF(\bP)}\sqrt{\sum_{\E\in \EdgesP}\hE\int_{\E}\ABS{\eta}\dS^2}.
  \end{align*}
  Therefore, plugging this last bound into \eqref{eq:infsupsteinbach},
   we obtain, 
  for every $\eta\in\Nk{0}(\bP)\subseteq\Naux(\bP)$, 
  \begin{align*}
    \SNORM{\eta}{-1/2,\bP}
    \lesssim \sup_{\vs\in \tNaux } \frac{\NORM{\vs}{\LINF(\bP)}}{\SNORM{\vs}{1/2,\bP}} \sqrt{\sum_{\E\in \EdgesP}\hE\int_{\E}\ABS{\eta}^2\dS}.
  \end{align*}
  It remains to bound the $\LINF(\bP)$ norm of $\vs$ in terms of its
  $\HS{\frac12}(\bP)$ seminorm.
 To this aim, we use an inverse inequality on the space of continuous
 piecewise
  linear polynomials 
 $\tNaux(\bP)$,
 cf. \cite[Lemma~3.2(i)]{Bertoluzza:2003:MathComp}, and obtain
 \begin{align}
   \SNORM{\eta}{-1/2,\bP} \lesssim \sqrt{1+\log(\hP/\hhP)} \left(
    \sum_{\E\in \EdgesP}\hE\int_{\E}\ABS{\eta}^2\dS \right)^{1/2}
    =\sqrt{1+\log(\hP/\hhP)} \left ( \sum_{\E\in
      \EdgesP}\hE\NORM{\eta}{0,\E}^2 \right)^{1/2}.
 \end{align}
Remarking that
\[ \SNORM{\eta}{-1/2,\bP} =  \SNORM{\eta - \fint_{\bP} \eta}{-1/2,\bP}\]
 concludes the proof.
\end{proof}
  If we now use the bilinear form $s^0$ defined above  in the design of the
  stabilization bilinear form for the space $\VhkP$, we have that
  \eqref{eq:requirement:S} is satisfied possibly with the loss of
  a logarithmic factor if assumption \TERM{G2a} is violated, as stated
  by the following corollary.
\begin{corollary}
Let assumption  ~\SRtwopp~hold, and let $\tsP$ be
  defined by \eqref{eq:stP:def} with $s^0$ defined by
  \eqref{eq:s0:def}. Then, the dual bilinear form $\sP: \VhkP \times
  \VhkP \to \mathbb{R}$ defined by \eqref{eq:sP:def} verifies, for all $\vs, \ws \in \VhkP\cap
    \KER\big(\PinP{k}\big)$
  \begin{align*}
    a^P(\vs,\vs) \lesssim \sP(\vs,\vs) \lesssim (1 + \log(\hP/\hhP)) a^P(\vs,\vs).
  \end{align*}
\end{corollary}

\

\renewcommand{\circle}{\widehat\Gamma}
\newcommand{\map}{\Theta}
\newcommand{\auxmesh}{\mathcal{G}_{\text{aux}}}

\subsection{An optimal stabilization based on a wavelet decomposition}\label{sec:wavelets}
In order to define a bilinear form $\ss^0(\cdot,\cdot)$ satisfying
\eqref{eq:equivtilde} on $\Nk{0}(\bP)$, we can 
exploit some known norm equivalences for the space
$\HS{-\frac12}(\bP)$, based on wavelet decompositions.
On a circle $\circle$ of unitary length, we consider the increasing sequence of spaces
$\{\Vs_j\}_{j=0}^\infty$, where $\Vs_j\subset\LTWO(\circle)$ is the
space of piecewise constant functions on the uniform grid on $\circle$
with mesh size $2^{-j}$.
Let $\{s^j_k\}_{k=0}^{2^j-1}$ denote the nodes of the corresponding
mesh, which we assume to be ordered counter-clockwise.
As $\Vs_j\subset\Vs_{j+1}$, for all level $j$ we  can
decompose ${\eta}_{j+1}\in\Vs_{j+1}$ as
${\eta}_{j+1}={\eta}_j+{\delta}_j$, with ${\eta}_j\in\Vs_j$ obtained
by applying a suitable oblique projector $\Pj$ to $\eta_{j+1}$.
For a given $M > 0$, this gives us a telescopic expansion of all
function in $\Vs_M$ as $\eta_M = \eta_0 + \sum_{j=0}^{M-1} \delta_j$,
and, passing to the limit as $\Ms$ goes to infinity, of all functions
$\eta$ in $\LTWO(\circle)$ as $\eta = \eta_0 +
\sum_{j=0}^{\infty}\delta_j$.
For $\eta\in L^2(\circle)$, we  can  introduce the
vector $\kappav_j(\eta)$ of length $2^j$, that uniquely determines
$Q_j \eta$:
\begin{align*}
\kappav_j(\eta) &:=\{\kappa_{jk}\}_{k=0}^{2^j-1}
\quad\textrm{with}\quad
\kappa_{jk}=2^{j/2}\int_{s^j_k}^{s^j_{k+1}} \Pj \eta\dS.
\end{align*}
As $\Pj$, whose precise definition is out of the scope or this paper, is a projector, for $\eta \in \Vs_j$ we have $\Pj \eta = \eta$ and hence, in such a case, $\kappa_{jk} = \int_{s^j_k}^{s^j_{k+1}} \eta$. 

  Let $\deltav_j(\eta)$ be the vector of coefficients of $\delta_j =
  (Q_{j+1} - Q_j)\eta$ with respect to a suitable basis for the space
  $W_j = (1-\Pj)V_{j+1}$, whose definition is also out of the scope of this paper (see \cite{Cohen-Daubechies-Feauveau:1992} for more details).
  Given $\kappav_{j+1}(\eta)$, we compute
  $\kappav_j(\eta):=\{\kappa_{jk}\}_{k=0}^{2^j-1} $ by subsampled convolution
  with a \emph{low-pass filter} $\lowpass$ of length $L+1$, which is
  strictly related with the projector $\Pj$, and
  $\deltav_j(\eta):=\{\delta_{jk}\}_{k=0}^{2^j-1}$ by subsampled  convolution with
  the \emph{band-pass filter} $\bandpass=[1,-1]$.
  More precisely, we have
\begin{align*}
  \kappa_{jk}   = \sum_{l=0}^L\dfrac{\sqrt{2}}{2} \lowpass(l)\,\kappa_{j+1, 2k+l} \quad \mbox{and}\quad
  \delta_{jk} = \sum_{l=0}^1\dfrac{\sqrt{2}}{2} \bandpass(l)\,\kappa_{j+1, 2k+l} = \dfrac{\sqrt{2}}{2} \left( \kappa_{j+1,2k} - \kappa_{j+1,2k+1} \right).
\end{align*}
In the above computations,  the function $\eta$ is
considered as periodic, so that
we  extend 
the vector $\kappav_{j+1}(\eta)$ as
$\kappa_{j+1,2^{j+1}+k}=\kappa_{j+1,k}$, $k\ge0$, when $2k+l>2^{j+1}-1$. 
For suitable choices of the low pass filter $\lowpass$, the following
norm equivalence holds for all $\eta \in H^{-\frac12}(\circle)$ (see
\cite{Dahmen:1996})
\begin{align*}
  \SNORM{\eta}{-1/2,\circle}^2 \simeq
  \sum_{j=0}^\infty 2^{-j} \NORM{\deltav_j(\eta)}{2}^2,
\end{align*}
where $\NORM{\cdot}{2}$ denotes the Euclidean norm. There are several possible choices for the oblique projector $\Pj$ and the relative low pass filter $\lowpass$ (see \cite{Cohen-Daubechies-Feauveau:1992}). 
In our experiments, we choose the so called \emph{(2,2)-biorthogonal
wavelet}, cf.~\cite{Cohen-Daubechies-Feauveau:1992}, for which the low pass filter $\lowpass$ is
\begin{align*}
  \lowpass = \dfrac{\sqrt{2}}{2}\,
  \bigg[\,\,
    \frac{3}{128},      \,\,-\frac{3}{128},
    \,\,-\frac{11}{64}, \,\,\frac{11}{64},
    \,\,1,              \,\,1,
    \,\,\frac{11}{64},  \,\,-\frac{11}{64},
    \,\,-\frac{3}{128}, \,\,\frac{3}{128}
    \,\,\bigg].
\end{align*}

\

In order to exploit such a norm equivalence, we embed the grid on
$\bP$, whose elements are the edges of $\P$, in a quasi uniform mesh
$\auxmesh$ with $2^M$ elements, where $M$ is the smallest integer such
that $M > \log_2(\sum_{e\in \EdgesP}h_e/(\min_{e\in \EdgesP} h_e))$.

We then consider a continuous piecewise linear
 (in the curvilinear abscissas) 
mapping
$\map:\circle\to\bP$, such that the nodes of the uniform dyadic grid
of $\circle$ with $2^M$ elements are mapped to the nodes of $\auxmesh$.
A change of variable argument yields the scaling relation
\begin{align*}
  \SNORM{\eta}{-1/2,\bP}
  \simeq \hP\SNORM{\eta\circ\map}{-1/2,\circle}.
\end{align*}
Then, for $\eta,\mu \in \Nk{0}(\bP)$, we define
\begin{align}\label{defs0wav}
  \ss^0 (\eta,\mu) = \sswav(\eta,\mu) =  \hP^2 \sum_{j=0}^{M} 2^{-j} \deltav_j(\eta\circ\Theta)^T \deltav_j(\mu\circ\Theta)^T.
\end{align}
The vectors  $\deltav_j(\eta\circ\Theta)$ and
$\deltav_j(\mu\circ\Theta)$ can be computed efficiently by a
\emph{fast wavelet transform}.
We have the following corollary.
\begin{corollary}
  Let  assumption  \SRthree~hold, and let $\tsP$ be
  defined by \eqref{eq:stP:def} with $s^0$ defined by
  \eqref{defs0wav}.
  Then, the dual bilinear form $\sP: \VhkP \times \VhkP \to
  \mathbb{R}$ defined by \eqref{eq:sP:def} verifies,
 for all $\vs, \ws \in \VhkP\cap \KER\big(\PinP{k}\big)$, 
  \begin{align*}
    a^P(\vs,\vs) \lesssim \sP(\vs,\vs) \lesssim  a^P(\vs,\vs).
  \end{align*}
\end{corollary}

\

\subsection{An optimal stabilization based on a known dual space}
  In the spirit of Remark \ref{rem:twicedual}, we can look at
  $\Nk{0}(\bP)$ as the stable dual space of a third, explicitly known
  space $\tNk{0}(\bP)\subset\HS{\frac12}(\bP)$.
  Then, we can construct an optimal stabilizing form $s^0$ on
  $\Nk{0}(\bP)$ if we are able to construct a bilinear form on
  $\tNk{0}(\bP)$ that is spectrally equivalent to the
  $\HS{\frac12}(\bP)$ semi-inner product.
A key ingredient in the construction is an oblique projector onto the
continuous piecewise linears, studied by Steinbach in \cite{Steinbach:2002}.

Let $\G$ and $\hG$ denote, respectively, a grid on $\bP$, and the dual grid, whose nodes are the
midpoints of the elements of $\G$.
We let $\Const$ and $\Lin$ denote the space of piecewise constant
functions on $\G$ and space of continuous linear functions on $\hG$.
We can define the projector $\widetilde{\Qs}:\LTWO(\bP)\to\Lin$ as
\begin{align*}
  \bil{\widetilde{\Qs}\vs-\vs}{\wsh} = 0
  \quad\forall\wsh\in\Const.
\end{align*}
The following theorem holds.
\begin{theorem}\label{thm:steinbach}
  Assume that there exists a constant $\cs'\geq 1$ such that for any
  two adjacent intervals $\E$ and $\Ep$ in $\G$ it holds that
  \begin{align*}
    \frac {\hE}{\hEp} \leq \cs'.
  \end{align*}
  Then, if $\cs'< \constSteinbach \approx 3.672688104237926$, the
  projector $\widetilde{\Qs}_{\hh}$ is bounded in $\HS{\frac12}(\bP)$:
  \begin{align*}
    \NORM{ \widetilde{\Qs}\vs}{-1/2,\bP} \lesssim
    \NORM{\vs}{1/2,\bP},
  \end{align*}
  the implicit constant in the inequality depending on $\cs'$.
\end{theorem}
The proof of Theorem \ref{thm:steinbach}, which also yields the value of the constant $c_0$,
 is the same as the proof of
 the analogous result in
\cite[Theorem~4.3 and~Section 5]{Steinbach:2002}, where the roles of the
grids $\mathcal{G}$ and $\hG$ are, however, switched. Nevertheless,
the arguments therein apply unchanged to the present case, though resulting
 in a different value of $\cs_0$, as detailed in
Appendix \ref{appendix:Steinbach}.
Using $\widetilde{\Qs}_{\hh}$ as a Fortin projector, we find the
result stated in the following corollary.
\begin{corollary}
  \label{cor:Steinbach}
  Under the assumptions of Theorem \ref{thm:steinbach}, the following
  inf-sup conditions hold:
  \begin{align*}
    \inf_{\ws\in\Const}\sup_{\vs\in\Lin  }\frac{\int_{\bP}\ws\vs}{\NORM{\vs}{1/2,\bP}\,\NORM{\ws}{-1/2,\bP}}\gtrsim 1,
    \qquad
    \inf_{\vs\in\Lin  }\sup_{\ws\in\Const}\frac{\int_{\bP}\ws\vs}{\NORM{\vs}{1/2,\bP}\,\NORM{\ws}{-1/2,\bP}}\gtrsim 1.
  \end{align*}
\end{corollary}  

\

Assume now that the tessellation satisfies Assumption \SRtwop~with a
constant $\gamma_2<\constSteinbach$, where $\constSteinbach$ is given
in Theorem \ref{thm:steinbach}.
 We let  $\tNk{0}(\bP)$ denote
the space of piecewise  linear (in the arclength ascissa on the boundary) functions 
 on the grid $\hG$ whose
nodes are the midpoints of the edges of $\P$.
Corollary \ref{cor:Steinbach} implies the inf-sup
  condition
\begin{align*}
  \inf_{\eta\in\Nk{0}(\bP)}\sup_{\phi\in\tNk{0}(\bP)}
  \frac{\int_{\bP}\lambda\phi}{\NORM{\lambda}{-1/2,\bP}\,\NORM{\phi}{1/2,\bP}}
  \gtrsim 1.
\end{align*}
Once again, we resort to the duality technique presented in
Section~\ref{sec:dual:scalar:product}
by setting, this time,
\begin{itemize}[parsep=2mm,leftmargin=*,labelindent=\parindent]
\item  $\Vs=\HS{\frac12}(\bP)$, $\Vsp=\HS{-\frac12}(\bP)$,
\item $\Ws=\widetilde{N}_0(\bP)$, $\Wss=\Nk{0}(\bP)$,
\item $\Wbar=\widehat{\Ws}^*= \mathbb{P}_0(\bP)$.
\end{itemize}

We take the projection $\Pbar$ as the $\LTWO(\bP)$-orthogonal
projection onto the constants.
We need to define a bilinear form $\sPlin$ on the known space
$\tNk{0}(\bP)$ satisfying
\begin{align}
  \sPlin(\phi,\phi) \simeq \SNORM{\phi}{1/2,\bP}^2,
  \qquad
  \sPlin(\phi,\psi) \lesssim \snorm{\phi}{1/2,\bP}\,\SNORM{\psi}{1/2,\bP}.
  \label{eq:requirement:Stilde}
\end{align}
The problem of defining bilinear forms
satisfying~\eqref{eq:requirement:Stilde} on the space of  continuous piecewise
linear functions has been addressed in other numerical frameworks,
such as, for example, the one of domain decomposition methods (see,
for instance, \cite{Bramble-Pasciak-Schatz:1986}) or the 
stabilization of the conforming VEM
method (see \cite{BeiraodaVeiga-Lovadina-Russo:2017}).
We will present two possible options at the end of this Section.

Assume now to have such a bilinear form. Let $\basedofs_0 = \{\zeta^e_0,\ e\in \EdgesP \}$, with $\zeta_0^e$ denoting the characteristic function of the edge $e$, denote the natural basis for $\Nk{0}(\bP)$, and 
$\Baselin_0=\{ \tphi_{\Ei},
\ i=1,\cdots,\#\EdgesP \}$ be the basis for $\tNk{0}(\bP)$, dual to
$\basedofs_0$, that is, $\tphi_{\Ei}$ is the
unique piecewise linear function on the dual grid $\hG$ such that
\begin{align*}
  \int_{\bP}\tphi_{\Ei}\zeta^{\Ej}_{0}\dS = \int_{\Ej}\tphi_{\Ei}\dS =
  \delta_{i,j}, \qquad j = 1,\cdots,\#\EdgesP.
\end{align*}
We define the relative stiffness matrix
\begin{align*}
  \widetilde{\matrSigma} = (\widetilde{\sigma}_{i,j}),
  \qquad
  \widetilde{\sigma}_{i,j} = \sPlin(\tphi_{\Ei},\tphi_{\Ej}).
\end{align*}
Let $\matS_0 = \invtSigma = (s^0_{i,j})$. We can define
$s^0:\Nk{0}(\bP) \times \Nk{0}(\bP) \to \REAL$ by
setting
\begin{equation}\label{eq:s0:def3}
  s^0(\zeta_{\Ei},\zeta_{\Ej})
  = s^0_{i,j}, \qquad i,j
  = 1,\cdots,\#\EdgesP.
\end{equation}
Proposition \ref{prop:2.7} yields, for all $\eta \in \Nk{0}(\bP)$ 
\begin{align*}
  s^0(\eta,\eta) \simeq | \eta |_{-1/2,\bP}^2.
\end{align*}

We have the following corollary.
\begin{corollary}
  Let  assumption  \SRtwop~hold with $\gamma_2 < c_0$, $c_0$ given by Theorem \ref{thm:steinbach}, and let $\tsP$ be
  defined by \eqref{eq:stP:def} with $s^0$ defined by
  \eqref{eq:s0:def3}.
  Then, the dual bilinear form $\sP: \VhkP \times \VhkP \to
  \mathbb{R}$ defined by \eqref{eq:sP:def} verifies, for all $\vs, \ws
  \in \VhkP\cap \KER\big(\PinP{k}\big)$
  \begin{align*}
    a^P(\vs,\vs) \lesssim \sP(\vs,\vs) \lesssim  a^P(\vs,\vs).
  \end{align*}
\end{corollary}
As  observed in Remark
\ref{rem:twicedual}, we find that $\invS_0=\invinvtSigma=\tSigma$, and,
consequently, the bilinear form $\sP$ mentioned in the above
corollary, takes the form
\begin{align*}
  \sP(\us,\vs)
  = \sum_{i,j=1}^{\#\EdgesP}\widetilde\sigma_{i,j}
  \bil{\zeta^{\Ei}_{0}}{\us}\,\bil{\zeta^{\Ej}_{0}}{\vs}  + \sum_{\E}\hE^{-1}\int_{\E}\portedge\us\,\portedge\vs\dS.
\end{align*}

\begin{remark}
Analogously to what we proposed in Section \ref{sec:wavelets},  if the tessellation does not  satisfy the gradedness Assumption \SRtwop~with
  $\gamma_2<\constSteinbach$, it is always possible to embed the mesh
  induced on $\bP$ by the vertexes of $\P$ in a finer mesh satisfying
  the assumptions of Theorem \ref{thm:steinbach}.
Then, we can define
  $\tsP$ on the space of piecewise constants on such a finer grid and
  then restrict it to $\Nk{0}(\bP)$.
In such a case, the
  finer mesh is only needed for the computation of low order component
  of the stabilization bilinear form.
\end{remark}

We conclude this section by recalling two possibilities for the definition of the bilinear form $\sPlin$.
We let
\begin{align*}
  \lb(u,w) = \int_{\bP} u' v' \dS, \qquad 
\end{align*}
denote the bilinear form relative to the Laplace-Beltrami operator on
$\bP$, with  $\tmatR=(r_{i,j})$,
$\rs_{i,j}=\lb(\tphi_{\Ei},\tphi_{\Ej})$, being
the stiffness matrix relative to its Galerkin discretization, and $\tmatM = (m_{i,j})$, $m_{i,j} = \int_{\bP} \tphi_{\Ei}\tphi_{\Ej}$  the corresponding mass matrix.
The first possibility is to define $\sPlin$ as the scaled
Laplace-Beltrami operator, which corresponds to setting
\begin{align*}
  \widetilde{\matrSigma} = h_P \tmatR.
\end{align*}
This bilinear form has been proposed in \cite{BeiraodaVeiga-Lovadina-Russo:2017} as a stabilization for the conforming virtual element methods, where the traces of virtual functions on the boundary of the elements are continuous piecewise linear polynomials.	We let $\sslb:\Nk{0}(\bP) \times \Nk{0}(\bP)  \to \REAL$ denote the bilinear form resulting from such a choice.
	
	\medskip

The second  possibility, originally proposed in the domain decomposition framework (cf. \cite{Bramble-Pasciak-Schatz:1986} and \cite[page 1110]{Bjorstad:1986:IMS}),  is to define
$\sPlin$ as the square root of the Laplace-Beltrami operator, which
correspond to setting
\begin{align*}
 \widetilde{\matrSigma} =\tmatM^{1/2}(\tmatM^{-1/2} \tmatR \tmatM^{-1/2})^{1/2}\tmatM^{1/2}.
\end{align*}
Remark that if the grid $\hG$ is quasi-uniform, then, using mass lumping, the contributes of the mass matrix $\tmatM$ cancel out and the above definition reduces to $ \widetilde{\matrSigma}=  \tmatR ^{1/2}$.  We let $\ssrlb:\Nk{0}(\bP) \times \Nk{0}(\bP)  \to \REAL$ denote the bilinear form resulting from such a choice.

\section{Implementation and Numerical Experiments}
\label{sec:numerical} 

\subsection{Construction of the stabilizing bilinear form} 
Before presenting the numerical tests, we give some detail on the
algebraic realization of the bilinear form $\sP(\cdot,\cdot)$.

For convenience of exposition, we introduce a local numbering of the
elemental edges, e.g., we denote the $\ell$-th edge in $\EdgesP$ by
$\E_{\ell}$ with subindex $\ell$ running from $1$ to
$N=\sharp(\EdgesP)$, the cardinality of the edge set $\EdgesP$. We
select the boundary degrees of freedom \TERM{D1} on each edge of $\bP$
so that $\basedofs$ contains the subset
$\basedofs_0=\big\{\zeta_0^{\E_\ell},\,\ell=1,\cdots,N\}$, forming a
basis of $\Nk{0}(\bP)$, with the remaining basis functions having zero
average on the elemental edges. The basis $\basedofs$ will then have
the form \(\basedofs = \basedofs_0 \cup \{\zeta^{\E_\ell}_i,
\ \ell=1,\cdots,N, \ i = 1,\cdots,k-1 \}\), where, for $\ell =
1,\cdots,N$, the set $\{\zeta^{\E_\ell}_i, \ \ i = 1,\cdots,k-1 \}$ is
a basis for the space of average free polynomials of order at most
$k-1$ on $\E_\ell$.

Let $\matS_0=(\ss_{\ell,\ell'})$ be the matrix having coefficients
$\ss_{\ell,\ell'}=\ss^0(\zeta^{\E_\ell}_0,\zeta^{\E_{\ell'}}_0)$.
This choice of the basis implies that matrix $\matS$ is block diagonal
up to a permutation of its rows and columns, which corresponds to a
suitable renumbering of the edge basis functions, and takes the form
\begin{align*}
  \matS = \left(
  \begin{array}{cccc}
    \matS_0 & \matr 0           & \cdots & \matr 0 \\
    \matr 0 & \hh_{\E_1} \matM_1 & \cdots & \matr 0 \\
    \vdots  & \vdots            & \ddots & \vdots \\ 
    \matr 0 & \matr 0           & \cdots & \hh_{\E_N} \matM_N \\
  \end{array}
  \right),
\end{align*}
where $\matM_{\ell}$ is the mass matrix for the space of average free
polynomials of order at most $k-1$ on the edge $\E_{\ell}$,
which is given by the edge integral
\begin{align*}
  \restrict{\matM_{\ell}}{i,j} = \int_{\E_{\ell}}\zeta^{\E_{\ell}}_{i}\zeta^{\E_{\ell}}_{j}\dS
  \quad
  i,j=1,\ldots,k-1.
\end{align*}
As the matrices $\matM_{\ell}$ are nonsingular, it is not difficult to
check that the reflexive generalized inverse $\invS$ of $\matS$, as
defined in Section \ref{sec:dual:scalar:product}, has a block diagonal
structure and it is given by
\begin{align*}
  \invS
  = \left(
  \begin{array}{cccc}
    \invS_0  & \matr 0                   & \cdots  & \matr 0 \\
    \matr 0  & \hh_{\E_1}^{-1} \matM^{-1}_1 & \cdots  & \matr 0 \\
    \vdots   & \vdots                    & \ddots  & \vdots \\ 
    \matr 0  & \matr 0                   & \cdots  & \hh_{\E_N}^{-1} \matM^{-1}_N 
  \end{array}
  \right).
\end{align*}

\begin{figure}
  \centering
  \subfigure{\includegraphics[width=0.3\textwidth]{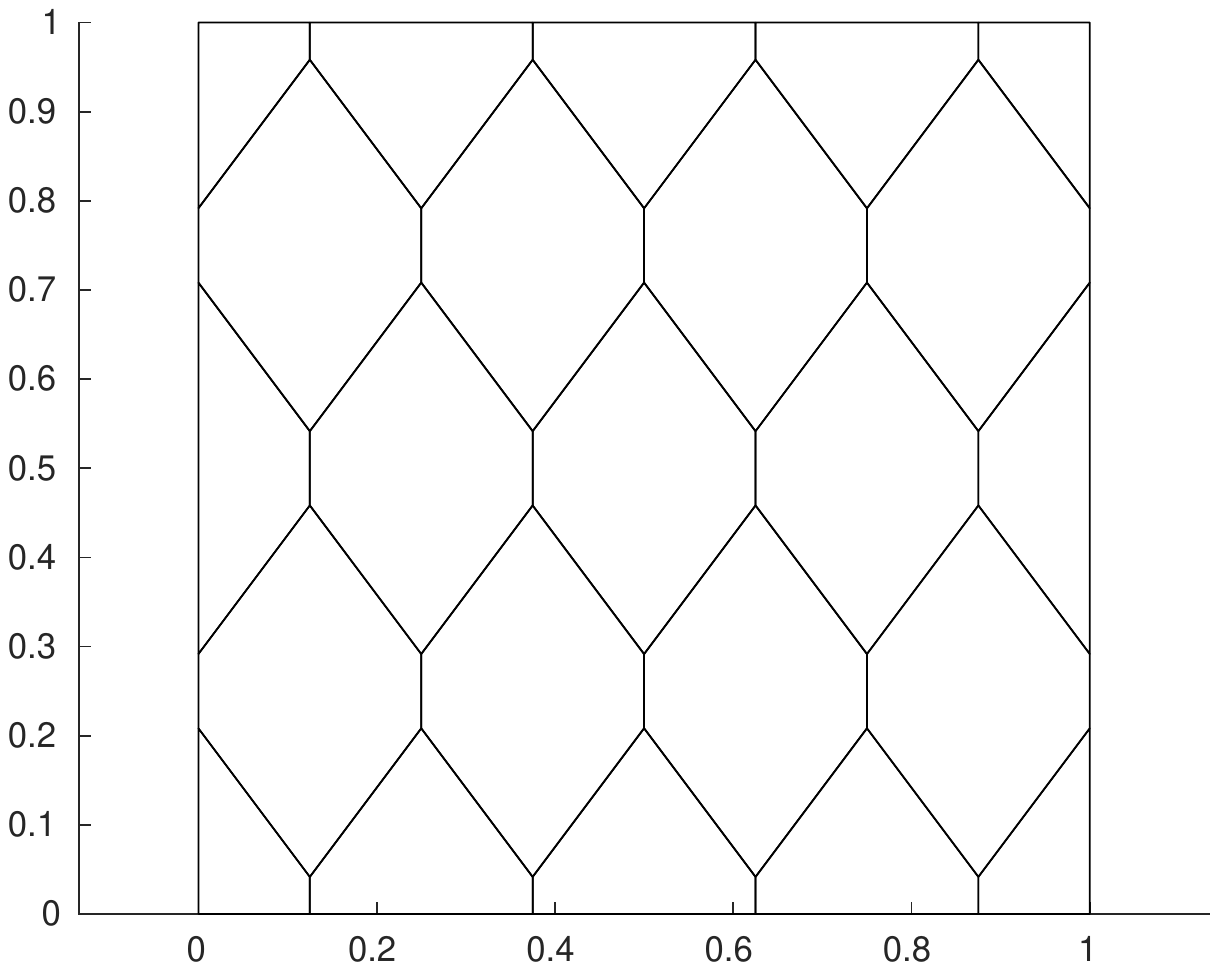}}\quad 
  \subfigure{\includegraphics[width=0.3\textwidth]{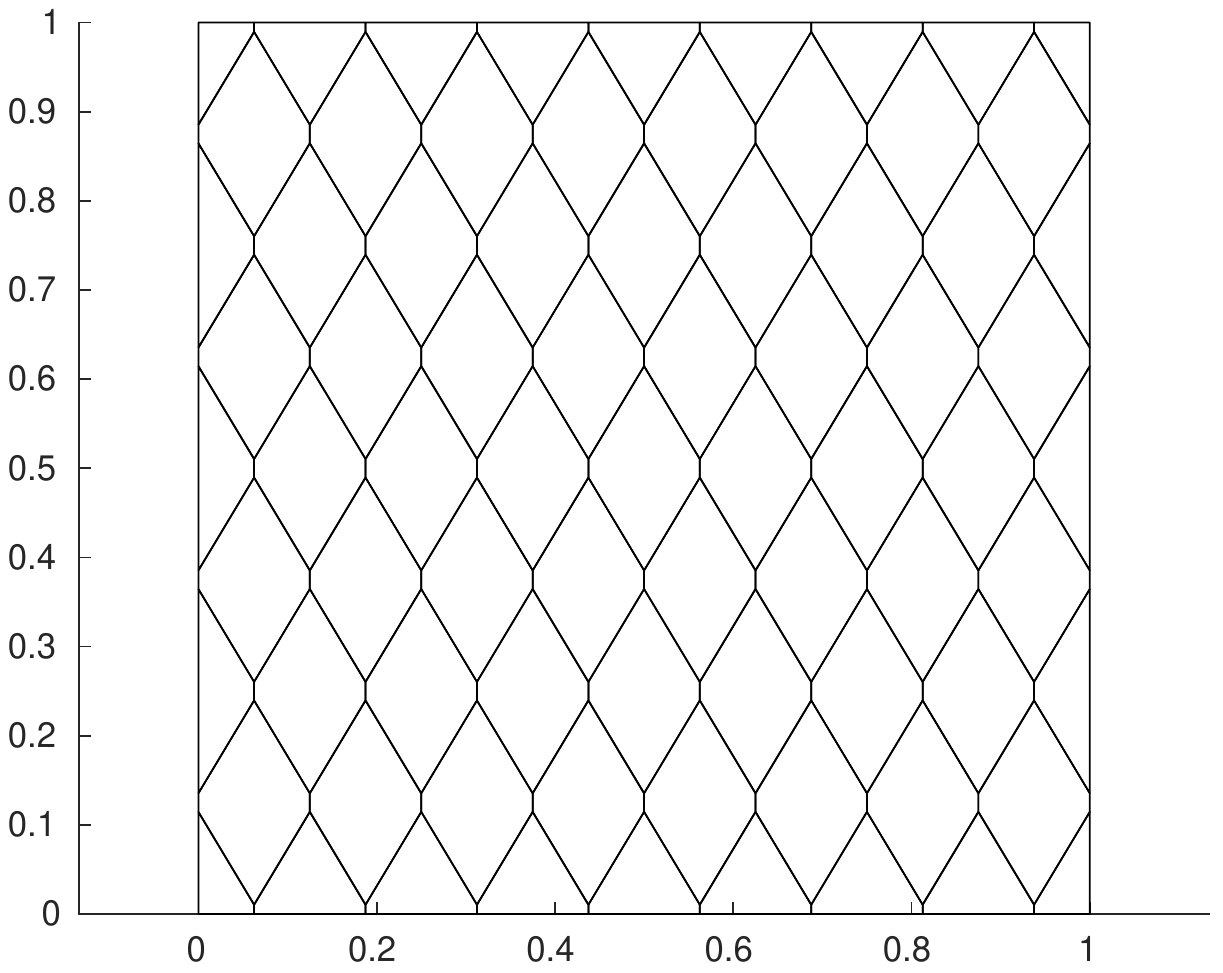}}\quad 
  \subfigure{\includegraphics[width=0.3\textwidth]{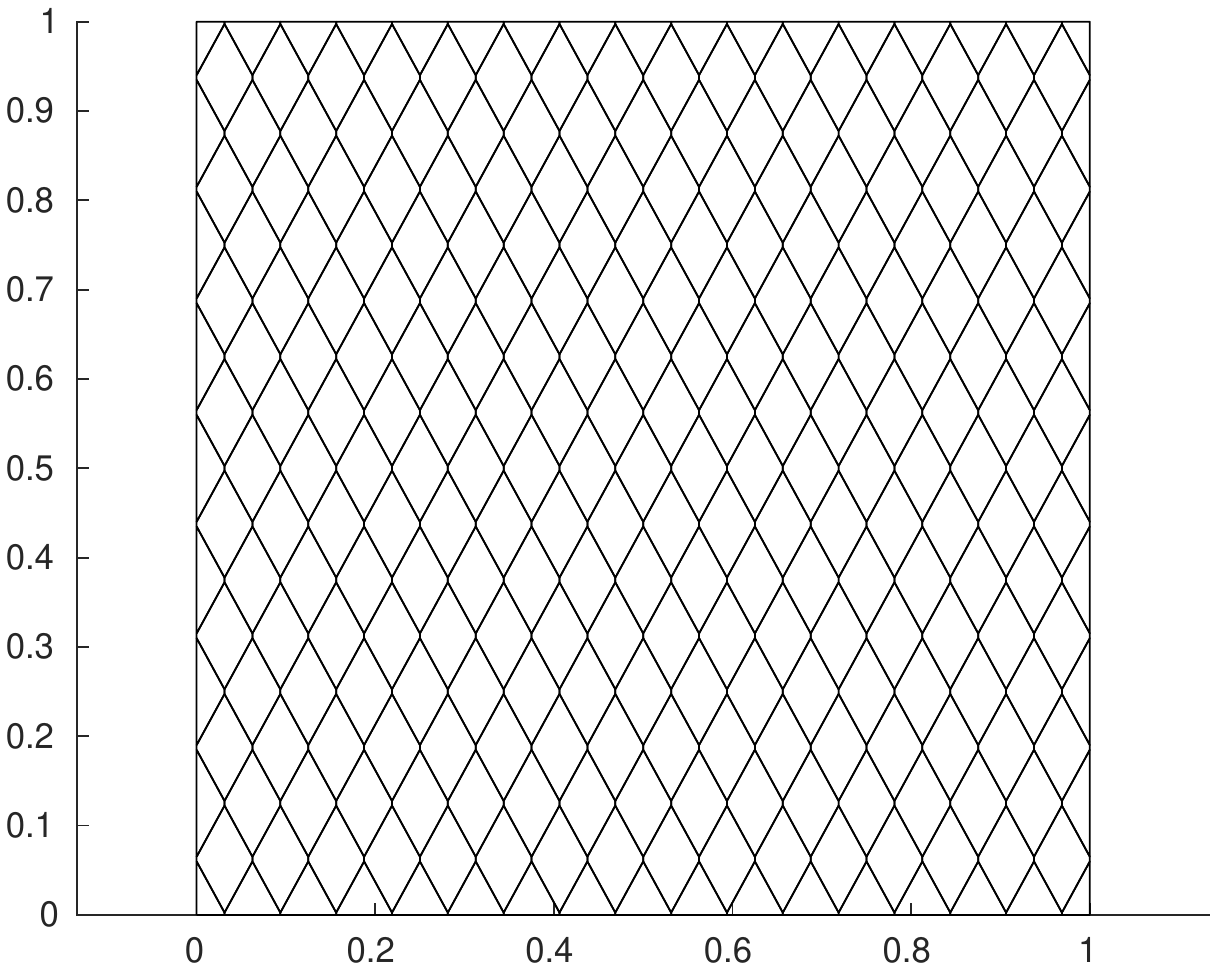}} 
  \caption{Test Case 1: Mesh family \MeshOne~ (exagonal elements with progressively collapsing edges)}
  \label{fig:s-hexa}
\end{figure}

\begin{figure}
  \centering
  \subfigure{\includegraphics[width=0.3\textwidth]{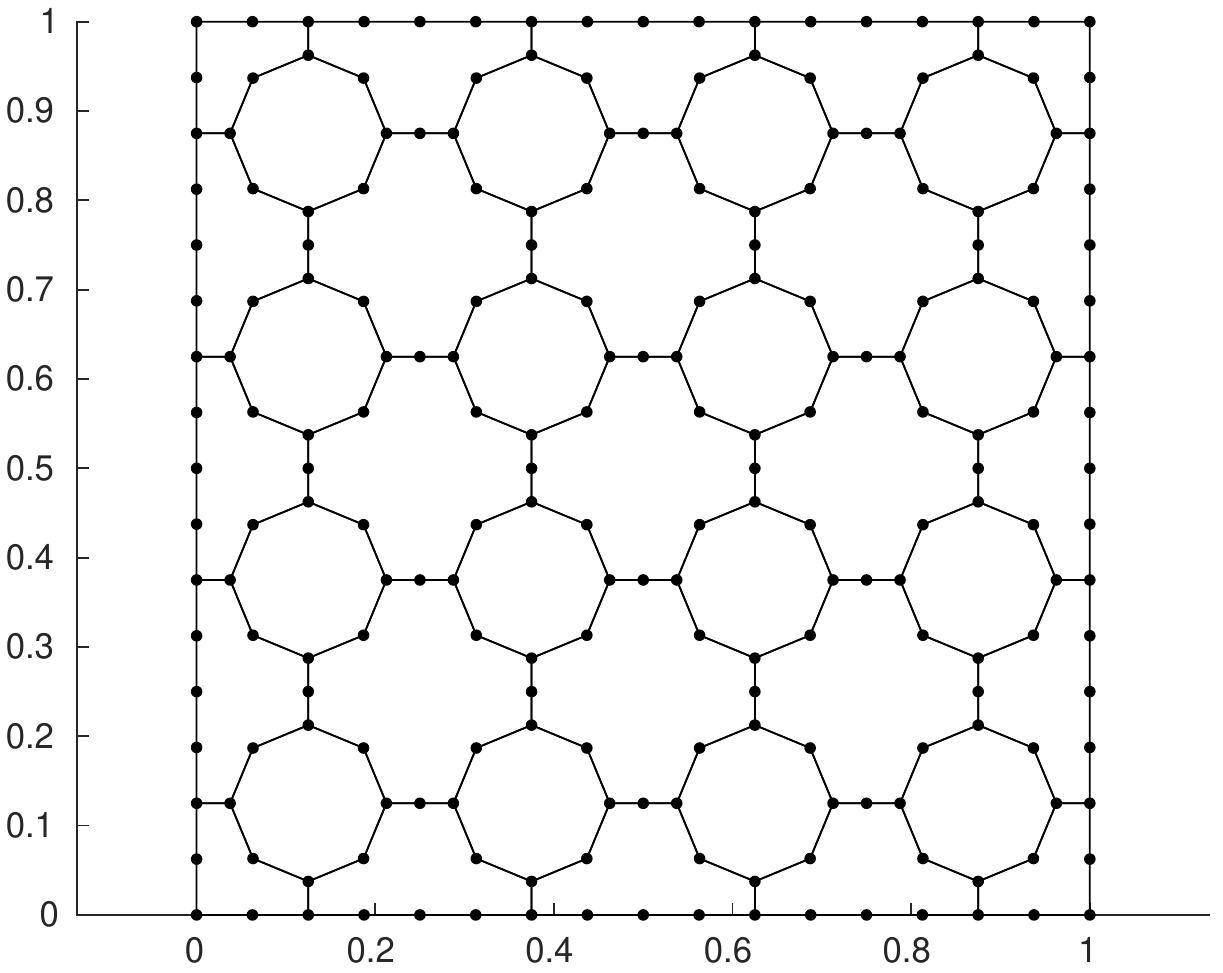}}\quad 
  \subfigure{\includegraphics[width=0.3\textwidth]{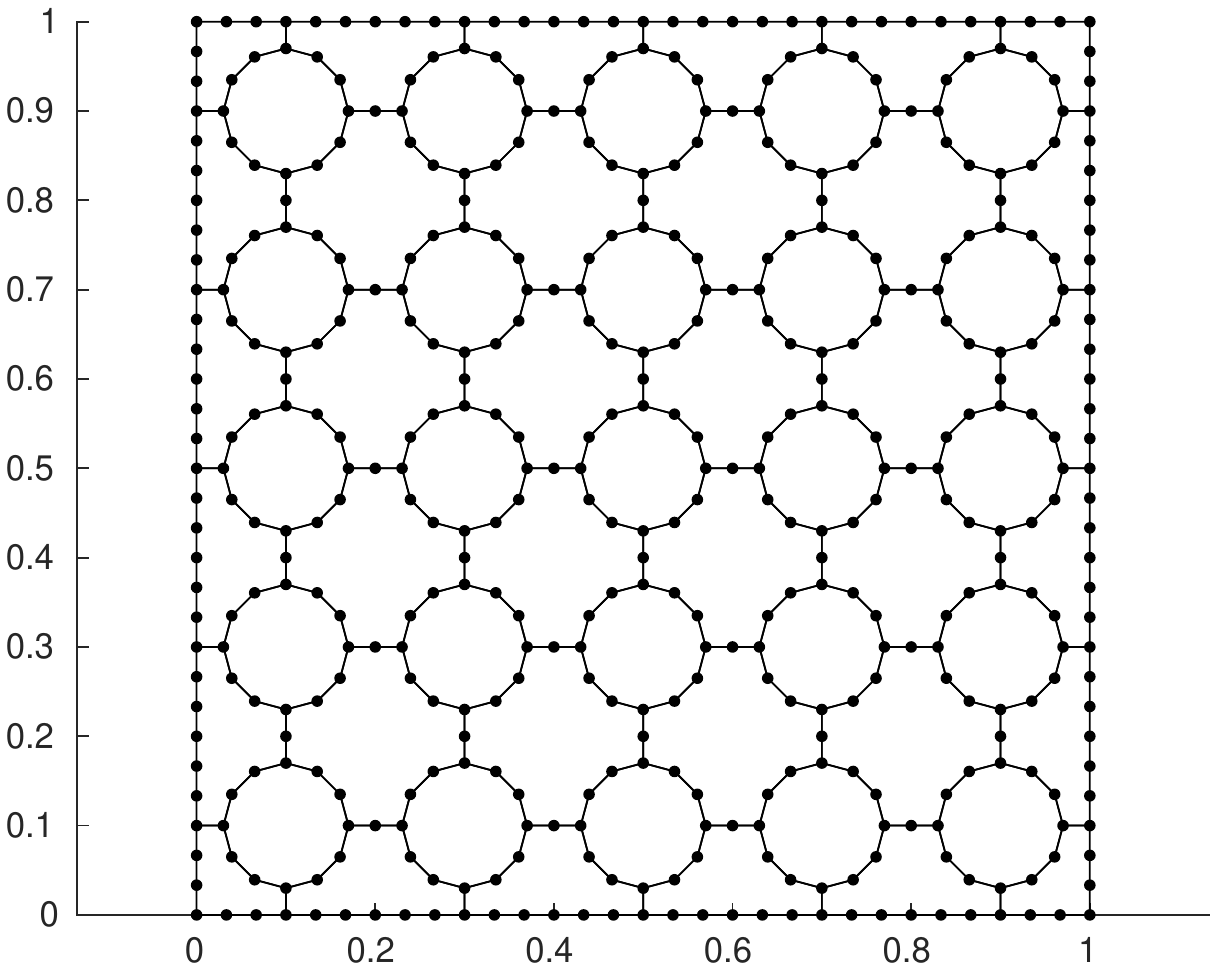}}\quad 
  \subfigure{\includegraphics[width=0.3\textwidth]{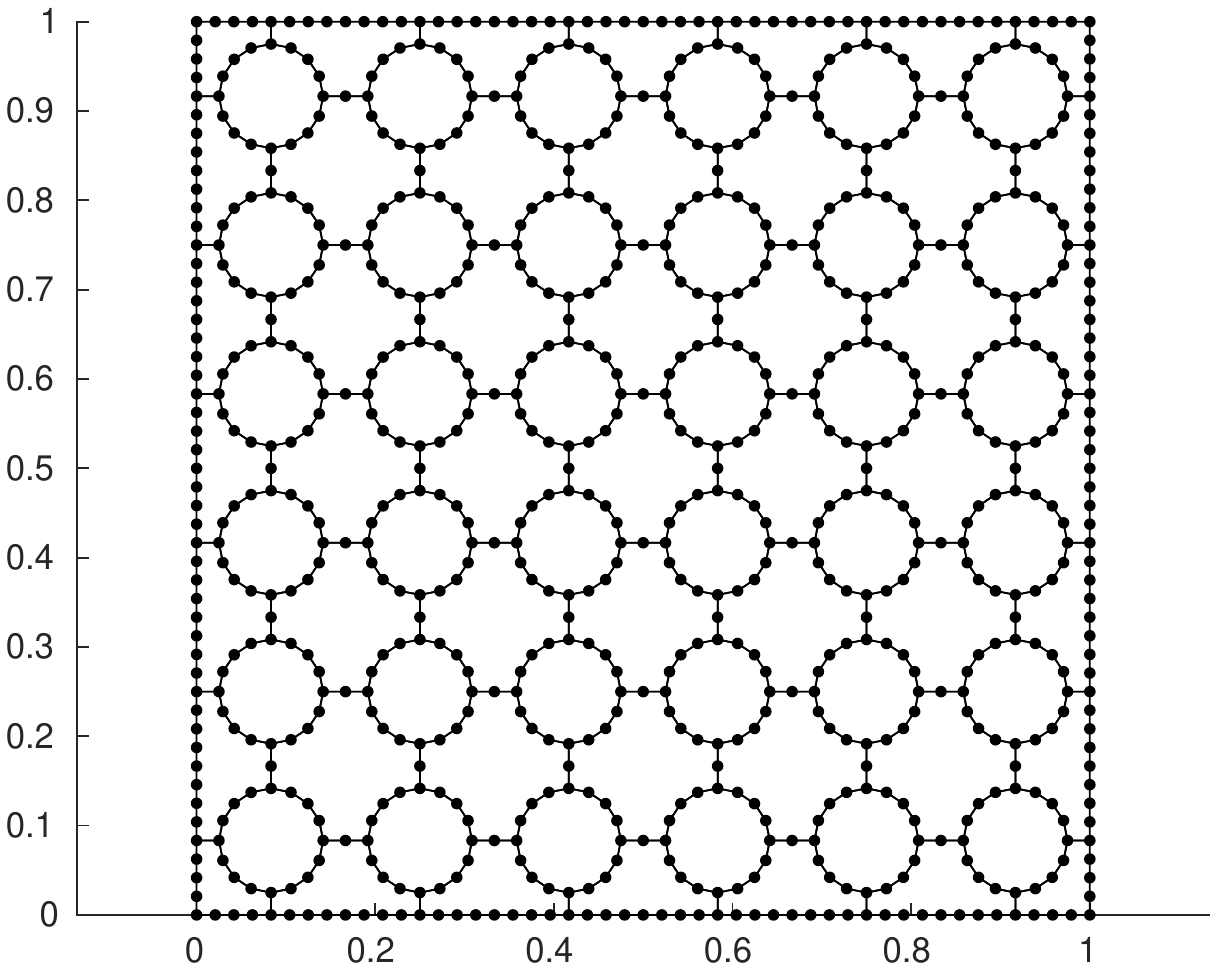}} 
  \caption{Test Case 2: Mesh family \MeshTwo~ (polygonal elements with an increasing number of edges)}
  \label{fig:n-side}
\end{figure}

\begin{figure}
  \centering
  \subfigure{\includegraphics[width=0.3\textwidth]{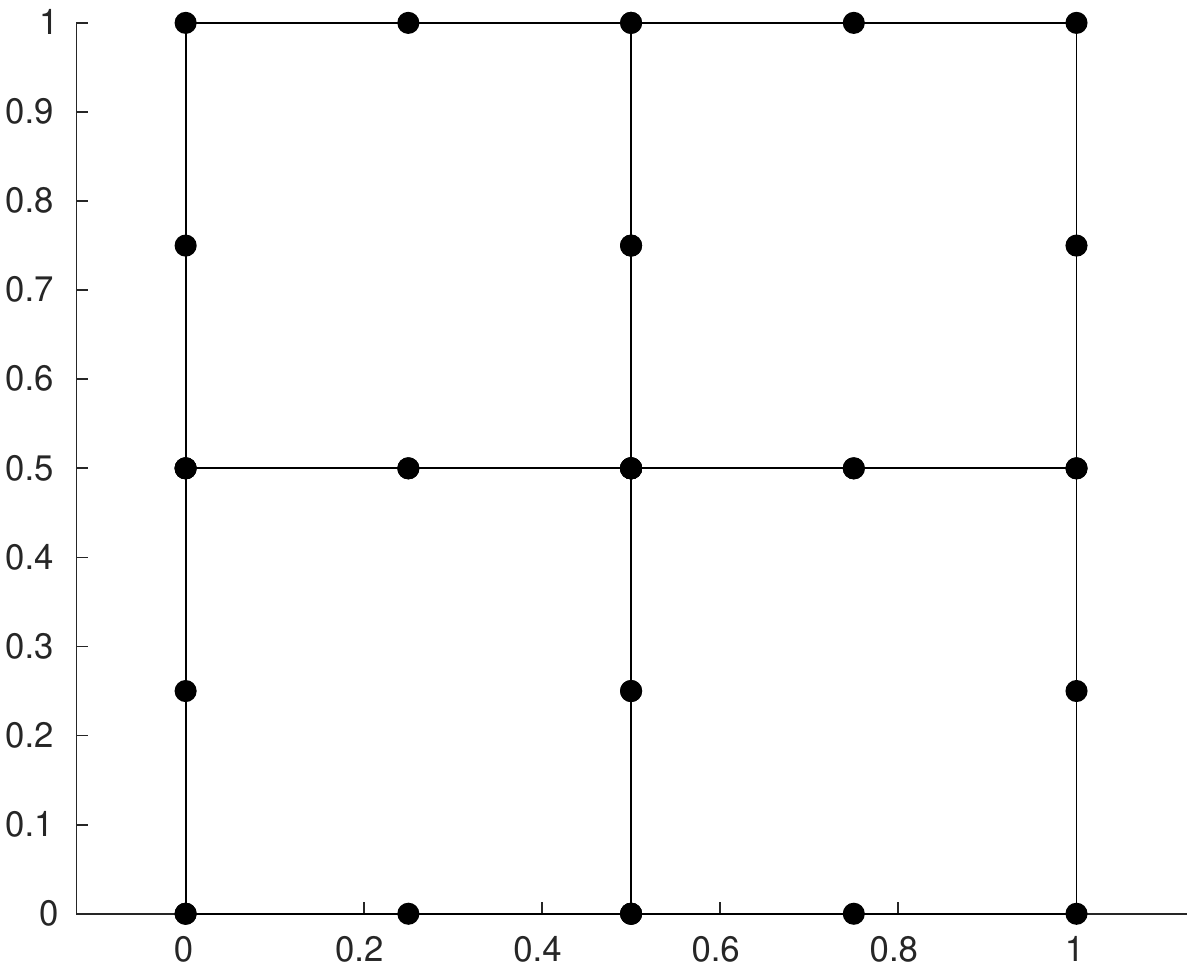}}\quad 
  \subfigure{\includegraphics[width=0.3\textwidth]{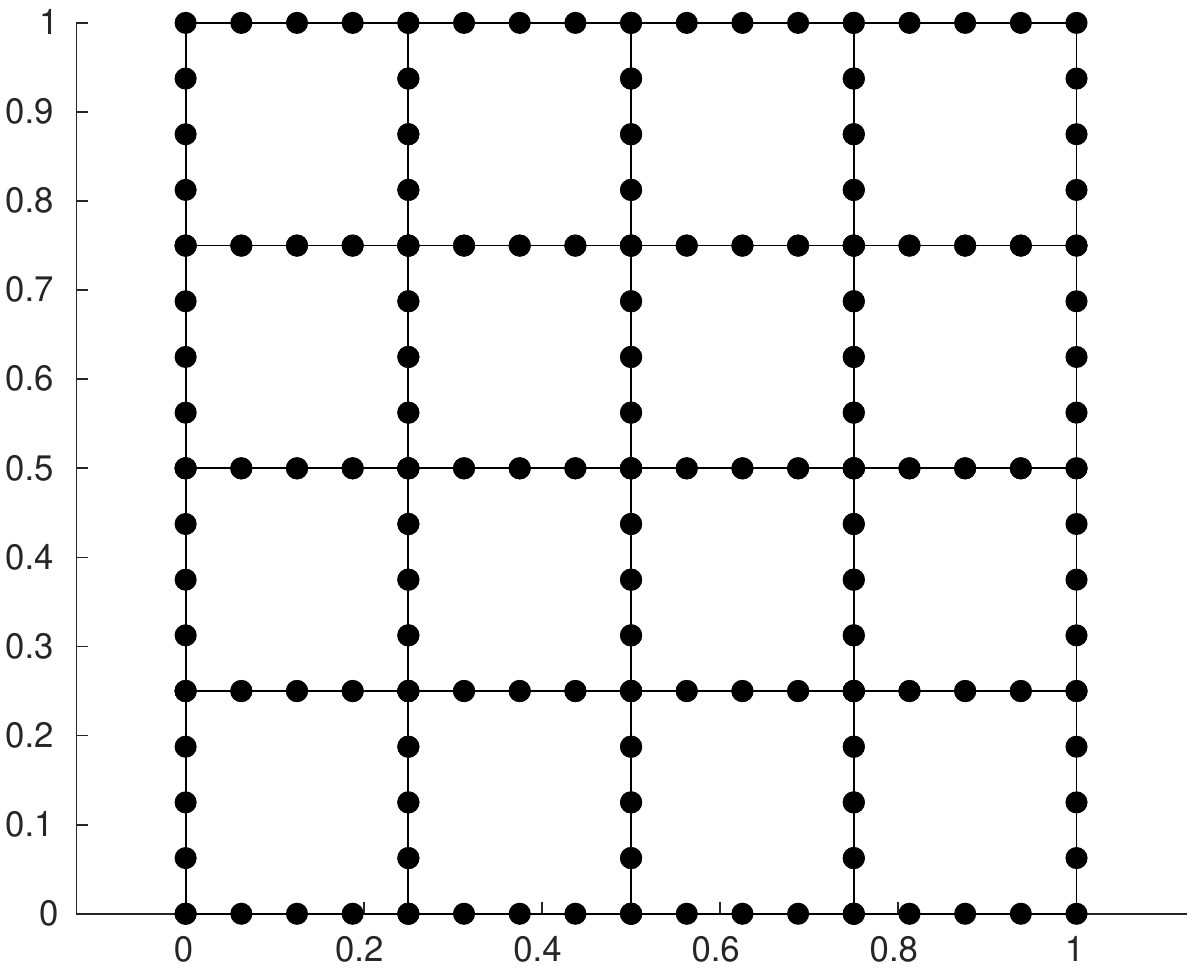}}\quad 
  \subfigure{\includegraphics[width=0.3\textwidth]{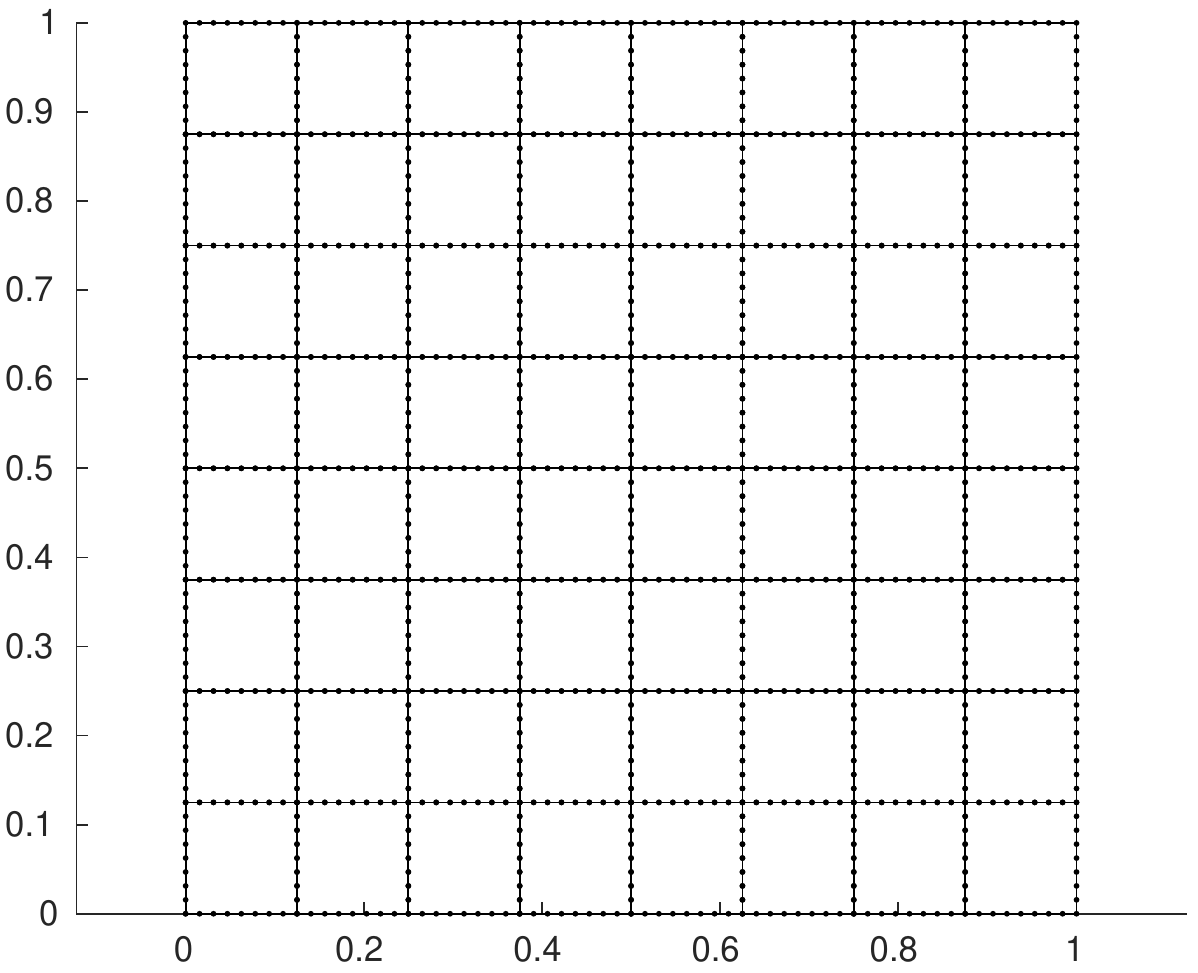}} 
  \caption{Test Case 3: Mesh family \MeshThree~(polygonal elements with a number of edges that doubles at each refinement)}
  \label{fig:dyadic}
\end{figure}

\subsection{Numerical tests}
The main goal of this section is to assess the effectiveness of the
virtual element method with the stabilization forms proposed in the
previous sections.
In particular, we want to investigate experimentally the robustness of
the approximation when using sequence of meshes with possibly
unbounded number of edges per element, and possibly very small edges
adjacent to large edges.
We recall that such kind of mesh sequences violate the mesh regularity
assumption~\SRtwo, although they may satisfy the relaxed condition
\SRtwop~or~\SRtwopp, and the weaker assumption~\SRthree.
To this end, we compare the accuracy of \emph{five} different
numerical approximations~\eqref{eq:Poisson} obtained by using these
stabilizations in the practical implementation of the VEM:

\begin{itemize}[itemsep=2mm,topsep=2mm]
\item  \Stab{1}: standard choice as proposed in \cite{AyusodeDios-Lipnikov-Manzini:2016};
\item \Stab{i}, $i=2,3,4,5$ obtained by duality with $\sigma_i^*$ defined by \eqref{eq:stP:def} with the following choices for the bilinear form $s^0$: 
  \begin{itemize}[itemsep=1.5mm,topsep=1mm]
  \item[-] $\sigma_2$: $s^0 = \ssltwo$ (weighted $L^2$ scalar product);
  \item[-] $\sigma_3$: $s^0 = \sslb$ (scaled Laplace--Beltrami operator);
  \item[-] $\sigma_4$: $s^0 = \ssrlb$ (square root of the Laplace--Beltrami operator).\
  \item[-] $\sigma_5$: $s^0 = \sswav$ (wavelet bases equivalent norm);				
  \end{itemize}
\end{itemize}

We solve Poisson problem~\eqref{eq:Poisson} on the computational
domain $\Omega=(0,1)\times(0,1)$ after setting the load term $\fs$ and
nonhomogeneous Dirichlet boundary conditions $\gs$ on the domain
boundary $\Gamma$ in accordance with the exact solution:
\begin{align}
  \us(\xs,\ys) = \frac{1}{2\pi^2}\cos(\pi\xs)\cos(\pi\ys).
  \label{eq:exact-solution}
\end{align}
All tests are performed using the enhanced non conforming virtual
element discretization space $\VhkenP$, and setting $f_h = \Piz k f$,
where $\Piz k: L^2(\Omega) \to \mathbb{P}_k(\Th)$ is the $L^2$
orthogonal projection onto the space of discontinuous piecewise
polynomials of order up to $k$ on $\Th$.  In all our implementations,
we use orthogonal polynomials as the basis in $\PS{k}(\P)$ for every
$\P\in\Th$ and $\PS{k}(\E)$ for every $\E\in\Edges$ (see Remark
\ref{rem:basis}).
The linear system assembled in any implementation of the VEM is solved
by applying the direct
solver~\verb|PaStiX|~\cite{Henon-Ramet-Roman:2002:pastix}.

\medskip
We run our numerical calculations on three different mesh families:

\medskip
\noindent
$\bullet$ \MeshOne: meshes of hexagonal elements with progressively
collapsing edges, see Figure~\ref{fig:s-hexa};

\smallskip
\noindent
$\bullet$ \MeshTwo: meshes of polygonal elements with an increasing
number of edges, see Figure~\ref{fig:n-side};

\smallskip
\noindent
$\bullet$ \MeshThree: meshes of polygonal elements with a square
boundary $\bP$ partitioned in a number of edges \\ \phantom{$\bullet$
  \MeshThree:} that doubles at each refinement, see
Figure~\ref{fig:dyadic}.

\medskip
\noindent
Three meshes of each family are shown in Figures~\ref{fig:s-hexa},
\ref{fig:n-side}, and~\ref{fig:dyadic}
For each mesh, we provide the following data:
$\Nel$, the number of elements of $\Th$;
$\Ned$, the number of edges of $\Edges$;
$\hh=\max_{\P\in\Th}\hP$, the mesh size coefficient;
$\widehat{\hh}=\min_{e \in\Edges}\hE$, length of the smallest edge, 
$\gamh = \max_{\P\in\Th}(\hP/\hhP)$, largest ratio element
diameter/smallest edge, where, we recall $\hhP = \min_{\E \in \EdgesP
} \hE$. All mesh families satisfy Assumption \TERM{G1} and
\TERM{G3}. The family \MeshOne~ does not satisfy Assumption
\TERM{G2a}, while the families \MeshTwo~ and \MeshThree~ do not
satisfy Assumption \TERM{G2b}.

\medskip
We test the convergence of the VEM by computing the relative
approximation errors defined as:
\begin{align*}
  \Error{0} =\frac{\NORM{\us-\Piz{k}\ush}{0,\Omega}}{\NORM{\us}{0,\Omega}}
  \quad\textrm{and}\quad
  \Error{1} = \frac{\NORM{\nabla\us-\Piz{k-1}\nabla\ush}{0,\Omega}}{\NORM{\nabla\us}{0,\Omega}}
\end{align*}
for $k=1,2,3,4$.

\subsubsection*{Test Case 1}



\begin{table} 
  \centering
  \begin{tabular}{
      c
      S[table-format=6.0]
      S[table-format=6.0]
      S[table-format=3.{\roundPrecision}e-1]
      S[table-format=3.{\roundPrecision}e-1]
      S[table-format=3.{\roundPrecision}e+1]
    }
    \toprule
        {Mesh} & {$\Nel$} & {$\Ned$} & {$\hh$} & {$\widehat{\hh}$} & {$\gamh$}\\
        \midrule        
        1 &   77 &   232 & 2.083333e-01 & 2.083333e-02 & 6.082763e+00\\
        2 &  281 &   844 & 1.145833e-01 & 5.208333e-03 & 1.252996e+01\\
        3 & 1073 &  3220 & 5.989583e-02 & 1.302083e-03 & 2.594224e+01\\
        4 & 4193 & 12580 & 3.059896e-02 & 3.255208e-04 & 5.277310e+01\\
        \bottomrule
  \end{tabular}
  \caption{Test Case~1: data of mesh family \MeshOneA~(shrinking
    factor = $1/2$).}
  \label{tab:s-hexa-a}
\end{table}

\begin{table}
  \centering
  \begin{tabular}{
      c
      S[table-format=6.0]
      S[table-format=6.0]
      S[table-format=3.{\roundPrecision}e-1]
      S[table-format=3.{\roundPrecision}e-1]
      S[table-format=3.{\roundPrecision}e+1]
    }
    \toprule
        {Mesh} & {$\Nel$} & {$\Ned$} & {$\hh$} & {$\widehat{\hh}$} & {$\gamh$}\\
        \midrule
        1 &   77 &   232 & 2.083333e-01 & 2.083333e-02 & 6.082763e+00\\
        2 &  281 &   844 & 1.248372e-01 & 8.138021e-05 & 8.577558e+02\\
        3 & 1073 &  3220 & 6.249936e-02 & 3.178914e-07 & 1.099063e+05\\
        4 & 4193 & 12580 & 3.125000e-02 & 1.241763e-09 & 1.406812e+07\\
        \bottomrule
  \end{tabular}
  \caption{Test Case~1: data of mesh family \MeshOneB~(shrinking
    factor = $1/128$).}
  \label{tab:s-hexa-b}
\end{table}

In Test Case~1, we apply the VEM to the family of hexagonal meshes
with collapsing edges shown in Fig.~\ref{fig:s-hexa}.
We want to investigate the robustness of the different stabilizations
\Stab{i}, $i=1,2,3,4$, with respect to the rate at which $\gamh$
grows.
To this end, at each refinement step we shrink the minimum edge length
by a \emph{shrinking factor} so that $\hhP$ decreases faster than the
mesh size factor $\hP$.
In practice, we consider two different families of refined meshes,
e.g., \MeshOneA~ and \MeshOneB, with a shrinking factor for the
minimum edge length equal to $1/2$ and $1/128$, respectively.
We report the data for these meshes in Tables~\ref{tab:s-hexa-a} and
\ref{tab:s-hexa-b}.


%
%
%


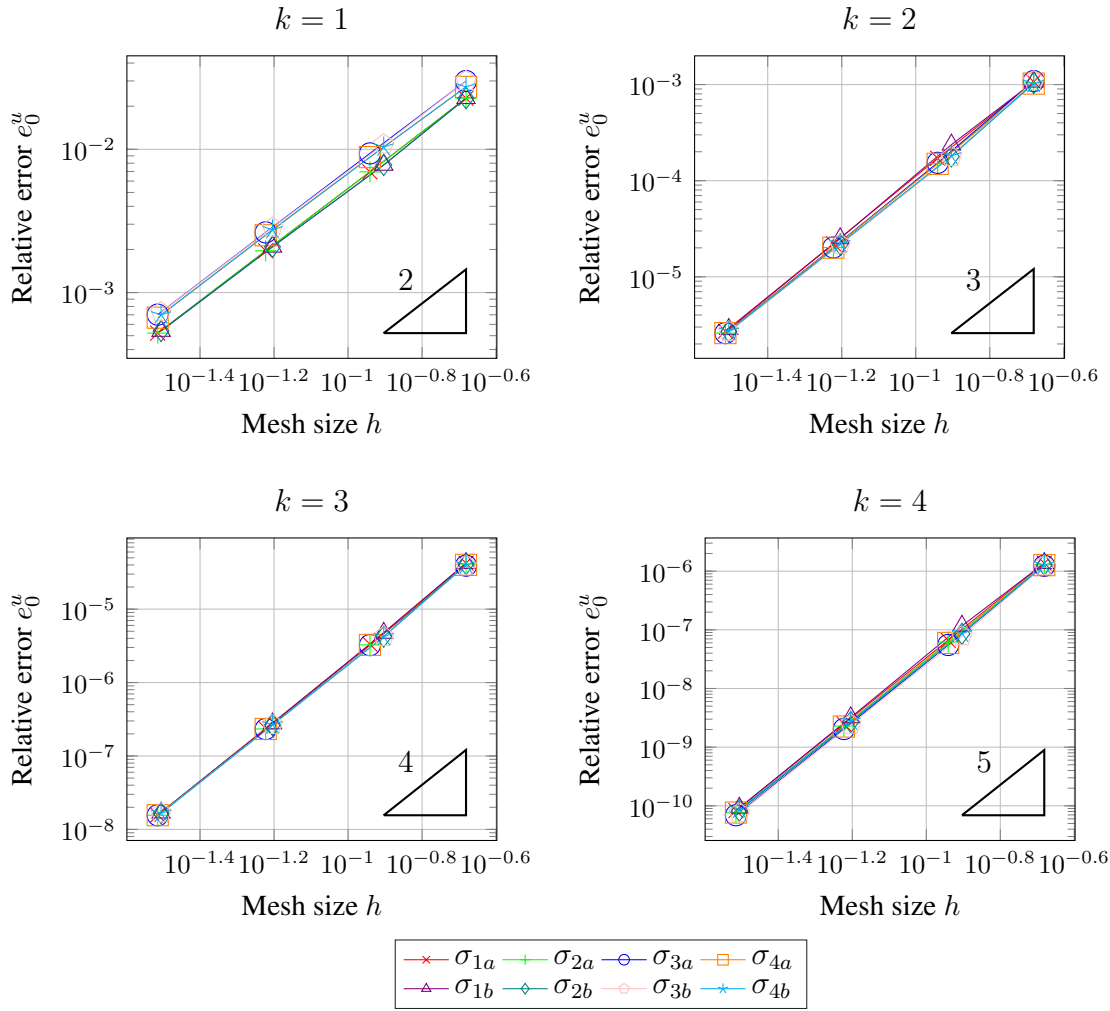
\begin{figure}\centering
  \begin{tabular}{rl}

    \begin{tikzpicture}[trim axis left]
      \begin{loglogaxis}
        [ mark size=4pt, grid=major, small,
          xlabel={Mesh size $h$},
          ylabel={Relative error $\Error{0}$},
          title ={$k=1$},
          legend columns=4, legend to name=s-hexa-L2-1, cycle list name=color ]
        
        \addplot[color=red,mark=x] coordinates {
          (0.208333,0.022726)
          (0.114583,0.00695831)
          (0.0598958,0.00195863)
          (0.030599,0.000521239)
        };
        \addplot[color=green,mark=+] coordinates {
          (0.208333,0.022726)
          (0.114583,0.00695831)
          (0.0598958,0.00195863)
          (0.030599,0.000521239)
        };
        \addplot[color=blue,mark=o] coordinates {
          (0.208333,0.0299662)
          (0.114583,0.00936208)
          (0.0598958,0.00263092)
          (0.030599,0.00070171)
        };
        \addplot[color=orange,mark=square] coordinates {
          (0.208333,0.027283)
          (0.114583,0.00882399)
          (0.0598958,0.00251179)
          (0.030599,0.000667895)
        };
        \addplot[color=violet,mark=triangle] coordinates {
          (0.208333,0.022726)
          (0.124837,0.00777765)
          (0.0624994,0.00208639)
          (0.03125,0.000539045)
        };
        \addplot[color=teal,mark=diamond] coordinates {
          (0.208333,0.022726)
          (0.124837,0.00777765)
          (0.0624994,0.00208639)
          (0.03125,0.000539045)
        };
        \addplot[color=pink,mark=pentagon] coordinates {
          (0.208333,0.0299662)
          (0.124837,0.0108332)
          (0.0624994,0.00284066)
          (0.03125,0.000729818)
        };
        \addplot[color=cyan,mark=star] coordinates {
          (0.208333,0.027283)
          (0.124837,0.0103756)
          (0.0624994,0.00274445)
          (0.03125,0.000698746)
        };
        \addplot[thick,color=black,no markers] coordinates {
          (0.124837, 0.000521239)
          (0.208333, 0.000521239)
          (0.208333, 0.00145166)
          (0.124837, 0.000521239)
        };
        \node [anchor=south east] at (0.161269,0.000869865) {$2$};    
      \end{loglogaxis}
    \end{tikzpicture}
    
    &
    
    \begin{tikzpicture}[trim axis right]
      \begin{loglogaxis}
        [ mark size=4pt, grid=major, small,
          xlabel={Mesh size $h$},
          ylabel={Relative error $\Error{0}$},
          title ={$k=2$},
          legend columns=4, legend to name=s-hexa-L2-2, cycle list name=color ]
        
        \addplot[color=red,mark=x] coordinates {
          (0.208333,0.00106758)
          (0.114583,0.000171111)
          (0.0598958,2.21097e-05)
          (0.030599,2.71663e-06)
        };
        \addplot[color=green,mark=+] coordinates {
          (0.208333,0.00104036)
          (0.114583,0.000150427)
          (0.0598958,2.01211e-05)
          (0.030599,2.59403e-06)
        };
        \addplot[color=blue,mark=o] coordinates {
          (0.208333,0.00108823)
          (0.114583,0.000153363)
          (0.0598958,2.01711e-05)
          (0.030599,2.59343e-06)
        };
        \addplot[color=orange,mark=square] coordinates {
          (0.208333,0.00102412)
          (0.114583,0.000149588)
          (0.0598958,2.00938e-05)
          (0.030599,2.59547e-06)
        };
        \addplot[color=violet,mark=triangle] coordinates {
          (0.208333,0.00106758)
          (0.124837,0.000236901)
          (0.0624994,2.50675e-05)
          (0.03125,2.85174e-06)
        };
        \addplot[color=teal,mark=diamond] coordinates {
          (0.208333,0.00104036)
          (0.124837,0.000178575)
          (0.0624994,2.17746e-05)
          (0.03125,2.69376e-06)
        };
        \addplot[color=pink,mark=pentagon] coordinates {
          (0.208333,0.00108823)
          (0.124837,0.000176392)
          (0.0624994,2.16912e-05)
          (0.03125,2.69001e-06)
        };
        \addplot[color=cyan,mark=star] coordinates {
          (0.208333,0.00102412)
          (0.124837,0.00018138)
          (0.0624994,2.20749e-05)
          (0.03125,2.71603e-06)
        };
        \addplot[thick,color=black,no markers] coordinates {
          (0.124837, 2.59343e-06)
          (0.208333, 2.59343e-06)
          (0.208333, 1.20537e-05)
          (0.124837, 2.59343e-06)
        };
        \node [anchor=south east] at (0.161269,5.59109e-06) {$3$};        
      \end{loglogaxis}
    \end{tikzpicture}

    \\[0.75em]

    \begin{tikzpicture}[trim axis left]
      \begin{loglogaxis}
        [ mark size=4pt, grid=major, small,
          xlabel={Mesh size $h$},
          ylabel={Relative error $\Error{0}$},
          title ={$k=3$},
          legend columns=4, legend to name=s-hexa-L2-3, cycle list name=color ]
        
        \addplot[color=red,mark=x] coordinates {
          (0.208333,4.22785e-05)
          (0.114583,3.32759e-06)
          (0.0598958,2.37004e-07)
          (0.030599,1.57459e-08)
        };
        \addplot[color=green,mark=+] coordinates {
          (0.208333,4.1024e-05)
          (0.114583,3.25936e-06)
          (0.0598958,2.33406e-07)
          (0.030599,1.56031e-08)
        };
        \addplot[color=blue,mark=o] coordinates {
          (0.208333,3.90914e-05)
          (0.114583,3.23537e-06)
          (0.0598958,2.32696e-07)
          (0.030599,1.55717e-08)
        };
        \addplot[color=orange,mark=square] coordinates {
          (0.208333,4.00776e-05)
          (0.114583,3.25821e-06)
          (0.0598958,2.33386e-07)
          (0.030599,1.55981e-08)
        };
        \addplot[color=violet,mark=triangle] coordinates {
          (0.208333,4.22785e-05)
          (0.124837,4.73671e-06)
          (0.0624994,2.76623e-07)
          (0.03125,1.68732e-08)
        };
        \addplot[color=teal,mark=diamond] coordinates {
          (0.208333,4.1024e-05)
          (0.124837,4.21485e-06)
          (0.0624994,2.65403e-07)
          (0.03125,1.6618e-08)
        };
        \addplot[color=pink,mark=pentagon] coordinates {
          (0.208333,3.90914e-05)
          (0.124837,4.19903e-06)
          (0.0624994,2.65064e-07)
          (0.03125,1.661e-08)
        };
        \addplot[color=cyan,mark=star] coordinates {
          (0.208333,4.00776e-05)
          (0.124837,4.21543e-06)
          (0.0624994,2.65497e-07)
          (0.03125,1.66219e-08)
        };
        \addplot[thick,color=black,no markers] coordinates {
          (0.124837, 1.55717e-08)
          (0.208333, 1.55717e-08)
          (0.208333, 1.2078e-07)
          (0.124837, 1.55717e-08)
        };
        \node [anchor=south east] at (0.161269,4.33676e-08) {$4$};
      \end{loglogaxis}
    \end{tikzpicture}
    
    &
    
    \begin{tikzpicture}[trim axis right]
      \begin{loglogaxis}
        [ mark size=4pt, grid=major, small,
          xlabel={Mesh size $h$},
          ylabel={Relative error $e_0^u$},
          title ={$k=4$},
          legend columns=4, legend to name=s-hexa-L2-4, cycle list name=color ]
        
        \addplot[color=red,mark=x] coordinates {
          (0.208333,1.36766e-06)
          (0.114583,6.63776e-08)
          (0.0598958,2.49141e-09)
          (0.030599,8.28085e-11)
        };
        \addplot[color=green,mark=+] coordinates {
          (0.208333,1.285e-06)
          (0.114583,6.17192e-08)
          (0.0598958,2.28953e-09)
          (0.030599,7.72515e-11)
        };
        \addplot[color=blue,mark=o] coordinates {
          (0.208333,1.23271e-06)
          (0.114583,5.51814e-08)
          (0.0598958,2.03809e-09)
          (0.030599,6.9009e-11)
        };
        \addplot[color=orange,mark=square] coordinates {
          (0.208333,1.26551e-06)
          (0.114583,5.97832e-08)
          (0.0598958,2.2479e-09)
          (0.030599,7.68336e-11)
        };
        \addplot[color=violet,mark=triangle] coordinates {
          (0.208333,1.36766e-06)
          (0.124837,1.16607e-07)
          (0.0624994,3.20905e-09)
          (0.03125,9.20641e-11)
        };
        \addplot[color=teal,mark=diamond] coordinates {
          (0.208333,1.285e-06)
          (0.124837,8.53929e-08)
          (0.0624994,2.66504e-09)
          (0.03125,8.30875e-11)
        };
        \addplot[color=pink,mark=pentagon] coordinates {
          (0.208333,1.23271e-06)
          (0.124837,7.81673e-08)
          (0.0624994,2.41833e-09)
          (0.03125,7.50272e-11)
        };
        \addplot[color=cyan,mark=star] coordinates {
          (0.208333,1.26551e-06)
          (0.124837,8.70319e-08)
          (0.0624994,2.76806e-09)
          (0.03125,8.71125e-11)
        };
        \addplot[thick,color=black,no markers] coordinates {
          (0.124837, 6.9009e-11)
          (0.208333, 6.9009e-11)
          (0.208333, 8.93263e-10)
          (0.124837, 6.9009e-11)
        };
        \node [anchor=south east] at (0.161269,2.4828e-10) {$5$};
      \end{loglogaxis}
    \end{tikzpicture}
    
  \end{tabular}

  \begin{tabular}{c}
    \begin{tikzpicture} 
      \begin{axis}[%
          hide axis,
          xmin=10,
          xmax=50,
          ymin=0,
          ymax=0.4,
          legend style={draw=white!15!black,legend cell align=left},
          legend columns=4
        ]
        
        \addlegendimage{color=red,mark=x}
        \addlegendentry{\Stab{1a}}
        
        \addlegendimage{color=green,mark=+}
        \addlegendentry{\Stab{2a}}
        
        \addlegendimage{color=blue,mark=o}
        \addlegendentry{\Stab{3a}}
        
        \addlegendimage{color=orange,mark=square}
        \addlegendentry{\Stab{4a}}
        
        \addlegendimage{color=violet,mark=triangle}
        \addlegendentry{\Stab{1b}}
        
        \addlegendimage{color=teal,mark=diamond}
        \addlegendentry{\Stab{2b}}
        
        \addlegendimage{color=pink,mark=pentagon}
        \addlegendentry{\Stab{3b}}
        
        \addlegendimage{color=cyan,mark=star}
        \addlegendentry{\Stab{4b}}
      \end{axis}
    \end{tikzpicture}
  \end{tabular}

  \caption{Test Case~1: convergence plots for $\Error{0}$
    using the hexagonal meshes \MeshOneA~and \MeshOneB, and the four
    stabilization strategies \Stab{i}, $i=1,2,3,4$.
    Top row: plots for $k=1$ (left panel) and $k=2$ (right panel);
    bottom row: plots for $k=3$ (left panel) and $k=4$ (right panel).}
  \label{fig:s-hexa-L2}
\end{figure}


\begin{figure}\centering
  \begin{tabular}{cc}

    \begin{tikzpicture}[trim axis left]
      \begin{loglogaxis}
        [ mark size=4pt, grid=major, small,
          xlabel={Mesh size $h$},
          ylabel={Relative error $\Error{1}$},
          title ={$k=1$},
          legend columns=4, legend to name=s-hexa-H1-1 ]
        \addplot[color=red,mark=x] coordinates {
          (0.208333,0.169548)
          (0.114583,0.0875832)
          (0.0598958,0.0440485)
          (0.030599,0.0221442)
        };
        \addplot[color=green,mark=+] coordinates {
          (0.208333,0.169548)
          (0.114583,0.0875832)
          (0.0598958,0.0440485)
          (0.030599,0.0221442)
        };
        \addplot[color=blue,mark=o] coordinates {
          (0.208333,0.1627)
          (0.114583,0.0863716)
          (0.0598958,0.0441559)
          (0.030599,0.0222617)
        };
        \addplot[color=orange,mark=square] coordinates {
          (0.208333,0.160732)
          (0.114583,0.0848095)
          (0.0598958,0.0435191)
          (0.030599,0.0220622)
        };
        \addplot[color=violet,mark=triangle] coordinates {
          (0.208333,0.169548)
          (0.124837,0.0941241)
          (0.0624994,0.0453339)
          (0.03125,0.0224331)
        };
        \addplot[color=teal,mark=diamond] coordinates {
          (0.208333,0.169548)
          (0.124837,0.0941241)
          (0.0624994,0.0453339)
          (0.03125,0.0224331)
        };
        \addplot[color=pink,mark=pentagon] coordinates {
          (0.208333,0.1627)
          (0.124837,0.0913506)
          (0.0624994,0.045365)
          (0.03125,0.0225534)
        };
        \addplot[color=cyan,mark=star] coordinates {
          (0.208333,0.160732)
          (0.124837,0.0888322)
          (0.0624994,0.0445384)
          (0.03125,0.0223219)
        };
        \addplot[thick,color=black,no markers] coordinates {
          (0.124837, 0.0220622)
          (0.208333, 0.0220622)
          (0.208333, 0.0368183)
          (0.124837, 0.0220622)
        };
        \node [anchor=south east] at (0.161269,0.0285007) {$1$};        
      \end{loglogaxis}
    \end{tikzpicture}
    
    &

    \begin{tikzpicture}[trim axis right]
      \begin{loglogaxis}
        [ mark size=4pt, grid=major, small,
          xlabel={Mesh size $h$},
          ylabel={Relative error $\Error{1}$},
          title ={$k=2$},
          legend columns=4, legend to name=s-hexa-H1-2, cycle list name=color ]
        
        \addplot[color=red,mark=x] coordinates {
          (0.208333,0.0155473)
          (0.114583,0.00493841)
          (0.0598958,0.00120886)
          (0.030599,0.000284077)
        };
        \addplot[color=green,mark=+] coordinates {
          (0.208333,0.0127528)
          (0.114583,0.00375567)
          (0.0598958,0.000983585)
          (0.030599,0.000249722)
        };
        \addplot[color=blue,mark=o] coordinates {
          (0.208333,0.0146027)
          (0.114583,0.00396522)
          (0.0598958,0.00100208)
          (0.030599,0.000251553)
        };
        \addplot[color=orange,mark=square] coordinates {
          (0.208333,0.0128533)
          (0.114583,0.00374038)
          (0.0598958,0.000983842)
          (0.030599,0.00025034)
        };
        \addplot[color=violet,mark=triangle] coordinates {
          (0.208333,0.0155473)
          (0.124837,0.00666872)
          (0.0624994,0.00137974)
          (0.03125,0.000300109)
        };
        \addplot[color=teal,mark=diamond] coordinates {
          (0.208333,0.0127528)
          (0.124837,0.00443581)
          (0.0624994,0.00106089)
          (0.03125,0.000258895)
        };
        \addplot[color=pink,mark=pentagon] coordinates {
          (0.208333,0.0146027)
          (0.124837,0.00451534)
          (0.0624994,0.00107456)
          (0.03125,0.000260514)
        };
        \addplot[color=cyan,mark=star] coordinates {
          (0.208333,0.0128533)
          (0.124837,0.00456979)
          (0.0624994,0.00109181)
          (0.03125,0.000263746)
        };
        \addplot[thick,color=black,no markers] coordinates {
          (0.124837, 0.000249722)
          (0.208333, 0.000249722)
          (0.208333, 0.000695483)
          (0.124837, 0.000249722)
        };
        \node [anchor=south east] at (0.161269,0.000416746) {$2$};
      \end{loglogaxis}
    \end{tikzpicture}

    \\[1em]

    \begin{tikzpicture}[trim axis left]
      \begin{loglogaxis}
        [ mark size=4pt, grid=major, small,
          xlabel={Mesh size $h$},
          ylabel={Relative error $\Error{1}$},
          title ={$k=3$},
          legend columns=4, legend to name=s-hexa-H1-3, cycle list name=color ]
        
        \addplot[color=red,mark=x] coordinates {
          (0.208333,0.000674046)
          (0.114583,0.000104232)
          (0.0598958,1.44128e-05)
          (0.030599,1.88572e-06)
        };
        \addplot[color=green,mark=+] coordinates {
          (0.208333,0.000641602)
          (0.114583,9.92041e-05)
          (0.0598958,1.39441e-05)
          (0.030599,1.85014e-06)
        };
        \addplot[color=blue,mark=o] coordinates {
          (0.208333,0.000643385)
          (0.114583,9.95734e-05)
          (0.0598958,1.39768e-05)
          (0.030599,1.85259e-06)
        };
        \addplot[color=orange,mark=square] coordinates {
          (0.208333,0.000640451)
          (0.114583,9.92164e-05)
          (0.0598958,1.39447e-05)
          (0.030599,1.85013e-06)
        };
        \addplot[color=violet,mark=triangle] coordinates {
          (0.208333,0.000674046)
          (0.124837,0.000160578)
          (0.0624994,1.70883e-05)
          (0.03125,2.02453e-06)
        };
        \addplot[color=teal,mark=diamond] coordinates {
          (0.208333,0.000641602)
          (0.124837,0.000124338)
          (0.0624994,1.56312e-05)
          (0.03125,1.95873e-06)
        };
        \addplot[color=pink,mark=pentagon] coordinates {
          (0.208333,0.000643385)
          (0.124837,0.00012474)
          (0.0624994,1.56673e-05)
          (0.03125,1.96134e-06)
        };
        \addplot[color=cyan,mark=star] coordinates {
          (0.208333,0.000640451)
          (0.124837,0.00012431)
          (0.0624994,1.56289e-05)
          (0.03125,1.95856e-06)
        };
        \addplot[thick,color=black,no markers] coordinates {
          (0.124837, 1.85013e-06)
          (0.208333, 1.85013e-06)
          (0.208333, 8.59897e-06)
          (0.124837, 1.85013e-06)
        };
        \node [anchor=south east] at (0.161269,3.98864e-06) {$3$};
      \end{loglogaxis}
    \end{tikzpicture}
    
    &
    
    \begin{tikzpicture}[trim axis right]
      \begin{loglogaxis}
        [ mark size=4pt, grid=major, small,
          xlabel={Mesh size $h$},
          ylabel={Relative error $\Error{1}$},
          title ={$k=4$},
          legend columns=4, legend to name=s-hexa-H1-4, cycle list name=color ]
        
        \addplot[color=red,mark=x] coordinates {
          (0.208333,3.25352e-05)
          (0.114583,3.31521e-06)
          (0.0598958,2.69625e-07)
          (0.030599,1.70134e-08)
        };
        \addplot[color=green,mark=+] coordinates {
          (0.208333,2.5936e-05)
          (0.114583,2.18664e-06)
          (0.0598958,1.61485e-07)
          (0.030599,1.09294e-08)
        };
        \addplot[color=blue,mark=o] coordinates {
          (0.208333,2.57535e-05)
          (0.114583,2.17284e-06)
          (0.0598958,1.6242e-07)
          (0.030599,1.10216e-08)
        };
        \addplot[color=orange,mark=square] coordinates {
          (0.208333,2.58293e-05)
          (0.114583,2.16445e-06)
          (0.0598958,1.60757e-07)
          (0.030599,1.09141e-08)
        };
        \addplot[color=violet,mark=triangle] coordinates {
          (0.208333,3.25352e-05)
          (0.124837,9.32246e-06)
          (0.0624994,4.41285e-07)
          (0.03125,2.11989e-08)
        };
        \addplot[color=teal,mark=diamond] coordinates {
          (0.208333,2.5936e-05)
          (0.124837,3.20183e-06)
          (0.0624994,1.94463e-07)
          (0.03125,1.19282e-08)
        };
        \addplot[color=pink,mark=pentagon] coordinates {
          (0.208333,2.57535e-05)
          (0.124837,3.22989e-06)
          (0.0624994,1.96352e-07)
          (0.03125,1.20464e-08)
        };
        \addplot[color=cyan,mark=star] coordinates {
          (0.208333,2.58293e-05)
          (0.124837,3.21331e-06)
          (0.0624994,1.96455e-07)
          (0.03125,1.20993e-08)
        };
        \addplot[thick,color=black,no markers] coordinates {
          (0.124837, 1.09141e-08)
          (0.208333, 1.09141e-08)
          (0.208333, 8.46539e-08)
          (0.124837, 1.09141e-08)
        };
        \node [anchor=south east] at (0.161269,3.03961e-08) {$4$};
      \end{loglogaxis}
    \end{tikzpicture}
    
  \end{tabular}

  \begin{tabular}{c}
    \begin{tikzpicture} 
      \begin{axis}[%
          hide axis,
          xmin=10,
          xmax=50,
          ymin=0,
          ymax=0.4,
          legend style={draw=white!15!black,legend cell align=left},
          legend columns=4
        ]
        
        \addlegendimage{color=red,mark=x}
        \addlegendentry{\Stab{1a}}
        
        \addlegendimage{color=green,mark=+}
        \addlegendentry{\Stab{2a}}
        
        \addlegendimage{color=blue,mark=o}
        \addlegendentry{\Stab{3a}}
        
        \addlegendimage{color=orange,mark=square}
        \addlegendentry{\Stab{4a}}
        
        \addlegendimage{color=violet,mark=triangle}
        \addlegendentry{\Stab{1b}}
        
        \addlegendimage{color=teal,mark=diamond}
        \addlegendentry{\Stab{2b}}
        
        \addlegendimage{color=pink,mark=pentagon}
        \addlegendentry{\Stab{3b}}
        
        \addlegendimage{color=cyan,mark=star}
        \addlegendentry{\Stab{4b}}
      \end{axis}
    \end{tikzpicture}    
  \end{tabular}
  \caption{Test Case~1: convergence plots for  $\Error{1}$
    using the hexagonal meshes \MeshOneA~and \MeshOneB, and the four
    stabilization strategies \Stab{i}, $i=1,2,3,4$.
    Top row: plots for $k=1$ (left panel) and $k=2$ (bottom panel);
    bottom row: plots for $k=3$ (left panel) and $k=4$ (right panel).}
  \label{fig:s-hexa-H1}
\end{figure}
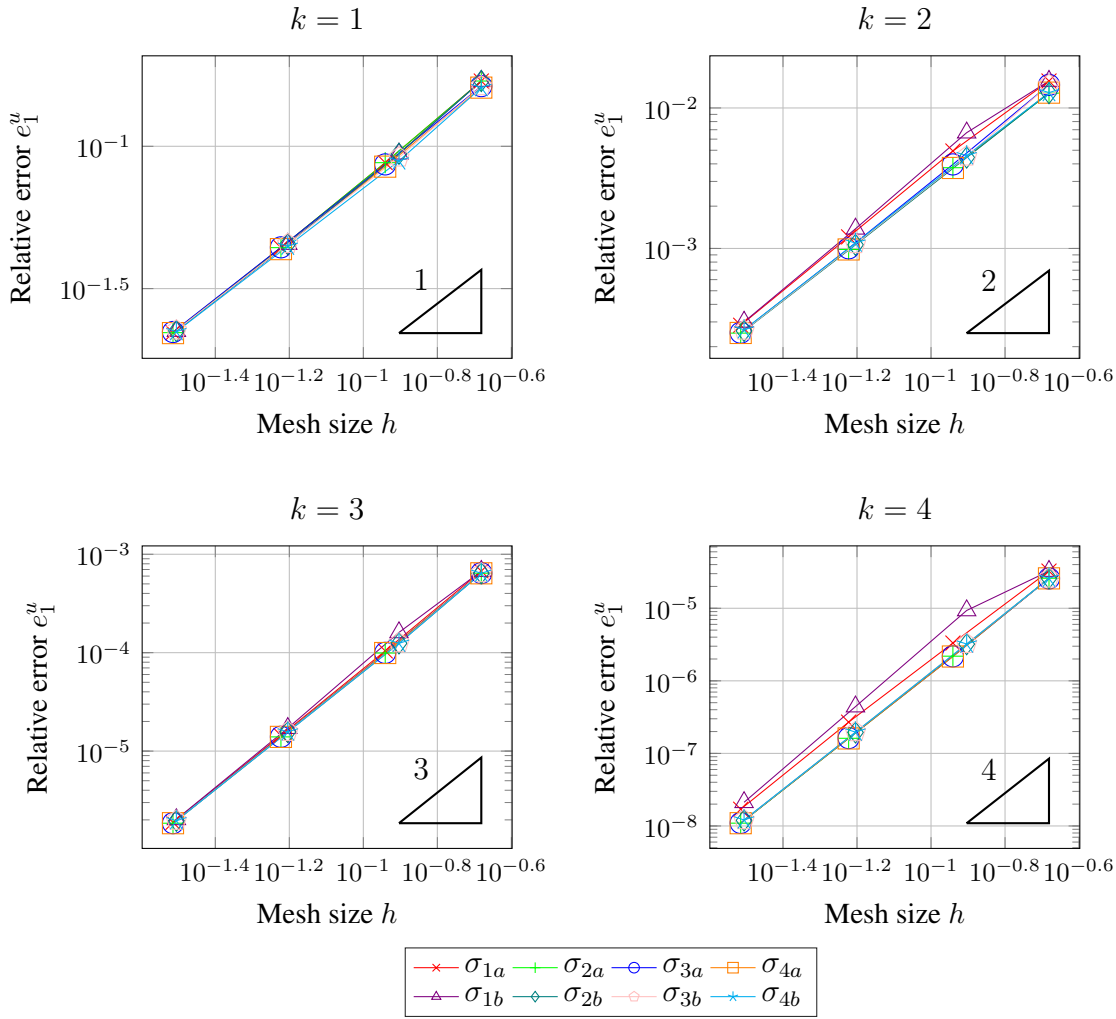

Figures~\ref{fig:s-hexa-L2}, \ref{fig:s-hexa-H1} show the convergence
plots of $\Error{0}$ and $\Error{1}$.
As expected, the errors $\Error{0}$ and $\Error{1}$ behave like
$\calO(h^{k+1})$ and $\calO(h^{k})$, respectively.
All the stabilization strategies exhibit similar performance.
When $k=1$, as far as $\Error{0}$ is concerned, the stabilizations
\Stab{1} and \Stab{2} exhibit a slightly more favorable error constant
with respect to \Stab{3} and \Stab{4}.
When $k=4$, as far as $\Error{1}$ is concerned, the \Stab{1}
stabilization does not perform as well as the other ones.

\subsubsection*{Test Case 2}



\begin{table}[htb]
  \centering
  \begin{tabular}{
  		c
      S[table-format=6.0]
      S[table-format=6.0]
      S[table-format=3.{\roundPrecision}e-1]
      S[table-format=3.{\roundPrecision}e-1]
      S[table-format=3.{\roundPrecision}e+1]
    }
    \toprule
        {Mesh} & {$\Nel$} & {$\Ned$} & {$\hh$} & {$\widehat{\hh}$} & {$\gamh$}\\
        \midrule
        1 &   41 &  448  & 2.610077e-01 & 3.125000e-02 & 8.089499e+00\\
        2 &  145 &  2368 & 1.305038e-01 & 9.375000e-03 & 1.392041e+01\\
        3 &  545 & 11136 & 6.525192e-02 & 4.288250e-03 & 1.521644e+01\\
        4 & 2113 & 55552 & 3.262596e-02 & 1.562500e-03 & 2.088061e+01\\
        \bottomrule
  \end{tabular}
  \caption{Test Case~2: data of mesh family \MeshTwoA.}
  \label{tab:Mesh2A:data}
\end{table}

\begin{table}[htb]
  \centering
  \begin{tabular}{
      c
      S[table-format=6.0]
      S[table-format=6.0]
      S[table-format=3.{\roundPrecision}e-1]
      S[table-format=3.{\roundPrecision}e-1]
      S[table-format=3.{\roundPrecision}e+1]
    }
    \toprule
        {Mesh} & {$\Nel$} & {$\Ned$} & {$\hh$} & {$\widehat{\hh}$} & {$\gamh$}\\
        \midrule
        1 &   41 &    448 & 2.610077e-01 & 3.125000e-02 & 8.089499e+00\\
        2 &  145 &   3008 & 1.305038e-01 & 8.576500e-03 & 1.521644e+01\\
        3 &  545 &  22400 & 6.525192e-02 & 2.083333e-03 & 3.039623e+01\\
        4 & 2113 & 175360 & 3.262596e-02 & 5.208333e-04 & 6.264184e+01\\
        \bottomrule
  \end{tabular}
  \caption{Test Case~2: data of mesh family \MeshTwoB.}
  \label{tab:Mesh2B:data}
\end{table}



\newcommand{\TABROW}[9]{#1 & \quad#2 & \!\!(#3) & \quad#4 & \!\!(#5) & \quad#6 & \!\!(#7) & \quad#8 & \!\!(#9)\\}
\newcommand{\ROW}[9]{#1\,\,\,#2(#3)\,\,\,#4(#5)\,\,\,#6(#7)\,\,\,#8(#9)} 
\begin{table}[htb]
  \centering
  \begin{tabular}{rrlrlrlrl}
    \toprule
    \midrule
    \TABROW{1}{14}{4}{16}{  16}{20}{ 12}{24}{  9}
    \TABROW{2}{20}{4}{24}{  64}{30}{ 28}{40}{ 49}
    \TABROW{3}{26}{4}{32}{ 256}{38}{ 60}{48}{225}
    \TABROW{4}{34}{4}{40}{1024}{50}{124}{64}{961}
    \bottomrule
  \end{tabular}
  \caption{Test Case~2: additional data of mesh family \MeshTwoA.}
  \label{tab:mesh2A:additional}
\end{table}

\begin{table}[htb]
	\centering
	\begin{tabular}{rrlrlrlrl}
		\toprule
		\midrule
 \TABROW{1} { 14}{4}{ 16}{  16}{ 20}{ 12}{  24}{  9}
 \TABROW{2} { 26}{4}{ 32}{  64}{ 38}{ 28}{  48}{ 49}
 \TABROW{3} { 54}{4}{ 64}{ 256}{ 78}{ 60}{  96}{225}
 \TABROW{4} {105}{4}{125}{1024}{155}{124}{ 195}{961}
		\bottomrule
	\end{tabular}
	\caption{Test Case~2: additional data of mesh family \MeshTwoB.}
	\label{tab:mesh2B:additional}
\end{table}

In the second test case, we consider the family of meshes shown in
Figure~\ref{fig:n-side}.
At each mesh-refinement step, we increase the number of edges per
element.
We consider two families of meshes, \MeshTwoA{} and \MeshTwoB, which
are characterized by a different growth rate of $\hP/\hhP$.
The data of these meshes are collected in
Tables~\ref{tab:Mesh2A:data}-\ref{tab:Mesh2B:data}.
Additionally, for these meshes we report in
Tables~\ref{tab:mesh2A:additional}--\ref{tab:mesh2B:additional} the
mesh number, the number of edges per element and for each one of these
data, the number of elements having that specific number of edges.
For example, the first line of Table~\ref{tab:mesh2A:additional},
i.e., ``\ROW{1}{14}{4}{16}{16}{20}{12}{24}{9}'' must be read as:
``\emph{Mesh (refinement) $1$ has $4$ elements with $14$ edges, $16$
elements with $16$ edges, $12$ elements with $20$ edges and $9$
elements with $24$ edges}".



\begin{figure}\centering
  \begin{tabular}{cc}
    
    \begin{tikzpicture}[trim axis left]
      \begin{loglogaxis}
        [ mark size=4pt, grid=major, small,
          xlabel={Mesh size $\hh$},
          ylabel={Relative error $\Error{0}$},
          title ={$k=1$},
          legend columns=4, legend to name=n-side-L2-1 ]
        
        \addplot[color=red,mark=x] coordinates {
          (0.261008,0.0317165)
          (0.130504,0.00926166)
          (0.0652519,0.0024825)
          (0.032626,0.000660216)
        };
        \addplot[color=green,mark=+] coordinates {
          (0.261008,0.0317165)
          (0.130504,0.00926166)
          (0.0652519,0.0024825)
          (0.032626,0.000660216)
        };
        \addplot[color=blue,mark=o] coordinates {
          (0.261008,0.0497374)
          (0.130504,0.0147918)
          (0.0652519,0.00402225)
          (0.032626,0.00105073)
        };
        \addplot[color=orange,mark=square] coordinates {
          (0.261008,0.0311057)
          (0.130504,0.00853502)
          (0.0652519,0.00223087)
          (0.032626,0.000568926)
        };
        \addplot[color=violet,mark=triangle] coordinates {
          (0.261008,0.0317165)
          (0.130504,0.0094974)
          (0.0652519,0.00276675)
          (0.032626,0.000852521)
        };
        \addplot[color=teal,mark=diamond] coordinates {
          (0.261008,0.0317165)
          (0.130504,0.0094974)
          (0.0652519,0.00276675)
          (0.032626,0.000852521)
        };
        \addplot[color=pink,mark=pentagon] coordinates {
          (0.261008,0.0497374)
          (0.130504,0.0150205)
          (0.0652519,0.00414966)
          (0.032626,0.00108562)
        };
        \addplot[color=cyan,mark=star] coordinates {
          (0.261008,0.0311057)
          (0.130504,0.00848767)
          (0.0652519,0.00221651)
          (0.032626,0.000566005)
        };
        \addplot[thick,color=black,no markers] coordinates {
          (0.130504, 0.000566005)
          (0.261008, 0.000566005)
          (0.261008, 0.00226402)
          (0.130504, 0.000566005)
        };
        \node [anchor=south east] at (0.184561,0.00113201) {$2$};
        
      \end{loglogaxis}
    \end{tikzpicture}
    
    &
    
    \begin{tikzpicture}[trim axis right]
      \begin{loglogaxis}
        [ mark size=4pt, grid=major, small,
          xlabel={Mesh size $\hh$},
          ylabel={Relative error $\Error{0}$},
          title ={$k=2$},
          legend columns=4, legend to name=n-side-L2-2 ]
        
        \addplot[color=red,mark=x] coordinates {
          (0.261008,0.00270682)
          (0.130504,0.000366127)
          (0.0652519,4.71062e-05)
          (0.032626,5.96509e-06)
        };
        \addplot[color=green,mark=+] coordinates {
          (0.261008,0.00271794)
          (0.130504,0.000366379)
          (0.0652519,4.71168e-05)
          (0.032626,5.96543e-06)
        };
        \addplot[color=blue,mark=o] coordinates {
          (0.261008,0.00300651)
          (0.130504,0.000381559)
          (0.0652519,4.79161e-05)
          (0.032626,6.03563e-06)
        };
        \addplot[color=orange,mark=square] coordinates {
          (0.261008,0.00279333)
          (0.130504,0.000368949)
          (0.0652519,4.72338e-05)
          (0.032626,5.97218e-06)
        };
        \addplot[color=violet,mark=triangle] coordinates {
          (0.261008,0.00270682)
          (0.130504,0.000365967)
          (0.0652519,4.7104e-05)
          (0.032626,5.9614e-06)
        };
        \addplot[color=teal,mark=diamond] coordinates {
          (0.261008,0.00271794)
          (0.130504,0.000366168)
          (0.0652519,4.71068e-05)
          (0.032626,5.96144e-06)
        };
        \addplot[color=pink,mark=pentagon] coordinates {
          (0.261008,0.00300651)
          (0.130504,0.000381465)
          (0.0652519,4.80369e-05)
          (0.032626,6.04588e-06)
        };
        \addplot[color=cyan,mark=star] coordinates {
          (0.261008,0.00279333)
          (0.130504,0.00036825)
          (0.0652519,4.71355e-05)
          (0.032626,5.96107e-06)
        };
        \addplot[thick,color=black,no markers] coordinates {
          (0.130504, 5.96107e-06)
          (0.261008, 5.96107e-06)
          (0.261008, 4.76886e-05)
          (0.130504, 5.96107e-06)
        };
        \node [anchor=south east] at (0.184561,1.68605e-05) {$3$};    
      \end{loglogaxis}
    \end{tikzpicture}

    \\[0.75em]
    
    \begin{tikzpicture}[trim axis left]
      \begin{loglogaxis}
        [ mark size=4pt, grid=major, small,
          xlabel={Mesh size $\hh$},
          ylabel={Relative error $\Error{0}$},
          title ={$k=3$},
          legend columns=4, legend to name=n-side-L2-3 ]
        
        \addplot[color=red,mark=x] coordinates {
          (0.261008,0.000138064)
          (0.130504,9.67947e-06)
          (0.0652519,6.40067e-07)
          (0.032626,4.32841e-08)
        };
        \addplot[color=green,mark=+] coordinates {
          (0.261008,0.000139291)
          (0.130504,9.78388e-06)
          (0.0652519,6.45767e-07)
          (0.032626,4.35305e-08)
        };
        \addplot[color=blue,mark=o] coordinates {
          (0.261008,0.000238412)
          (0.130504,1.72974e-05)
          (0.0652519,1.12717e-06)
          (0.032626,7.27265e-08)
        };
        \addplot[color=orange,mark=square] coordinates {
          (0.261008,0.000158378)
          (0.130504,1.04859e-05)
          (0.0652519,6.74686e-07)
          (0.032626,4.26485e-08)
        };
        \addplot[color=violet,mark=triangle] coordinates {
          (0.261008,0.000138064)
          (0.130504,9.91104e-06)
          (0.0652519,7.46594e-07)
          (0.032626,5.67118e-08)
        };
        \addplot[color=teal,mark=diamond] coordinates {
          (0.261008,0.000139291)
          (0.130504,1.00001e-05)
          (0.0652519,7.48828e-07)
          (0.032626,5.67515e-08)
        };
        \addplot[color=pink,mark=pentagon] coordinates {
          (0.261008,0.000238412)
          (0.130504,1.74831e-05)
          (0.0652519,1.16562e-06)
          (0.032626,7.49619e-08)
        };
        \addplot[color=cyan,mark=star] coordinates {
          (0.261008,0.000158378)
          (0.130504,1.04347e-05)
          (0.0652519,6.67416e-07)
          (0.032626,4.20956e-08)
        };
        \addplot[thick,color=black,no markers] coordinates {
          (0.130504, 4.20956e-08)
          (0.261008, 4.20956e-08)
          (0.261008, 6.7353e-07)
          (0.130504, 4.20956e-08)
        };
        \node [anchor=south east] at (0.184561,1.68382e-07) {$4$};
      \end{loglogaxis}
    \end{tikzpicture}
    
    &
    
    \begin{tikzpicture}[trim axis right]
      \begin{loglogaxis}
        [ mark size=4pt, grid=major, small,
          xlabel={Mesh size $\hh$},
          ylabel={Relative error $\Error{0}$},
          title ={$k=4$},
          legend columns=4, legend to name=n-side-L2-4 ]
        
        \addplot[color=red,mark=x] coordinates {
          (0.261008,7.07697e-06)
          (0.130504,2.36429e-07)
          (0.0652519,7.74435e-09)
          (0.032626,2.51095e-10)
        };
        \addplot[color=green,mark=+] coordinates {
          (0.261008,6.77732e-06)
          (0.130504,2.3691e-07)
          (0.0652519,7.78623e-09)
          (0.032626,2.52669e-10)
        };
        \addplot[color=blue,mark=o] coordinates {
          (0.261008,7.88205e-06)
          (0.130504,4.61336e-07)
          (0.0652519,1.66601e-08)
          (0.032626,5.834e-10)
        };
        \addplot[color=orange,mark=square] coordinates {
          (0.261008,6.74206e-06)
          (0.130504,2.43289e-07)
          (0.0652519,8.00199e-09)
          (0.032626,2.55844e-10)
        };
        \addplot[color=violet,mark=triangle] coordinates {
          (0.261008,7.07697e-06)
          (0.130504,2.35071e-07)
          (0.0652519,8.12067e-09)
          (0.032626,2.96877e-10)
        };
        \addplot[color=teal,mark=diamond] coordinates {
          (0.261008,6.77732e-06)
          (0.130504,2.36303e-07)
          (0.0652519,8.15972e-09)
          (0.032626,2.97511e-10)
        };
        \addplot[color=pink,mark=pentagon] coordinates {
          (0.261008,7.88205e-06)
          (0.130504,4.81159e-07)
          (0.0652519,1.87857e-08)
          (0.032626,2.51523e-09)
        };
        \addplot[color=cyan,mark=star] coordinates {
          (0.261008,6.74206e-06)
          (0.130504,2.42861e-07)
          (0.0652519,7.96989e-09)
          (0.032626,2.53804e-10)
        };
        \addplot[thick,color=black,no markers] coordinates {
          (0.130504, 2.51095e-10)
          (0.261008, 2.51095e-10)
          (0.261008, 8.03504e-09)
          (0.130504, 2.51095e-10)
        };
        \node [anchor=south east] at (0.184561,1.42041e-09) {$5$};
      \end{loglogaxis}
    \end{tikzpicture}
  \end{tabular}
  
  \begin{tabular}{c}
    \begin{tikzpicture} 
      \begin{axis}[%
          hide axis,
          xmin=10,
          xmax=50,
          ymin=0,
          ymax=0.4,
          legend style={draw=white!15!black,legend cell align=left},
          legend columns=4
        ]
        
        \addlegendimage{color=red,mark=x}
        \addlegendentry{\Stab{1a}}
        
        \addlegendimage{color=green,mark=+}
        \addlegendentry{\Stab{2a}}
        
        \addlegendimage{color=blue,mark=o}
        \addlegendentry{\Stab{3a}}
        
        \addlegendimage{color=orange,mark=square}
        \addlegendentry{\Stab{4a}}
        
        \addlegendimage{color=violet,mark=triangle}
        \addlegendentry{\Stab{1b}}
        
        \addlegendimage{color=teal,mark=diamond}
        \addlegendentry{\Stab{2b}}
        
        \addlegendimage{color=pink,mark=pentagon}
        \addlegendentry{\Stab{3b}}
        
        \addlegendimage{color=cyan,mark=star}
        \addlegendentry{\Stab{4b}}
      \end{axis}
    \end{tikzpicture}
  \end{tabular}

  \caption{Test Case~2: convergence plots for $\Error{0}$
    using the polygonal meshes \MeshTwoA~and \MeshTwoB, and the four
    stabilization strategies \Stab{i}, $i=1,2,3,4$.
    Top row: plots for $k=1$ (left panel) and $k=2$ (right panel);
    right row: plots for $k=3$ (left panel) and $k=4$ (bottom
    panel).}
    \label{fig:n-side-L2}
\end{figure}
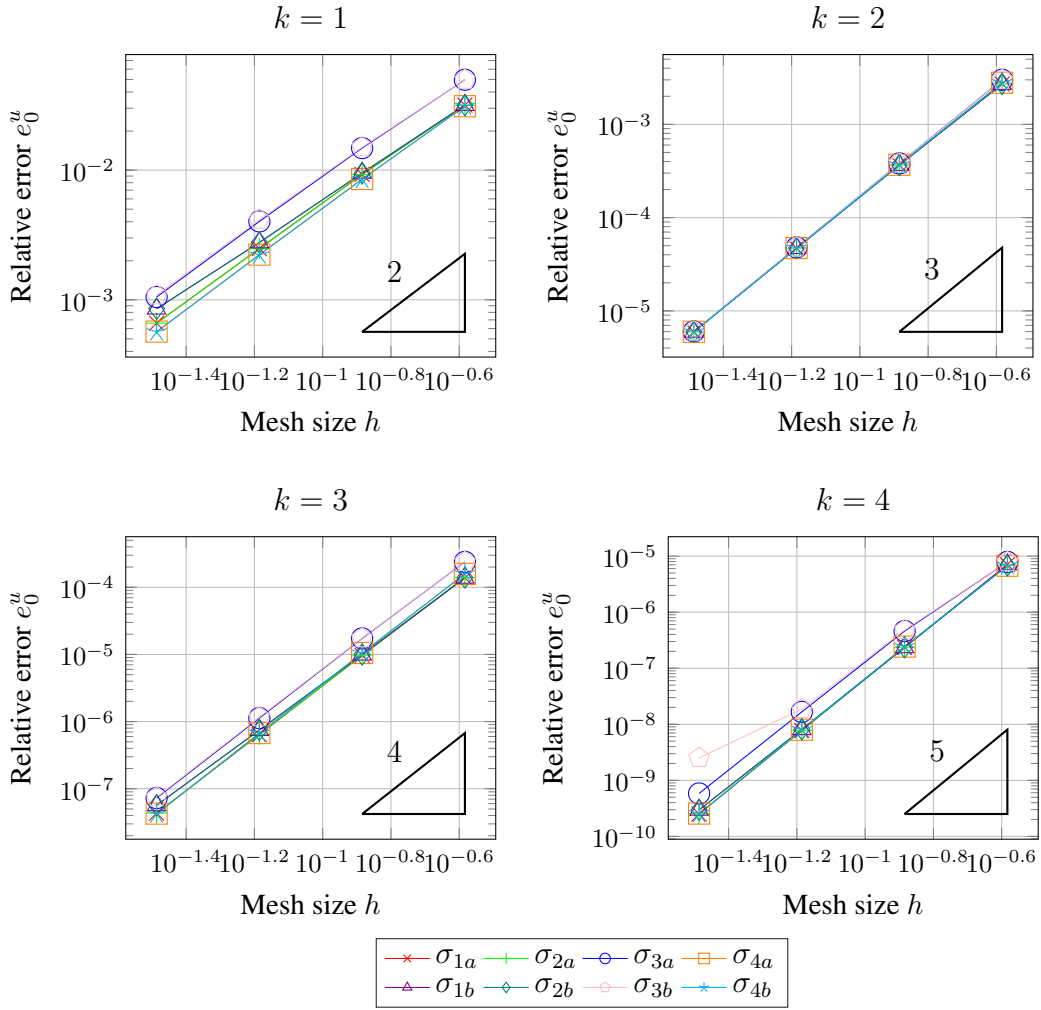


\begin{figure}\centering
  \begin{tabular}{cc}
    
    \begin{tikzpicture}[trim axis left]
      \begin{loglogaxis}
        [ mark size=4pt, grid=major, small,
          xlabel={Mesh size $\hh$},
          ylabel={Relative error $\Error{1}$},
          title ={$k=1$},
          legend columns=4, legend to name=n-side-H1-1 ]
        
        \addplot[color=red,mark=x] coordinates {
          (0.261008,0.217592)
          (0.130504,0.113399)
          (0.0652519,0.0576423)
          (0.032626,0.0290364)
        };
        \addplot[color=green,mark=+] coordinates {
          (0.261008,0.217592)
          (0.130504,0.113399)
          (0.0652519,0.0576423)
          (0.032626,0.0290364)
        };
        \addplot[color=blue,mark=o] coordinates {
          (0.261008,0.223077)
          (0.130504,0.114486)
          (0.0652519,0.0578262)
          (0.032626,0.0290608)
        };
        \addplot[color=orange,mark=square] coordinates {
          (0.261008,0.218289)
          (0.130504,0.11332)
          (0.0652519,0.0575912)
          (0.032626,0.0290151)
        };
        \addplot[color=violet,mark=triangle] coordinates {
          (0.261008,0.217592)
          (0.130504,0.113393)
          (0.0652519,0.0576742)
          (0.032626,0.0290538)
        };
        \addplot[color=teal,mark=diamond] coordinates {
          (0.261008,0.217592)
          (0.130504,0.113393)
          (0.0652519,0.0576742)
          (0.032626,0.0290538)
        };
        \addplot[color=pink,mark=pentagon] coordinates {
          (0.261008,0.223077)
          (0.130504,0.114558)
          (0.0652519,0.0578429)
          (0.032626,0.0290581)
        };
        \addplot[color=cyan,mark=star] coordinates {
          (0.261008,0.218289)
          (0.130504,0.11327)
          (0.0652519,0.0575631)
          (0.032626,0.0290032)
        };
        \addplot[thick,color=black,no markers] coordinates {
          (0.130504, 0.0290032)
          (0.261008, 0.0290032)
          (0.261008, 0.0580064)
          (0.130504, 0.0290032)
        };
        \node [anchor=south east] at (0.184561,0.0410167) {$1$};
        
      \end{loglogaxis}
    \end{tikzpicture}
    
    &
    
    \begin{tikzpicture}[trim axis right]
      \begin{loglogaxis}
        [ mark size=4pt, grid=major, small,
          xlabel={Mesh size $\hh$},
          ylabel={Relative error $\Error{1}$},
          title ={$k=2$},
          legend columns=4, legend to name=n-side-H1-2 ]
        
        \addplot[color=red,mark=x] coordinates {
          (0.261008,0.0239179)
          (0.130504,0.00628504)
          (0.0652519,0.00160673)
          (0.032626,0.000405871)
        };
        \addplot[color=green,mark=+] coordinates {
          (0.261008,0.0238991)
          (0.130504,0.00628501)
          (0.0652519,0.00160664)
          (0.032626,0.000405853)
        };
        \addplot[color=blue,mark=o] coordinates {
          (0.261008,0.0245813)
          (0.130504,0.00631043)
          (0.0652519,0.0016005)
          (0.032626,0.000402448)
        };
        \addplot[color=orange,mark=square] coordinates {
          (0.261008,0.0239095)
          (0.130504,0.00628269)
          (0.0652519,0.0016059)
          (0.032626,0.000405411)
        };
        \addplot[color=violet,mark=triangle] coordinates {
          (0.261008,0.0239179)
          (0.130504,0.00627635)
          (0.0652519,0.00160468)
          (0.032626,0.000405498)
        };
        \addplot[color=teal,mark=diamond] coordinates {
          (0.261008,0.0238991)
          (0.130504,0.00627625)
          (0.0652519,0.00160466)
          (0.032626,0.000405496)
        };
        \addplot[color=pink,mark=pentagon] coordinates {
          (0.261008,0.0245813)
          (0.130504,0.00630154)
          (0.0652519,0.00159398)
          (0.032626,0.000401183)
        };
        \addplot[color=cyan,mark=star] coordinates {
          (0.261008,0.0239095)
          (0.130504,0.00627766)
          (0.0652519,0.00160354)
          (0.032626,0.000405032)
        };
        \addplot[thick,color=black,no markers] coordinates {
          (0.130504, 0.000401183)
          (0.261008, 0.000401183)
          (0.261008, 0.00160473)
          (0.130504, 0.000401183)
        };
        \node [anchor=south east] at (0.184561,0.000802366) {$2$};
        
      \end{loglogaxis}
    \end{tikzpicture}

    \\[0.75em]

    \begin{tikzpicture}[trim axis left]
      \begin{loglogaxis}
        [ mark size=4pt, grid=major, small,
          xlabel={Mesh size $\hh$},
          ylabel={Relative error $\Error{1}$},
          title={$k=3$},
          legend columns=4, legend to name=n-side-H1-3 ]
        
        \addplot[color=red,mark=x] coordinates {
          (0.261008,0.00182354)
          (0.130504,0.000240633)
          (0.0652519,3.08509e-05)
          (0.032626,3.93661e-06)
        };
        \addplot[color=green,mark=+] coordinates {
          (0.261008,0.00181789)
          (0.130504,0.000240679)
          (0.0652519,3.08729e-05)
          (0.032626,3.94011e-06)
        };
        \addplot[color=blue,mark=o] coordinates {
          (0.261008,0.00216238)
          (0.130504,0.000297862)
          (0.0652519,3.83933e-05)
          (0.032626,4.91206e-06)
        };
        \addplot[color=orange,mark=square] coordinates {
          (0.261008,0.00181571)
          (0.130504,0.000240438)
          (0.0652519,3.08621e-05)
          (0.032626,3.90603e-06)
        };
        \addplot[color=violet,mark=triangle] coordinates {
          (0.261008,0.00182354)
          (0.130504,0.000239782)
          (0.0652519,3.16025e-05)
          (0.032626,4.2729e-06)
        };
        \addplot[color=teal,mark=diamond] coordinates {
          (0.261008,0.00181789)
          (0.130504,0.000239943)
          (0.0652519,3.16272e-05)
          (0.032626,4.27417e-06)
        };
        \addplot[color=pink,mark=pentagon] coordinates {
          (0.261008,0.00216238)
          (0.130504,0.000299436)
          (0.0652519,3.92371e-05)
          (0.032626,5.01004e-06)
        };
        \addplot[color=cyan,mark=star] coordinates {
          (0.261008,0.00181571)
          (0.130504,0.000239535)
          (0.0652519,3.07324e-05)
          (0.032626,3.89132e-06)
        };
        \addplot[thick,color=black,no markers] coordinates {
          (0.130504, 3.89132e-06)
          (0.261008, 3.89132e-06)
          (0.261008, 3.11306e-05)
          (0.130504, 3.89132e-06)
        };
        \node [anchor=south east] at (0.184561,1.10063e-05) {$3$};
        
      \end{loglogaxis}
    \end{tikzpicture}
    
    &
    
    \begin{tikzpicture}[trim axis right]
      \begin{loglogaxis}
        [ mark size=4pt, grid=major, small,
          xlabel={Mesh size $\hh$},
          ylabel={Relative error $\Error{1}$},
          title ={$k=4$},
          legend columns=4, legend to name=n-side-H1-4 ]
        
        \addplot[color=red,mark=x] coordinates {
          (0.261008,0.000110156)
          (0.130504,7.2749e-06)
          (0.0652519,4.69044e-07)
          (0.032626,2.97719e-08)
        };
        \addplot[color=green,mark=+] coordinates {
          (0.261008,0.000107616)
          (0.130504,7.24572e-06)
          (0.0652519,4.68382e-07)
          (0.032626,2.97582e-08)
        };
        \addplot[color=blue,mark=o] coordinates {
          (0.261008,0.000110217)
          (0.130504,9.36189e-06)
          (0.0652519,6.3665e-07)
          (0.032626,4.31286e-08)
        };
        \addplot[color=orange,mark=square] coordinates {
          (0.261008,0.000105754)
          (0.130504,7.19066e-06)
          (0.0652519,4.66608e-07)
          (0.032626,2.96666e-08)
        };
        \addplot[color=violet,mark=triangle] coordinates {
          (0.261008,0.000110156)
          (0.130504,7.20217e-06)
          (0.0652519,4.66296e-07)
          (0.032626,3.02937e-08)
        };
        \addplot[color=teal,mark=diamond] coordinates {
          (0.261008,0.000107616)
          (0.130504,7.19082e-06)
          (0.0652519,4.66409e-07)
          (0.032626,3.03065e-08)
        };
        \addplot[color=pink,mark=pentagon] coordinates {
          (0.261008,0.000110217)
          (0.130504,9.60271e-06)
          (0.0652519,6.98151e-07)
          (0.032626,4.67841e-08)
        };
        \addplot[color=cyan,mark=star] coordinates {
          (0.261008,0.000105754)
          (0.130504,7.15939e-06)
          (0.0652519,4.64502e-07)
          (0.032626,2.95691e-08)
        };
        \addplot[thick,color=black,no markers] coordinates {
          (0.130504, 2.95691e-08)
          (0.261008, 2.95691e-08)
          (0.261008, 4.73106e-07)
          (0.130504, 2.95691e-08)
        };
        \node [anchor=south east] at (0.184561,1.18276e-07) {$4$};
        
      \end{loglogaxis}
    \end{tikzpicture}
  \end{tabular}

  \begin{tabular}{c}
    \begin{tikzpicture} 
      \begin{axis}[%
          hide axis,
          xmin=10,
          xmax=50,
          ymin=0,
          ymax=0.4,
          legend style={draw=white!15!black,legend cell align=left},
          legend columns=4
        ]
        
        \addlegendimage{color=red,mark=x}
        \addlegendentry{\Stab{1a}}
        
        \addlegendimage{color=green,mark=+}
        \addlegendentry{\Stab{2a}}
        
        \addlegendimage{color=blue,mark=o}
        \addlegendentry{\Stab{3a}}
        
        \addlegendimage{color=orange,mark=square}
        \addlegendentry{\Stab{4a}}
        
        \addlegendimage{color=violet,mark=triangle}
        \addlegendentry{\Stab{1b}}
        
        \addlegendimage{color=teal,mark=diamond}
        \addlegendentry{\Stab{2b}}
        
        \addlegendimage{color=pink,mark=pentagon}
        \addlegendentry{\Stab{3b}}
        
        \addlegendimage{color=cyan,mark=star}
        \addlegendentry{\Stab{4b}}
      \end{axis}
    \end{tikzpicture}
  \end{tabular}

  \caption{Test Case~2: convergence plots for  $\Error{1}$
    using the polygonal meshes \MeshTwoA~and \MeshTwoB, and the four
    stabilization strategies \Stab{i}, $i=1,2,3,4$.
    Top row: plots for $k=1$ (left panel) and $k=2$ (bottom panel);
    bottom row: plots for $k=3$ (left panel) and $k=4$ (bottom
    panel).}
  \label{fig:n-side-H1}
\end{figure}
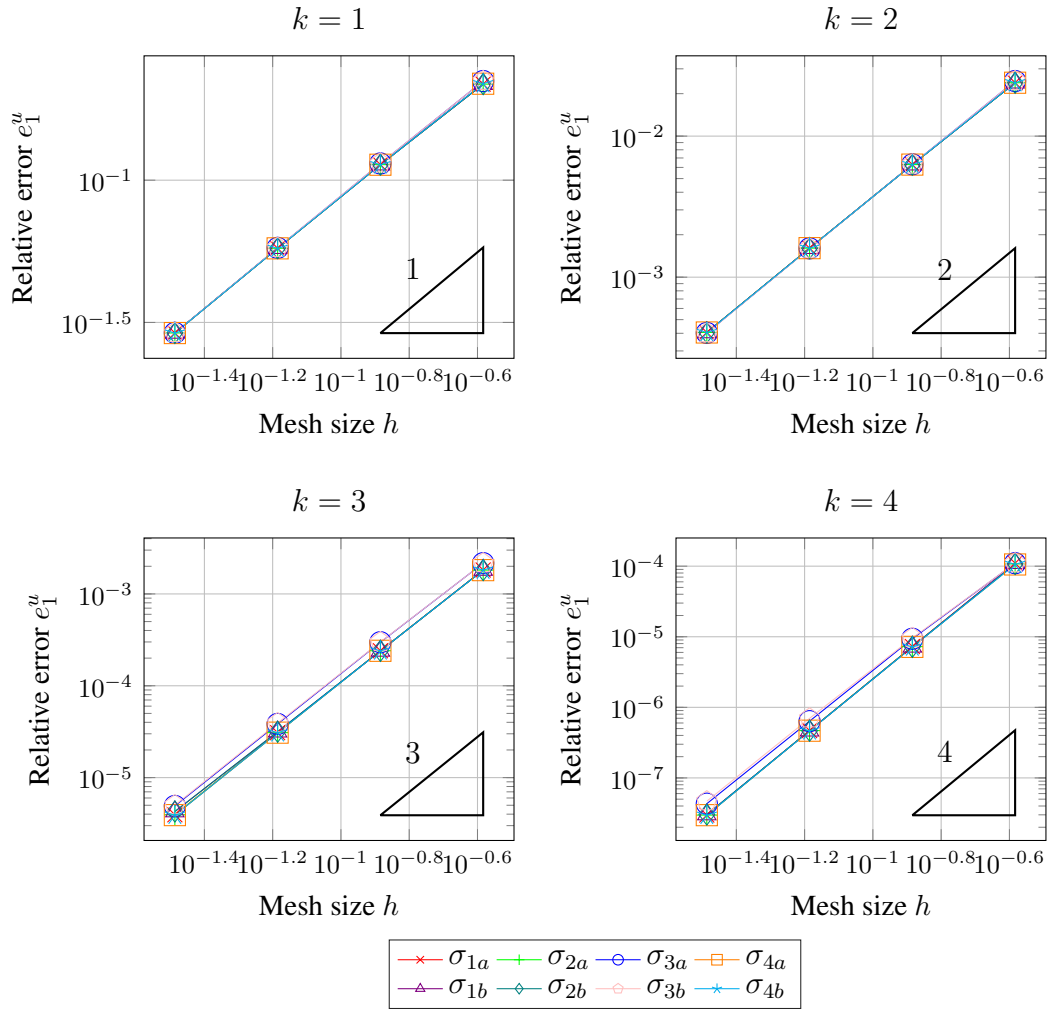


Figures~\ref{fig:n-side-L2} and \ref{fig:n-side-H1} show the
convergence plots of $\Error{0}$ and $\Error{1}$ that we obtain with
the VEM and the stabilizations \Stab{i}, $i=1,2,3,4$.

As expected, the errors $\Error{0}$ and $\Error{1}$ behave like
$\calO(h^{k+1})$ and $\calO(h^{k})$, respectively.
However, we note that there is some loss of accuracy for $k=1,3$ and
$4$ when we use the stabilizations \Stab{1} and \Stab{2}.
On the other hand, the VEM with stabilizations \Stab{3} and \Stab{4}
behaves similarly for $k=1,2$ and $3$ but for $k=4$, \Stab{3} induces
a visible loss in the convergence rate, possibly due to the effect of
round-off.
Instead, all versions of the VEM exhibit a similar behavior for
$\Error{1}$.

\subsubsection*{Test Case 3}

\begin{table} 
  \centering
  \begin{tabular}{
      c
      S[table-format=7.0]
      S[table-format=7.0]
      S[table-format=3.{\roundPrecision}e-1]
      S[table-format=3.{\roundPrecision}e-1]
      S[table-format=3.{\roundPrecision}e+1]
    }
    \toprule
        {Mesh} & {$\Nel$} & {$\Ned$} & {$\hh$} & {$\widehat{\hh}$} & {$\gamh$}\\
        \midrule
        1 &   64 &    576 & 1.767767e-01 & 3.125000e-02 & 5.656854e+00\\
        2 &  256 &   4352 & 8.838835e-02 & 7.812500e-03 & 1.131371e+01\\
        3 & 1024 &  33792 & 4.419417e-02 & 1.953125e-03 & 2.262742e+01\\
        4 & 4096 & 266240 & 2.209709e-02 & 4.882812e-04 & 4.525483e+01\\
        \bottomrule
  \end{tabular}
  \caption{Test Case~3: data of mesh family \MeshThree.}
  \label{tab:dyadic}
\end{table}

In the third test case, we consider the family of meshes shown in
Figure~\ref{fig:dyadic}.
At each mesh refinement step, the number of edges per element is an
increasing power of two, so that we can take the auxiliary grid
$\gridauxP$ as the grid whose elements are the edges in $\EdgesP$.
The data for these meshes are reported in Table~\ref{tab:dyadic}.



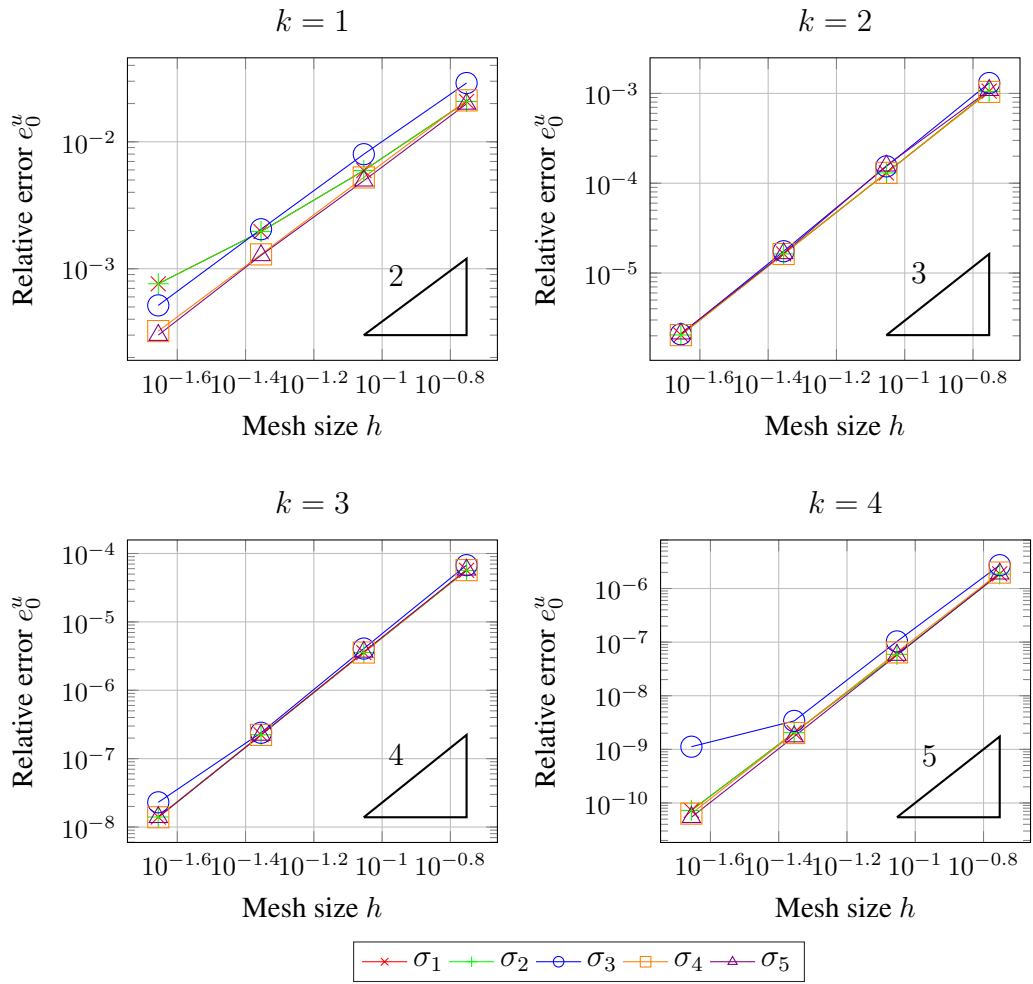
\begin{figure}\centering
  \begin{tabular}{rl}

    \begin{tikzpicture}[trim axis left]
      \begin{loglogaxis}
        [ mark size=4pt, grid=major, small,
          xlabel={Mesh size $\hh$},
          ylabel={Relative error $\Error{0}$},
          title ={$k=1$},
          legend columns=3, legend to name=d-square-L2-1 ]
        
        \addplot[color=red,mark=x] coordinates {
          (0.176777,0.020755)
          (0.0883884,0.0059456)
          (0.0441942,0.0019618)
          (0.0220971,0.000761046)
        };
        \addplot[color=green,mark=+] coordinates {
          (0.176777,0.020755)
          (0.0883884,0.0059456)
          (0.0441942,0.0019618)
          (0.0220971,0.000761046)
        };
        \addplot[color=blue,mark=o] coordinates {
          (0.176777,0.0290303)
          (0.0883884,0.00798594)
          (0.0441942,0.00204697)
          (0.0220971,0.000514381)
        };
        \addplot[color=orange,mark=square] coordinates {
          (0.176777,0.0212007)
          (0.0883884,0.00523864)
          (0.0441942,0.0012975)
          (0.0220971,0.000322657)
        };
        \addplot[color=violet,mark=triangle] coordinates {
          (0.176777,0.0198967)
          (0.0883884,0.00493532)
          (0.0441942,0.00127452)
          (0.0220971,0.000299988)
        };
        \addplot[thick,color=black,no markers] coordinates {
          (0.0883884, 0.000299988)
          (0.176777, 0.000299988)
          (0.176777, 0.00119995)
          (0.0883884, 0.000299988)
        };
        \node [anchor=south east] at (0.125,0.000599977) {$2$};
        
      \end{loglogaxis}
    \end{tikzpicture}
    
    &
    
    \begin{tikzpicture}[trim axis right]
      \begin{loglogaxis}
        [ mark size=4pt, grid=major, small,
          xlabel={Mesh size $\hh$},
          ylabel={Relative error $\Error{0}$},
          title ={$k=2$},
          legend columns=3, legend to name=d-square-L2-2 ]
        
        \addplot[color=red,mark=x] coordinates {
          (0.176777,0.00106009)
          (0.0883884,0.00013019)
          (0.0441942,1.6344e-05)
          (0.0220971,2.04373e-06)
        };
        \addlegendentry{$S_1$}
        \addplot[color=green,mark=+] coordinates {
          (0.176777,0.0010507)
          (0.0883884,0.000130232)
          (0.0441942,1.63457e-05)
          (0.0220971,2.04377e-06)
        };
        \addlegendentry{$S_2$}
        \addplot[color=blue,mark=o] coordinates {
          (0.176777,0.00130448)
          (0.0883884,0.000153993)
          (0.0441942,1.74873e-05)
          (0.0220971,2.08157e-06)
        };
        \addlegendentry{$S_3$}
        \addplot[color=orange,mark=square] coordinates {
          (0.176777,0.00103867)
          (0.0883884,0.00013033)
          (0.0441942,1.62915e-05)
          (0.0220971,2.03546e-06)
        };
        \addlegendentry{$S_4$}
        \addplot[color=violet,mark=triangle] coordinates {
          (0.176777,0.00108639)
          (0.0883884,0.000156552)
          (0.0441942,1.656e-05)
          (0.0220971,2.11339e-06)
        };
        \addlegendentry{$S_5$}
        \addplot[thick,color=black,no markers] coordinates {
          (0.0883884, 2.03546e-06)
          (0.176777, 2.03546e-06)
          (0.176777, 1.62837e-05)
          (0.0883884, 2.03546e-06)
        };
        \node [anchor=south east] at (0.125,5.75716e-06) {$3$};
        
      \end{loglogaxis}
    \end{tikzpicture}

    \\[0.75em]
    
    \begin{tikzpicture}[trim axis left]
      \begin{loglogaxis}
        [ mark size=4pt, grid=major, small,
          xlabel={Mesh size $\hh$},
          ylabel={Relative error $\Error{0}$},
          title ={$k=3$},
          legend columns=3, legend to name=d-square-L2-3 ]
        
        \addplot[color=red,mark=x] coordinates {
          (0.176777,5.6814e-05)
          (0.0883884,3.558e-06)
          (0.0441942,2.23824e-07)
          (0.0220971,1.40223e-08)
        };
        \addplot[color=green,mark=+] coordinates {
          (0.176777,5.65947e-05)
          (0.0883884,3.56013e-06)
          (0.0441942,2.23872e-07)
          (0.0220971,1.40228e-08)
        };
        \addplot[color=blue,mark=o] coordinates {
          (0.176777,6.74914e-05)
          (0.0883884,4.07281e-06)
          (0.0441942,2.39146e-07)
          (0.0220971,2.29847e-08)
        };
        \addplot[color=orange,mark=square] coordinates {
          (0.176777,5.71052e-05)
          (0.0883884,3.56717e-06)
          (0.0441942,2.22552e-07)
          (0.0220971,1.39006e-08)
        };
        \addplot[color=violet,mark=triangle] coordinates {
          (0.176777,5.78512e-05)
          (0.0883884,3.64551e-06)
          (0.0441942,2.22288e-07)
          (0.0220971,1.39073e-08)
        };
        \addplot[thick,color=black,no markers] coordinates {
          (0.0883884, 1.39006e-08)
          (0.176777, 1.39006e-08)
          (0.176777, 2.22411e-07)
          (0.0883884, 1.39006e-08)
        };
        \node [anchor=south east] at (0.125,5.56025e-08) {$4$};
        
      \end{loglogaxis}
    \end{tikzpicture}
    
    &
    
    \begin{tikzpicture}[trim axis right]
      \begin{loglogaxis}
        [ mark size=4pt, grid=major, small,
          xlabel={Mesh size $\hh$},
          ylabel={Relative error $\Error{0}$},
          title ={$k=4$},
          legend columns=3, legend to name=d-square-L2-4 ]
        
        \addplot[color=red,mark=x] coordinates {
          (0.176777,1.8787e-06)
          (0.0883884,5.86472e-08)
          (0.0441942,2.04316e-09)
          (0.0220971,7.25621e-11)
        };
        \addplot[color=green,mark=+] coordinates {
          (0.176777,1.87155e-06)
          (0.0883884,5.91131e-08)
          (0.0441942,2.05176e-09)
          (0.0220971,7.26496e-11)
        };
        \addplot[color=blue,mark=o] coordinates {
          (0.176777,2.72579e-06)
          (0.0883884,1.03088e-07)
          (0.0441942,3.3906e-09)
          (0.0220971,1.12113e-09)
        };
        \addplot[color=orange,mark=square] coordinates {
          (0.176777,1.98533e-06)
          (0.0883884,6.43383e-08)
          (0.0441942,2.01624e-09)
          (0.0220971,6.33575e-11)
        };
        \addplot[color=violet,mark=triangle] coordinates {
          (0.176777,1.89004e-06)
          (0.0883884,5.74795e-08)
          (0.0441942,1.73337e-09)
          (0.0220971,5.41401e-11)
        };
        \addplot[thick,color=black,no markers] coordinates {
          (0.0883884, 5.41401e-11)
          (0.176777, 5.41401e-11)
          (0.176777, 1.73249e-09)
          (0.0883884, 5.41401e-11)
        };
        \node [anchor=south east] at (0.125,3.06264e-10) {$5$};
        
      \end{loglogaxis}
    \end{tikzpicture}
  \end{tabular}

  \begin{tabular}{c}
    \begin{tikzpicture} 
      \begin{axis}[%
          hide axis,
          xmin=10,
          xmax=50,
          ymin=0,
          ymax=0.4,
          legend style={draw=white!15!black,legend cell align=left},
          legend columns=5
        ]
        
        \addlegendimage{color=red,mark=x}
        \addlegendentry{\Stab{1}}
        
        \addlegendimage{color=green,mark=+}
        \addlegendentry{\Stab{2}}
        
        \addlegendimage{color=blue,mark=o}
        \addlegendentry{\Stab{3}}
        
        \addlegendimage{color=orange,mark=square}
        \addlegendentry{\Stab{4}}
        
        \addlegendimage{color=violet,mark=triangle}
        \addlegendentry{\Stab{5}}        
      \end{axis}
    \end{tikzpicture}
  \end{tabular}

  \caption{Test Case~3: convergence plots for  $\Error{0}$
    using the ``squared'' polygonal meshes \MeshThree, and the five
    stabilization strategies \Stab{i}, $i=1,2,3,4,5$.
    Top row: plots for $k=1$ (left panel) and $k=2$ (bottom panel);
    bottom row: plots for $k=3$ (left panel) and $k=4$ (bottom panel).}
  \label{fig:d-square-L2}
\end{figure}


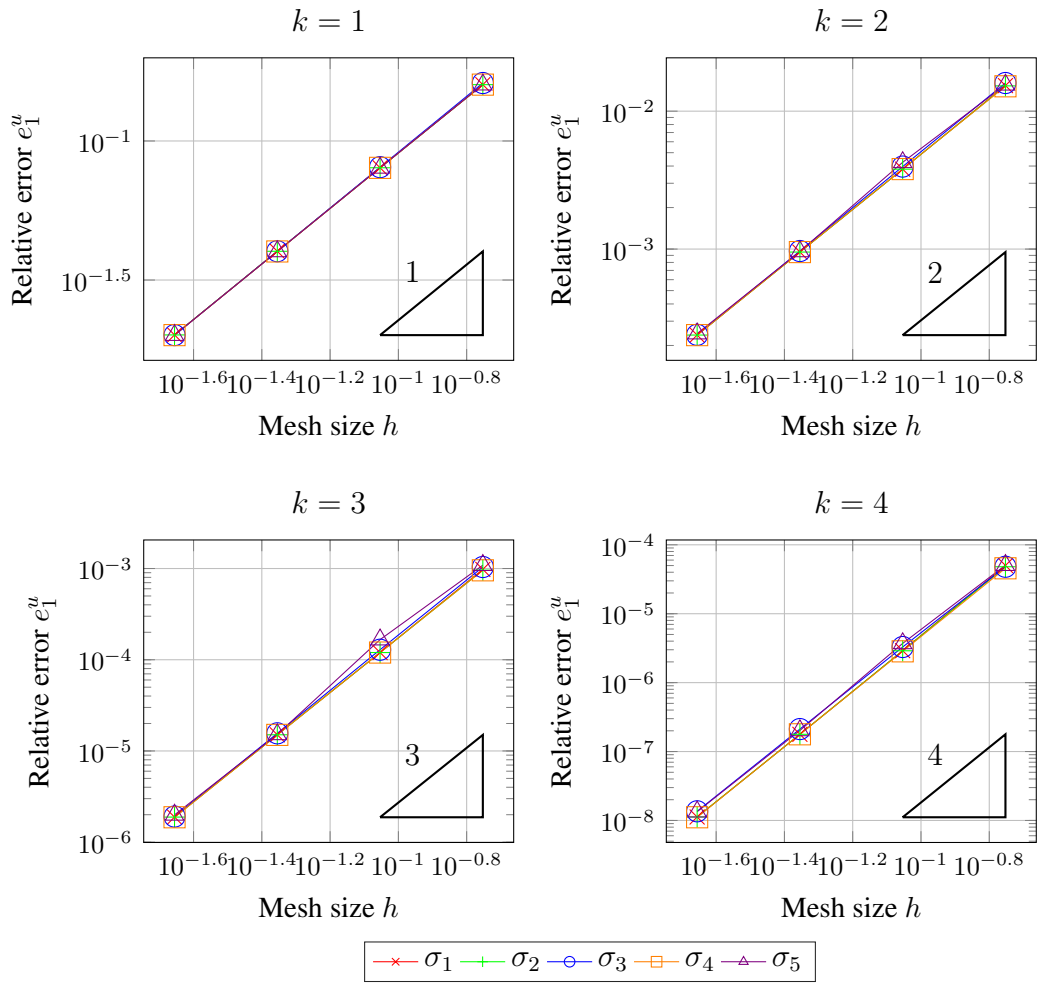
\begin{figure}\centering
  \begin{tabular}{rl}

    \begin{tikzpicture}[trim axis left]
      \begin{loglogaxis}
        [ mark size=4pt, grid=major, small,
          xlabel={Mesh size $\hh$},
          ylabel={Relative error $\Error{1}$},
          title ={$k=1$},
          legend columns=3, legend to name=d-square-H1-1 ]
        
        \addplot[color=red,mark=x] coordinates {
          (0.176777,0.159434)
          (0.0883884,0.0801138)
          (0.0441942,0.0401207)
          (0.0220971,0.0200715)
        };
        \addplot[color=green,mark=+] coordinates {
          (0.176777,0.159434)
          (0.0883884,0.0801138)
          (0.0441942,0.0401207)
          (0.0220971,0.0200715)
        };
        \addplot[color=blue,mark=o] coordinates {
          (0.176777,0.161582)
          (0.0883884,0.0805201)
          (0.0441942,0.0401342)
          (0.0220971,0.020047)
        };
        \addplot[color=orange,mark=square] coordinates {
          (0.176777,0.159465)
          (0.0883884,0.0800529)
          (0.0441942,0.0400663)
          (0.0220971,0.0200382)
        };
        \addplot[color=violet,mark=triangle] coordinates {
          (0.176777,0.159408)
          (0.0883884,0.0800454)
          (0.0441942,0.0400674)
          (0.0220971,0.0200381)
        };
        \addplot[thick,color=black,no markers] coordinates {
          (0.0883884, 0.0200381)
          (0.176777, 0.0200381)
          (0.176777, 0.0400762)
          (0.0883884, 0.0200381)
        };
        \node [anchor=south east] at (0.125,0.0283382) {$1$};    
      \end{loglogaxis}
    \end{tikzpicture}
    
    &
    
    \begin{tikzpicture}[trim axis right]
      \begin{loglogaxis}
        [ mark size=4pt, grid=major, small,
          xlabel={Mesh size $\hh$},
          ylabel={Relative error $\Error{1}$},
          title ={$k=2$},
          legend columns=3, legend to name=d-square-H1-2 ]
        
        \addplot[color=red,mark=x] coordinates {
          (0.176777,0.0153264)
          (0.0883884,0.00379668)
          (0.0441942,0.000951357)
          (0.0220971,0.000237947)
        };
        \addplot[color=green,mark=+] coordinates {
          (0.176777,0.0152559)
          (0.0883884,0.00379691)
          (0.0441942,0.000951383)
          (0.0220971,0.000237948)
        };
        \addplot[color=blue,mark=o] coordinates {
          (0.176777,0.0159959)
          (0.0883884,0.00396537)
          (0.0441942,0.000967786)
          (0.0220971,0.000238987)
        };
        \addplot[color=orange,mark=square] coordinates {
          (0.176777,0.0151478)
          (0.0883884,0.00380303)
          (0.0441942,0.000951061)
          (0.0220971,0.000237702)
        };
        \addplot[color=violet,mark=triangle] coordinates {
          (0.176777,0.0155332)
          (0.0883884,0.00425751)
          (0.0441942,0.00096129)
          (0.0220971,0.000243794)
        };
        \addplot[thick,color=black,no markers] coordinates {
          (0.0883884, 0.000237702)
          (0.176777, 0.000237702)
          (0.176777, 0.00095081)
          (0.0883884, 0.000237702)
        };
        \node [anchor=south east] at (0.125,0.000475405) {$2$};
        
      \end{loglogaxis}
    \end{tikzpicture}

    \\[0.75em]
    
    \begin{tikzpicture}[trim axis left]
      \begin{loglogaxis}
        [ mark size=4pt, grid=major, small,
          xlabel={Mesh size $\hh$},
          ylabel={Relative error $\Error{1}$},
          title ={$k=3$},
          legend columns=3, legend to name=d-square-H1-3 ]

        \addplot[color=red,mark=x] coordinates {
          (0.176777,0.000978814)
          (0.0883884,0.000119858)
          (0.0441942,1.50407e-05)
          (0.0220971,1.88403e-06)
        };
        \addplot[color=green,mark=+] coordinates {
          (0.176777,0.000959306)
          (0.0883884,0.000119883)
          (0.0441942,1.50426e-05)
          (0.0220971,1.88407e-06)
        };
        \addplot[color=blue,mark=o] coordinates {
          (0.176777,0.0010388)
          (0.0883884,0.000128809)
          (0.0441942,1.56057e-05)
          (0.0220971,1.91122e-06)
        };
        \addplot[color=orange,mark=square] coordinates {
          (0.176777,0.000955977)
          (0.0883884,0.000119933)
          (0.0441942,1.49997e-05)
          (0.0220971,1.87513e-06)
        };
        \addplot[color=violet,mark=triangle] coordinates {
          (0.176777,0.00108828)
          (0.0883884,0.000166146)
          (0.0441942,1.52021e-05)
          (0.0220971,1.9927e-06)
        };
        \addplot[thick,color=black,no markers] coordinates {
          (0.0883884, 1.87513e-06)
          (0.176777, 1.87513e-06)
          (0.176777, 1.50011e-05)
          (0.0883884, 1.87513e-06)
        };
        \node [anchor=south east] at (0.125,5.30368e-06) {$3$};
      \end{loglogaxis}
    \end{tikzpicture}
    
    &
    
    \begin{tikzpicture}[trim axis right]
      \begin{loglogaxis}
        [ mark size=4pt, grid=major, small,
          xlabel={Mesh size $\hh$},
          ylabel={Relative error $\Error{1}$},
          title ={$k=4$},
          legend columns=3, legend to name=d-square-H1-4 ]
        
        \addplot[color=red,mark=x] coordinates {
          (0.176777,5.01493e-05)
          (0.0883884,2.9098e-06)
          (0.0441942,1.77942e-07)
          (0.0220971,1.12321e-08)
        };
        \addplot[color=green,mark=+] coordinates {
          (0.176777,4.80403e-05)
          (0.0883884,2.89918e-06)
          (0.0441942,1.77935e-07)
          (0.0220971,1.12347e-08)
        };
        \addplot[color=blue,mark=o] coordinates {
          (0.176777,4.80583e-05)
          (0.0883884,3.28772e-06)
          (0.0441942,2.11908e-07)
          (0.0220971,1.35783e-08)
        };
        \addplot[color=orange,mark=square] coordinates {
          (0.176777,4.56278e-05)
          (0.0883884,2.84592e-06)
          (0.0441942,1.7792e-07)
          (0.0220971,1.11244e-08)
        };
        \addplot[color=violet,mark=triangle] coordinates {
          (0.176777,5.0622e-05)
          (0.0883884,3.71e-06)
          (0.0441942,2.02486e-07)
          (0.0220971,1.34022e-08)
        };
        \addplot[thick,color=black,no markers] coordinates {
          (0.0883884, 1.11244e-08)
          (0.176777, 1.11244e-08)
          (0.176777, 1.77991e-07)
          (0.0883884, 1.11244e-08)
        };
        \node [anchor=south east] at (0.125,4.44977e-08) {$4$};
      \end{loglogaxis}
    \end{tikzpicture}
    
  \end{tabular}

  \begin{tabular}{c}
    \begin{tikzpicture} 
      \begin{axis}[%
          hide axis,
          xmin=10,
          xmax=50,
          ymin=0,
          ymax=0.4,
          legend style={draw=white!15!black,legend cell align=left},
          legend columns=5
        ]
        
        \addlegendimage{color=red,mark=x}
        \addlegendentry{\Stab{1}}
        
        \addlegendimage{color=green,mark=+}
        \addlegendentry{\Stab{2}}
        
        \addlegendimage{color=blue,mark=o}
        \addlegendentry{\Stab{3}}
        
        \addlegendimage{color=orange,mark=square}
        \addlegendentry{\Stab{4}}
        
        \addlegendimage{color=violet,mark=triangle}
        \addlegendentry{\Stab{5}}        
      \end{axis}
    \end{tikzpicture}
  \end{tabular}

  \caption{Test Case~3: convergence plots for  $\Error{1}$
    using the ``squared'' polygonal meshes \MeshThree, and the five
    stabilization strategies \Stab{i}, $i=1,2,3,4,5$.
    Top row: plots for $k=1$ (left panel) and $k=2$ (bottom panel);
    bottom row: plots for $k=3$ (left panel) and $k=4$ (bottom
    panel).}
  \label{fig:d-square-H1}
\end{figure}

Figures~\ref{fig:d-square-L2} and \ref{fig:d-square-H1} show the
convergence plots for $\Error{0}$ and $\Error{1}$ that we obtain with
the stabilizations \Stab{i}, $i=1,2,3,4,5$.
We observe significative differences in the convergence rates for
$\Error{0}$.
In particular, for $k=1$, the stabilizations \Stab{1} and \Stab{2}
perform poorly.

For $k=4$, the performance of \Stab{3} is extremely poor on the finest
mesh.
On the other hand, all the convergence plots for $\Error{1}$ show the
optimal convergence rate proportional to $\calO(\hh^k)$ regardless of
the stabilization.

\medskip

The most robust stabilizations are $\sigma_4$ and, when this is
computed, $\sigma_5$ (for technical reasons, depending on the wavelet
implementation at our disposal, we only tested $\sigma_5$ on the
family \MeshThree).
However, the algorithm we used to compute the square-root of a matrix
in $\sigma_4$ (see~\cite{Higham:1987}) failed to run in some
experiments on hexagonal meshes with more extreme values of the ratio
$\hh\slash{\widehat{\hh}}$ than those of the \MeshTwo~family, due to
round-off errors.
More stable algorithms for computing the square-root of a matrix
should be considered (see, for example,
\cite{Deadman-Higham-Ralha:2013}).

\section{Conclusions}
\label{sec:conclusions}

We studied a novel approach to designing computable stabilizing
bilinear forms for the nonconforming virtual element method, based on
the duality technique first introduced in
\cite{Bertoluzza_algebraic_dual}. This consists in transfering the
definition of the bilinear form from the local virtual element space
to the dual space spanned by the functionals yielding the degrees of
freedom. In such a way we could overcome the difficulty posed by the
fact that, in the non conforming framework, the shape functions are
non computable (not even as far as their trace on the boundary of the
elements is concerned), and that the only information to which we have
access along the computation are the values of the degrees of freedom.
By applying this novel technique, we built new bilinear forms with
optimal or quasi-optimal stability bounds, under assumptions on the
mesh which are weaker than the ones usually made in the analysis of
the virtual element method, and which allow a mesh to have a very
large number of arbitrarily small edges per element.  The resulting
discretization of second-order elliptic problems is accurate and
robust, and allows for optimal or quasi-optimal error bounds. Finally,
we numerically investigated the behavior of a non conforming VEM,
implementing several examples of these new stabilization forms, and we
assessed its performance on a set of representative test cases. The
results of the numerical experiments confirmed the theoretical
expectations.

\section*{Acknowledgements}
This paper has been realised in the framework of ERC Project CHANGE,
which has received funding from the European Research Council (ERC)
under the European Union’s Horizon 202 (grant agreement no.~694515),
and of the project ``Virtual Element Methods: Analysis and
Applications", funded by the MIUR Progetti di Ricerca di Rilevante
Interesse Nazionale (PRIN) Bando 2017 (grant 201744KLJL).


\appendix

\section{Proof of Lemma \ref{lemma:enhanced}}\label{appendix:enhanced}

Let $\vsh \in \VhkP$ and $\vshh \in \VhkenP$ satisfy
\eqref{eq:cond:v:vhat}. We recall that such a condition implies that
$\PinP{k} \vsh = \PinP{k} \vshh$. As $\Delta \vsh \in \Poly_{k-2}(\P)$
and $(\nabla\vsh \cdot {\nor_\P})_{|e} \in \Poly_{k-1}(e)$ for all $e
\in \EdgesP$ we have
\begin{align*}
  \SNORM{\vsh}{1,\P}^2
  &=  \int_{\P }\ABS{\nabla\vsh}^2\dV
  =  -\int_{\P}\Delta\vsh\,\vsh\dV  + \int_{\bP}\nabla\vsh\cdot\norP\vsh\dV
  =  -\int_{\P}\Delta\vsh\,\vshh\dV + \int_{\bP}\nabla\vsh\cdot\norP\vshh\dV\\[0.5em]
  &=  \int_{\P}\nabla\vsh\cdot\nabla\vshh\dV
  \leq \SNORM{\vsh}{1,\P}\,\SNORM{\vshh}{1,\P}.
\end{align*}
We divide both sides by $\SNORM{\vsh}{1,\P}$ and obtain the upper
bound. 
On the other hand, we observe that, as $\Delta(\vshh - \PinP{k} \vshh) \in \Poly_k$, it can be split as
\[
\Delta(\vshh - \PinP{k} \vshh) = \Dlow(\vshh) + \Dhigh(\vshh)
\]
with $\Dlow(\vshh) \in \Poly_{k-2}(\P)$ and $\Dhigh(\vshh)$ belonging to the linear space spanned by $\monomials{k}{\P} \setminus \monomials{k-2}{\P}$, and that we have
\[
\| \Delta(\vshh - \PinP{k} \vshh) \|_{0,\P} \simeq \| \Dlow(\vshh) \|_{0,\P} + \|  \Dhigh(\vshh)\|_{0,\P}.
\]
We observe that, in view of the definition of the enhanced space, we have that
\[
\int_{\P}\Dhigh(\vshh) (\vshh - \PinP{k} \vshh) = 	\int_{\P}\Dhigh(\vshh) (\PinP{k} \vshh - \PinP{k} \vshh) = 0.
\]
Then	we can write:
\begin{multline*}
  \int_{\P} | \nabla(\vshh - \PinP{k} \vshh) |^2 = - \int_{\P} \Delta (\vshh - \PinP{k} \vshh)(\vshh - \PinP{k} \vshh) + \int_{\bP} \nabla(\vshh - \PinP{k} \vshh) \cdot \nor_\P (\vshh - \PinP{k} \vshh) = \\
  =- \int_{\P} \Dlow(\vshh) (\vshh - \PinP{k} \vshh)  + \int_{\bP} \nabla(\vshh - \PinP{k} \vshh) \cdot \nor_\P (\vsh - \PinP{k} \vsh) \\
  =- \int_{\P} \Dlow(\vshh) (\vsh - \PinP{k} \vsh) + 
  \int_{\P} \Delta(\vshh - \PinP{k} \vshh)(\vsh - \PinP{k} \vsh)\\ + \int_{\P}\nabla(\vshh - \PinP{k} \vshh) \cdot \nabla (\vsh - \PinP{k} \vsh).
\end{multline*}
Then we have
\[
\int_{\P} | \nabla(\vshh - \PinP{k} \vshh) |^2 \lesssim \| \Delta(\vshh - \PinP{k} \vshh) \|_{0,\P} \| \vsh - \PinP{k} \vsh \|_{0,\P} + |\vshh - \PinP{k} \vshh|_{1,\P} | \vsh - \PinP{k} \vsh |_{1,P}
\]
Using Lemma \ref{lem:inverse} and a Poincar\'e inequality finally yields
\[
\int_{\P} | \nabla(\vshh - \PinP{k} \vshh) |^2 \lesssim  |\vshh - \PinP{k} \vshh|_{1,\P} | \vsh - \PinP{k} \vsh |_{1,P}.
\]
Dividing both sides by $ |\vshh - \PinP{k} \vshh|_{1,\P} $ and using a triangular inequality yields the lower bound.



\section{Proof of Lemma \ref{lem:boundbelow}}\label{appendix:technical}
Let $\eta\in \Nk{k-1}(\bP)$ with $\int_{\bP}\eta\dS =0$.  Let
$\PS{k+1}^0(\E)=\PS{k+1}(\E)\cap\HONEzr(\E)$ and observe that
\begin{align}
  \inf_{\eta \in\PS{k-1}(\E)}\sup_{\qs\in\PS{k+1}^0(\E)}
  \frac{ \int_{\E}\eta\qs\dS }{\NORM{\eta}{0,\E}\,\NORM{\qs}{0,\E}}\gtrsim1.
  \label{eq:infsupe}
\end{align}
Relation~\eqref{eq:infsupe} can be proven on the reference interval
$\widehat{\E}=[0,1]$ by noting that the Riesz isomorphism between
$\HS{-1}(\E)$ and $\HONEzr(\E)$ maps $\PS{k-1}(\widehat{\E})$ to
$\PS{k+1}^0(\widehat{\E})$, and that all norms are the equivalent on
such finite dimensional spaces.
Then, we apply a scaling argument to obtain \eqref{eq:infsupe} for a
generic edge $\E$.
This implies that for every edge $\E \in \EdgesP$, a function
$\phi_{\E}(\eta)\in\PS{k+1}^0(\E)$ exists such that
\begin{align*}
  \NORM{\eta}{0,\E}^2  = \int_{\E}\eta\phi_\E(\eta)\dS,
  \qquad
  \NORM{\phi_{\E}(\eta)}{0,\E} \simeq \NORM{\eta}{0,\E}.
\end{align*}
We let $\phi\in\HS{\frac12}(\bP)$ denote the function satisfying
$\restrict{\phi}{\E}=\hE\phi_{\E}(\eta)$ for all $\E \in \EdgesP$, and
we write
\begin{align}\label{N1}
  \sum_\E\hE\NORM{\eta}{0,\E}^2
  \leq \sum_\E\hE\int_{\E}\eta\phi_{\E}(\eta)\dS
  =    \int_{\bP}\eta\phi\dS.
\end{align}
As $\int_{\bP}\eta\dS=0$, for
$\bar\phi=\fint_{\bP}\phi\dS$ we can write
\begin{align}\label{N2}
  \int_{\bP}\eta\phi\dS
  = \int_{\bP}\eta (\phi - \bar\phi)\dS
  \lesssim \SNORM{\eta}{-1/2,\bP}\,\SNORM{\phi - \bar\phi}{1/2,\bP}
  =        \SNORM{\eta}{-1/2,\bP}\,\SNORM{\phi}{1/2,\bP} .
\end{align}

It remains to bound $\snorm{\phi}{1/2,\bP}$.
First, we split $\phi$ as $\phi=\phi_1+\phi_2$ where $\phi_1$ is
supported on the edges of $\EdgesPp$ and $\phi_2$ on the edges of
$\EdgesPpp$.
Let $\widehat{\phi}_{\E}$ denote the pullback of $\phi_{\E}$ on the
reference edge $\widehat{\E}=[0,1]$.
We use again a scaling argument and the equivalence of all norms on
the finite dimensional space $\PS{k+1}^0(\widehat{\E})$ to find
that
\begin{align*}
  \SNORM{\phi_2}{1/2,\bP} &
  \lesssim \sum_{\E\in\EdgesPpp} \NORM{\hE\phi_\E(\eta)}{\HS{\frac12}_{00}(\E)}
  =        \sum_{\E\in\EdgesPpp} \hE\NORM{\widehat{\phi}_\E}{\HS{\frac12}_{00}(\widehat{\E})}
  \lesssim \sum_{\E\in\EdgesPpp} \hE\NORM{\widehat{\phi}_\E}{0,\widehat{\E}}\\
  &=       \sum_{\E\in\EdgesPpp} \hE^{1/2}\NORM{\phi_\E(\eta)}{0,\E}
  \lesssim \sqrt{N^*} \left( \sum_{\E\in\EdgesPpp} \hE\NORM{\phi_{\E}(\eta)}{0,\E}^2\right)^{1/2} 
  \leq     \sqrt{N^*} \left( \sum_{\E\in\EdgesPpp} \hE\NORM{\eta}{0,\E}^2\right )^{1/2},
\end{align*}
where we recall that, for an edge $\E \in \EdgesP$, the space
$H^{\frac12}_{00}(e)$ is the space of functions $\eta$ in
$H^{\frac12}(\E)$ such that the function $\Ext \eta \in L^2(\bP)$
satisfying $\Ext \eta_{|e} = \eta$ and $\Ext \eta_{|\bP\setminus e} =
0$ is in $H^{\frac12}(\bP)$, endowed with the norm $\| \eta
\|_{H^{\frac12}_{00}(e)} = | \Ext \eta |_{1/2,\bP}$.

To bound $\SNORM{\phi_1}{1/2,\bP}$, we proceed as in
\cite{Faermann:2000}, taking advantage that the grid is locally quasi
uniform on the support of $\phi_1$.
Let $\tEe$ denote the patch given by the union of $\E\in\EdgesP$ and
its two neighboring edges.
Then, we have that
\begin{align}\label{eq:app:1}
  \SNORM{\phi_1}{1/2,\bP}^2
  = \sum_{\E\in\EdgesP}\int_{\E}
  \left[
    \int_{\tEe}\frac{\ABS{\phi_1(x) - \phi_1(y)}^2}{\ABS{x-y}^2}\,dx\,dy
    + \int_{\bP\setminus\tEe}\frac{\ABS{\phi_1(x)-\phi_1(y)}^2}{\ABS{x-y}^2}\,dx\,dy
    \right].
\end{align}
With our definition of $\EdgesPp$ and $\phi_1$, we see that
\begin{align*}
  \sum_{\E\in\EdgesP} \int_{\E}\int_{\tEe} \frac{\ABS{\phi_1(x) - \phi_1(y)}^2}{\ABS{x-y}^2}\,dx\,dy
  \lesssim \sum_{\E\in\EdgesPp} \int_{\tEe}\int_{\tEe} \frac{\ABS{\phi_1(x) - \phi_1(y)}^2}{\ABS{x-y}^2}\,dx\,dy
  = \sum_{\E\in\EdgesPp} \SNORM{\phi_1}{1/2,\tEe}^2.
\end{align*}
Assumption \TERM{G3.1} allows us to use an
inverse inequality on $\tEe$, which yields
\begin{align*}
  \sum_{\E\in\EdgesPp} \SNORM{\phi_1}{1/2,\tEe}^2
  \lesssim \sum_{\E\in\EdgesPp}  \sum_{\Ep\subset\tEe}\hh_{\Ep}^{-1}\NORM{\phi_1}{0,\Ep}^2
  =     \sum_{\E\in\EdgesPp}     \hh_{\E }  \NORM{\eta}{0,\E}^2.
\end{align*}
On the other hand, we can write
\begin{align*}
  &\sum_{\E\in\EdgesP} \int_{\E}\int_{\bP\setminus\tEe}\frac{\ABS{\phi_1(x) - \phi_1(y)}^2}{\ABS{x-y}^2}\,dx\,dy\\
  &\qquad\qquad\lesssim
  \sum_{\E\in\EdgesP} \int_{\E}\int_{\bP\setminus\tEe}\frac{\ABS{\phi_1(x)}^2}{\ABS{x-y}^2}\,dx\,dy +
  \sum_{\E\in\EdgesP} \int_{\E}\int_{\bP\setminus\tEe}\frac{\ABS{\phi_1(y)}^2}{\ABS{x-y}^2}\,dx\,dy \\
  &\qquad\qquad= 2 \sum_{\E\in\EdgesPp} \int_{\E} \ABS{\phi_1(x)}^2\left(\int_{\bP\setminus\tEe}\frac{1}{\ABS{x-y}^2}\,dy\right)\,dx, 
\end{align*}
where the second term of the sum in the second step can be seen to be
equal to first one by splitting the integral in $y$ over the union of
edges in $\EdgesPp\setminus\tEe$ and switching the two integrals.
By direct calculation, under our assumptions, we find the bound
\begin{align*}
  \int_{\bP\setminus\tEe}\frac{1}{\ABS{x-y}^2}\,dy
  \lesssim \hE^{-1},
\end{align*}
finally yielding
\begin{align*}
  \sum_{\E\in\EdgesP} \int_{\E}\int_{\bP\setminus\tEe}\frac{\ABS{\phi_1(x) - \phi_1(y)}^2}{\ABS{x-y}^2}\,dx\,dy
  \lesssim \sum_{\E\in\EdgesPp}\hE^{-1}\NORM{\phi_1}{0,\E}^2
  \lesssim \sum_{\E\in\EdgesPp}\hE    \NORM{\eta}{0,\E}^2. 
\end{align*}
Collecting the contributions of $\phi_1$ and $\phi_2$ we finally have that
\[
| \phi |^2_{1/2,\bP} \lesssim \sum_{e \in \EdgesP} \hh_{\E} \| \eta \|^2_{0,\E}.
\]
Substituting such a bound in \eqref{N2} and using the result in \eqref{N1} we then write
\begin{align*}
  \sum_{\E} \hE\NORM{\eta}{0,\E}^2 = \int_{\bP} \eta \phi 
  \lesssim \SNORM{\eta}{-1/2,\bP}\left( \sum_{\E} \hE\norm{\eta}{0,\E}^2 \right)^{1/2},
\end{align*}
and dividing both sides by the square root of
$\sum_{\E}\hE\NORM{\eta}{0,\E}^2$ yields
\begin{align*}
  \SNORM{\eta}{-1/2,\bP} \gtrsim
  \left( \sum_{\E}\hE\NORM{\eta}{0,\E}^2 \right)^{1/2},
\end{align*}
which concludes the proof.

\section{The Steinbach projector}
\label{appendix:Steinbach}
\begin{figure}
  \centering
  \includegraphics[height=5cm]{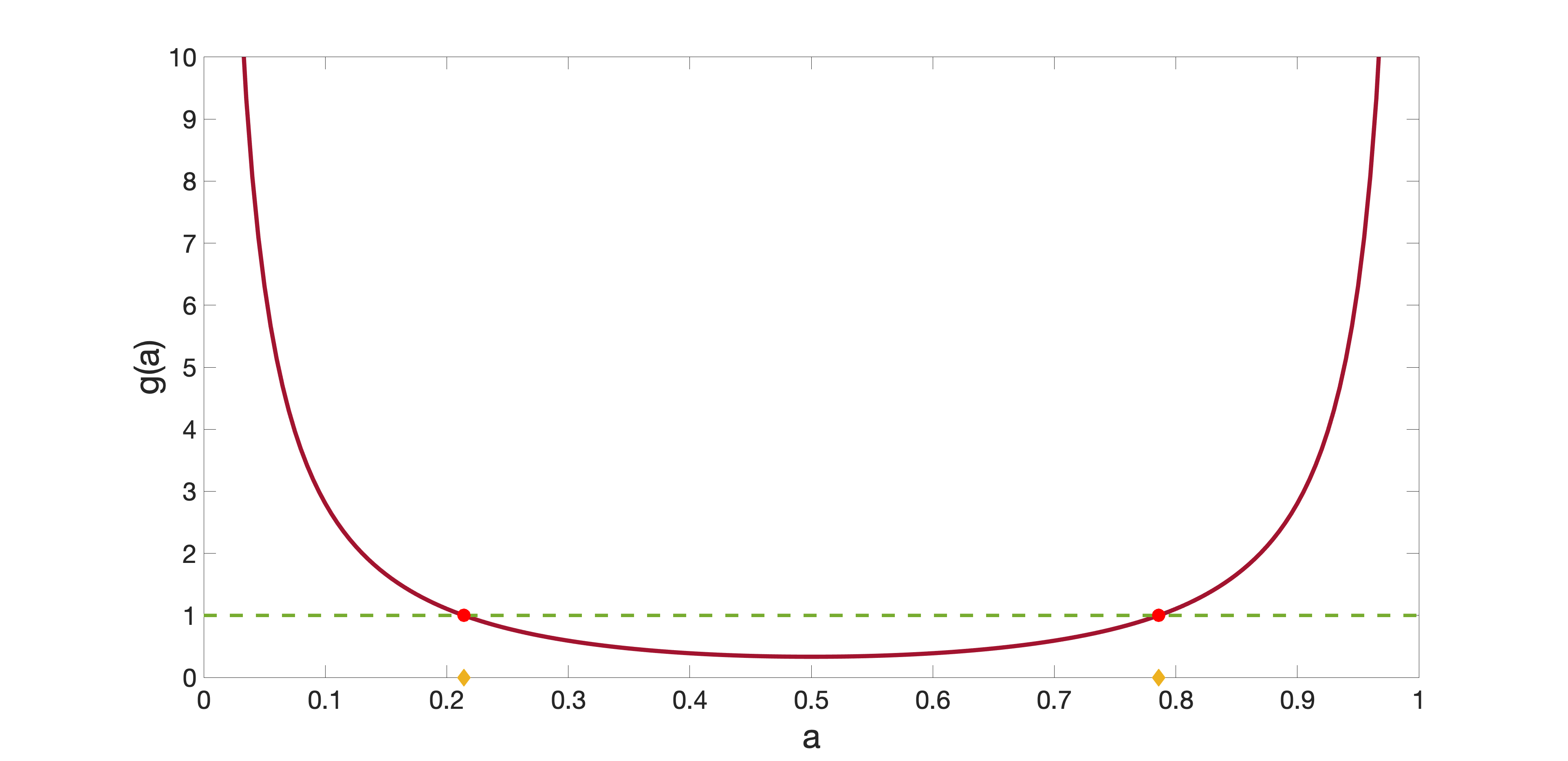} 
  \caption{The function $g(a)$ for which $\lambda(a) = 1\pm g(a)$}
  \label{fig:ga}
\end{figure}
In this section we review a result by
O.~Steinbach~\cite{Steinbach:2002} on the boundedness in $\HS{s}$ of the projector onto the
space of continuous piecewise linears, orthogonally to the space of
piecewise constants on the dual grid, which we adapt to the case at
hand by switching the roles of the two grids.
By a scaling argument it is sufficient to consider the case
$\mbP=1$.
Let $\G = \big\{\E_k,\,k=1,\cdots,M\}$ denote a decomposition of $\bP$
and let $\widehat{\hh}_k =\ABS{\E_k}$.
Let $x_k$ denote the midpoint of the interval $\E_k$ and
$\hG=\big\{\tau_\ell,\,\ell=1,\cdots,\Ms\big\}$, with
$\tau_\ell=\big[\xs_\ell,\xs_{\ell+1}\big]$,
the dual grid, with the ``cyclic'' convention that $\xs_{M+1}=\xs_1$,
$\E_{M+1}=\E_1$. We let $h_\ell = | \tau_\ell |$ denote the length of $\tau_\ell$. 

We make the assumption that $\G$ is locally quasi uniform, that is,
that there exists $\cs'\geq 1$ such that for all $k$ it holds that
\begin{align}
  \frac{1}{\cs'}\leq\frac{\widehat{\hh}_k}{\widehat{\hh}_{k+1}}\leq \cs'.
  \label{eq:graded}
\end{align}
We let $\Lin=\spa\{\phi_k\}_{k=1}^M \subset\HONE(\bP)$ and
$\Const=\spa\{\psi_k\}_{k=1}^M\subset\LTWO(\bP)$ denote, respectively,
the space of continuous piecewise linears on the grid $\hG$ and
the space of piecewise constants on the grid $\G$.
Here, $\phi_k$ is the nodal basis function corresponding to $\xs_k$ and
$\psi_k$ is the characteristic function of the interval $\E_k$.
Observe that the dual grid $\hG$ is itself locally quasi
uniform.
Moreover, the local mesh sizes are comparable, that is, there exist a
positive constant $c$ such that, for $\ell,k$ with
$\tau_\ell\cap\E_k\neq\emptyset$
\begin{align}
  \frac{1}{c}\leq\frac{\widehat{\hh}_k}{\hh_{\ell}} \leq c.
  \label{eq:graded2}
\end{align}

\medskip
We let $\Qtilde:\LTWO(\bP)\to\Lin$ denote the projection oprator defined as the solution to the variational
problem
\begin{align*}
  \int_{\bP} \Qtilde\us\ws \dS = \int_{\bP}\us\ws \dS,\qquad\forall\ws \in\Const.
\end{align*}

With the same proof as in \cite{Steinbach:2002}, we find that the operator
$\Qtilde$ is well defined and bounded in $\LTWO(\bP)$ with a constant
that does not depend on the size and number of the elements but only
on the constant $c'$ in \eqref{eq:graded}.
Let the local Gramian matrix be defined by
\begin{align*}
  \Gl[i,j] = \int_{\tau_\ell}\psi_{\ell+i-1}\phi_{\ell+j-1}\dS, \qquad 1\leq i,j\leq 2.
\end{align*}
Moreover, let $\matD_{\ell}$ and $\matH_{\ell}$ be the diagonal matrices defined by
\begin{align*}
  \restrict{(\matD_{\ell})}{i,i} = \restrict{(\Gl)}{i,i},
  \qquad
  \restrict{(\matH_\ell)}{i,i} = \widehat{\hh}^s_{\ell+i-1},\qquad 1 \leq i \leq 2.
\end{align*}
Then, we find that the results  stated in the following theorem holds.
The proof is the same of the analogous results in~\cite{Steinbach:2002}, though in such a paper the roles of the two grids are switched (the space of piecewise constants is definced on $\hG$ and the space of continuous piecewise linears on $\G$),
and  it hence  is omitted.
\begin{theorem}
  Assume that there exists a positive constant $\alpha_0$ such that
  \begin{equation}
    \xv^T \matH_{\ell}\Gl\matH_{\ell}^{-1}\xv \geq \alpha_0 \xv^T\matD_{\ell}\xv
    \label{eq:condsteinbach}
  \end{equation}
  for all $\xv\in\REAL^2$ and any integer $\ell=1,\cdots,\Ms$.
  Then,  $\Qtilde$ is bounded in $\HS{s}(\bP)$ by a constant not
  depending on the size and number of the elements.
\end{theorem}

\medskip
It is now possible to give an explicit sufficient condition on the constant $c'$ in \eqref{eq:graded}, in
order for \eqref{eq:condsteinbach} to hold for some positive constant
$\alpha_0$.
Indeed, let us focus on an element $\tau_\ell$ whose vertices are the
mid points of two adjacent elements $\E_1$ and $\E_2$ of length
$\widehat{\hh}_1$ and $\widehat{\hh}_2$.
We have $\hh_\ell = (\widehat{\hh}_1 + \widehat{\hh}_2)/2$.
Let us rescale everything in such a way that
\begin{align*}
  \hh_\ell = 1, \qquad \widehat{\hh}_1 = 2a,
  \qquad \widehat{\hh}_2 = 2(1-a),
  \qquad\text{with }a\in (0,1).
\end{align*}
A direct computation yields
\begin{align*}
  \Gl = \frac 1 2 \left(
  \begin{array}{cc}
    (2-a)a & a^2 \\
    (1-a)^2 & 1-a^2
  \end{array}
  \right).
\end{align*}

We can rewrite condition \eqref{eq:condsteinbach} in simmetric form as
\begin{align}
  \yv^T\matM_{\ell}\yv
  \geq \alpha_0 \yv^T\yv,
  \qquad\matM_{\ell}
  = \frac{1}{2}\matD_{\ell}^{-1/2}\Big(
  \matH_{\ell}\Gl\matH_{\ell}^{-1} + \matH^{-1}_\ell(\Gl)^T\matH_\ell\Big)\matD_\ell^{-1/2}.
  \label{eq:condsteinbacsym}
\end{align}

A constant $\alpha_0 > 0$ exists such that \eqref{eq:condsteinbach} holds for
all $\yv\in\REAL^2$ if and only if $\matM_\ell$ is a positive definite
matrix, and $\alpha_0$ is then its lowest eigenvalue.
Considering the case  $s=1/2$, a direct computation yields the following eigenvalues
for $\matM_{\ell}$
\begin{align}\label{eq:eigespression}
  \lambda
  = 1 \pm\frac{1}{2}\frac{\sqrt{a + 1}(3 a^2 - 3 a + 1)}{\sqrt{2 - a}(a - a^3)}
  = 1 \pm g(a)
\end{align}
where $g(a)$ is non negative for $a\in(0,1)$, see Figure \ref{fig:ga}.

The lowest eigenvalue stems then from the minus sign in \eqref{eq:eigespression}, and it is
positive if $g(a)<1$.
We solve such an inequality numerically and obtain that $a$ must
satisfy $a_0 < a < 1 - a_0$ with $a_0 \sim 0.214009576006805$ for
condition \eqref{eq:condsteinbach} to be true.
Now, we translate such condition on $a$ on a condition on the constant
$c'$ appearing in equation \eqref{eq:graded}.
More precisely, condition~\eqref{eq:condsteinbach} is satisfied if the
inequalities in~\eqref{eq:graded} hold with
$c'<(1-a_0)/a_0 \sim 3.672688104237926$.
The optimal value for $\alpha_0$ is attained when $a=1/2$, which
corresponds to $c'=1$, e.g., the uniform grid case.
In such a case, the smallest eigenvalue of $M_\ell$ is $\lambda=2/3$.

\end{document}